\documentclass[12pt]{amsart}

\usepackage{amsmath}
\usepackage{amsfonts}
\usepackage{amssymb}
\usepackage{graphicx}
\usepackage{mathrsfs}
\usepackage{pb-diagram}
\usepackage{epstopdf}
\usepackage{amscd}
\usepackage{color}
\usepackage{verbatim}
\usepackage[all]{xy}
\usepackage{caption}
\usepackage{subcaption}
\usepackage{pstricks}

\newtheorem{theorem}{Theorem}[section]
\newtheorem{lemma}[theorem]{Lemma}
\newtheorem{corollary}[theorem]{Corollary}

\newtheorem{prop}[theorem]{Proposition}
\newtheorem{definition}[theorem]{Definition}

\newtheorem{remark}[theorem]{Remark}

\newtheorem{assumption}[theorem]{Assumption}

\numberwithin{equation}{section}

\newcommand{\conn}{\nabla}

\newcommand{\bP}{\mathbb{P}}
\newcommand{\Z}{\mathbb{Z}}

\newcommand{\R}{\mathbb{R}}
\newcommand{\C}{\mathbb{C}}

\newcommand{\CL}{\mathcal{L}}

\newcommand{\bi}{\mathbf{i}\,}

\newcommand{\uz}{{\underline{z}}}
\newcommand{\up}{\underline{p}}

\newcommand{\CM}{\mathcal{M}}

\newcommand{\Hom}{\mathrm{Hom}}

\newcommand{\Id}{\mathrm{Id}}

\newcommand{\Jac}{\mathrm{Jac}}
\newcommand{\Hol}{\mathrm{Hol}}
\newcommand{\Fuk}{\mathrm{Fuk}}
\newcommand{\MF}{\mathrm{MF}}
\newcommand{\eff}{\mathrm{eff}}
\newcommand{\largewedge}{\mbox{\Large $\wedge$}}

\newcommand{\AI}{A_\infty}
\newcommand{\WT}[1]{\widetilde{#1}}

\newcommand{\first}{{\zeta^F}}
\newcommand{\last}{{\zeta^L}}
\newcommand{\one}{{\mathbf{1}}}

\newcommand{\cL}{\mathcal{L}}
\newcommand{\uL}{{\underline{L}}}
\newcommand{\bL}{{\mathbb{L}}}

\newcommand{\Pal}{\mathrm{Pal}}

\begin{document}

\title{Localized mirror functor constructed from a Lagrangian torus}


\author[Cho]{Cheol-Hyun Cho}
\address{Department of Mathematical Sciences, Research institute of Mathematics\\ Seoul National University\\ Gwanak-ro 1\\ Gwanak-gu \\Seoul 151-747\\ Korea\\
}
\email{chocheol@snu.ac.kr}

\author[Hong]{Hansol Hong}
\address{Center of Mathematical Sciences and Applications \\ Harvard University \\ 20 Garden Street \\ Cambridge MA 02138}
\email{hhong@cmsa.fas.harvard.edu, hansol84@gmail.com}

\author[Lau]{Siu-Cheong Lau}
\address{Department of Mathematics\\ Boston University\\ 111 Cummington Mall \\ Boston MA 02215}
\email{s.lau@math.harvard.edu}


\begin{abstract}
Fixing a weakly unobstructed Lagrangian torus in a symplectic manifold $X$, we define a holomorphic function $W$ known as the Floer potential.  We construct a canonical $\AI$-functor from the Fukaya category of $X$ to the category of matrix factorizations of $W$.  It provides a unified way to construct matrix factorizations from Lagrangian Floer theory. 
The technique is applied to toric Fano manifolds to transform Lagrangian branes to matrix factorizations.  Using the method, we also obtain an explicit expression of the matrix factorization mirror to the real locus of the complex projective space.
\end{abstract}

\maketitle

\section{Introduction}
Homological mirror symmetry conjecture by Kontsevich \cite{kontsevich94} asserts that for a pair of mirror manifolds $(X,\check{X})$, the derived Fukaya category of Lagrangian submanifolds in $X$ is equivalent to the derived category of coherent sheaves on $\check{X}$.  The study of homological mirror symmetry leads to many new insights to Fukaya categories and computational techniques for proving the conjecture in various cases.

More generally when $X$ is not required to be Calabi-Yau, the mirror of $X$ is a Landau-Ginzburg model $W$, which is a holomorphic function rather than a manifold.  Intuitively, the singular locus of $W$ (which is not necessarily smooth nor connected) is the space mirror to $X$.  Homological mirror symmetry can still be stated by using the category of matrix factorizations of $W$ \cite{Ei,B,Or} (in place of the derived category of coherent sheaves), or the Fukaya-Seidel category of $W$ \cite{Se} (in place of the Fukaya category).

In \cite{CHL}, we proposed and constructed a functor to realize homological mirror symmetry using immersed Lagrangian Floer theory.  We used the formal deformations and obstructions coming from the self-intersections of a fixed Lagrangian immersion $\bL$ in $X$ to construct a Floer potential which serves as a Landau-Ginzburg mirror $W$.  Given any Lagrangian $L_1$ in the Fukaya category, the Lagrangian intersection theory between the immersion $\bL$ and $L_1$ is used to construct the mirror object $R$ of $L_1$.  By the result of Orlov \cite{Or}, for a Landau-Ginzburg model $W$, the appropriate objects to consider for singularity theory of $W$ are matrix factorizations, which are endomorphisms $R$ of vector bundles satisfying $R^2 = (W - c) \cdot \Id$ for some constant $c$.   In \cite{CHL} the theory was applied to the orbifold spheres $\bP^1_{(a,b,c)}$, and an inductive method was found in \cite{CHKL} to deduce an explicit expression of the Landau-Ginzburg mirror $W$ (which contains infinitely many terms).

In this paper, we fix the reference to be a smooth Lagrangian torus $\bL$ rather than an immersed Lagrangian.\footnote{The assumption that $\bL$ is a torus is actually not essential.  We concentrate on torus because it plays a central role in SYZ mirror symmetry.}  For simplicity we make the technical assumption that $(X,\bL)$ is positive (see Assumption \ref{assum1}), although this is not necessary if one uses the full machinery of Lagrangian Floer theory.
Similar to \cite{CHL}, we use Lagrangian intersection theory to construct a Landau-Ginzburg model $W=W ({\bL})$ and an $\AI$-functor from the Fukaya category to the category of matrix factorizations of $W$.  Flat $\C^\times$ connections plays a key role in this setup, since they serve as the formal (complexified) deformations of $\bL$.  The essential issue coming from considering flat $\C^\times$ connections, rather than self-intersections of an immersion, is the choice of gauge.  Different choices of gauge for the same connection result in different expressions of the functor, and we need to make a consistent choice to make sure the functor is well-defined, and study the effect of gauge change.

The fundamental idea of constructing a mirror functor for a Lagrangian torus fibration (with mild singularities)
goes back to the work of Fukaya \cite{Fu2}, \cite{Fu3}, who introduced family Floer homology of Lagrangian torus fibers under certain assumptions.  More recently, Abouzaid \cite{Ab,Ab-faith} studied the family Floer theory for Lagrangian torus fibration without singular fibers.  The mirror functor constructed in this paper is a local piece of the family Floer functor near a Lagrangian torus fiber $\bL$, which has the advantage that it can be explicitly computed.  In the absence of singular fibers such as in the case of toric manifolds, these local pieces can be glued together to give the global functor.

We summarize the construction as follows.  First fix a weakly unobstructed 
smooth Lagrangian torus $\bL$ in a symplectic manifold $X$.  We define a holomorphic function $W$ 
on the space $(\C^\times)^n$ of flat $\C^\times$-connections $\nabla$ by using the $m_0$-term of the $\AI$ algebra $\mathrm{CF}^*((\bL,\nabla),(\bL,\nabla))$. 
Geometrically $W$ is obtained from counting holomorphic discs of Maslov index two bounded by $(\bL,\nabla)$ (see Definition \ref{def:W}).  In general $W$ should serve as a part of a global Landau-Ginzburg mirror to $X$.  
This method was used by the joint work \cite{CO} of the first author with Oh, and Fukaya-Oh-Ohta-Ono \cite{FOOOT} to construct the mirrors of toric manifolds.

Now comes the main construction of this paper.  To transform a (weakly unobstructed) Lagrangian $L_1$ to a matrix factorization, we take the Lagrangian Floer `complex' between $(\bL,\nabla)$ and $L_1$.  The differential does not square to zero; indeed it follows from the $A_\infty$ relations that the differential squares to $W - \lambda$, where $\lambda$ is given by $m_0^{L_1} = \lambda \cdot \one_{L_1}$, and thereby the Lagrangian Floer `complex' is indeed a matrix factorization of $W$.  The strategy of constructing matrix factorizations using Lagrangian Floer theory was found by Oh \cite{Oh1} and \cite{FOOO}.

Note that the same flat connection $\nabla$ admits different choices of gauge.  The resulting matrix factorizations depend on such a choice.   Moreover terms like $z^a$ for $a \not\in \Z$ could appear if the gauge is chosen arbitrarily.  To make sure the resulting matrix factorizations are still defined over the Laurent series ring, we make the following gauge choice for the flat connections over $\bL$.  Namely, we always require that the flat connections are trivial away from small neighborhoods of certain fixed codimension-one tori (called hyper-tori).  
Then holonomy of a flat connection over a path in $\bL$ can be expressed in terms of the number of intersections of the path with the hyper-tori, which are integer-valued.  It ensures that we still stay inside the Laurent series ring.  Moreover, we can show that the resulting matrix factorization does not depend on the choice of hyper-tori, nor a representative in the Hamiltonian isotopy class of $L_1$ (see Section \ref{sec:inv}).  As a result, we have the following.

\begin{theorem} \label{thm:intro-F}
Fix a smooth Lagrangian submanifold $\bL \subset X$.  There exists a Floer potential $W$ defined over the space of flat $\C^\times$ connections over $\bL$, and an $A_\infty$ functor from $\Fuk_\lambda (X)$ to $\MF(W-\lambda)$ for each $\lambda \in \Lambda$, where $\Fuk_\lambda (X)$ is the Fukaya category of weakly unobstructed Lagrangian submanifolds $L$ with $m_0^L = \lambda \one_L$, and $\Lambda$ denotes the Novikov field.
\end{theorem}


The mirror functor is computable by using {\em pearl complex} introduced by Biran-Cornea \cite{BC} (decorated with flat complex line bundles for the purpose of this paper), which is explained in Section \ref{sec:pearl}.  

In Section \ref{sec:toricFano}, we apply our construction to toric Fano manifolds, and transform Lagrangian torus fibers (decorated by flat connections) to matrix factorizations of the mirror.  \footnote{We expect that the method in this paper works for general compact toric manifolds.  To avoid technical issues in Lagrangian Floer theory, we restrict to the Fano case.}  However, the mirror matrix factorizations are hard to be fully computed.  

On the other hand, we find that the leading-order terms of such a matrix factorization $R$ always form another matrix factorization $R_0$.  We deduce an explicit closed formula for $R_0$ and show that it is of {\em wedge-contraction type}, and hence it is a generator of the category of matrix factorizations by the result of Dyckerhoff \cite{Dyc}.  By spectral-sequence argument of Polishchuk-Vaintrob \cite{PV}, we deduce that terms in $R-R_0$ do not contribute to cokernel, and hence $R$ is also a generator of the category.

We summarize the main result as follows.  We switch from the Novikov field $\Lambda$ to the complex field $\C$ by substituting the Novikov formal variable $T$ by $e^{-1}$, in order to match with Hori-Vafa mirrors of toric Fano manifolds.  It is valid since both $W$ and the matrix factorization $R$ have finitely many terms in the Fano case.  For non-Fano cases one should stick with the Novikov field $\Lambda$.

\begin{theorem} \label{thm:main}
Let $X$ be a toric Fano $n$-fold whose moment map polytope is given by
$$\{u: \langle v_i,u \rangle - \lambda_i \geq 0 \textrm{ for all } i=1,\ldots,m\}$$
where $m$ is the number of primitive generators $v_i$ of the corresponding fan and $\lambda_i \in \R$ are some fixed constants.  Without loss of generality we assume $v_1,\ldots,v_n$ form an integral basis.  Let
$$W = \sum_{i=1}^m e^{\lambda_i} t^{v_i} \in \C[t_1^{\pm 1}, \ldots, t_n^{\pm 1}]$$ 
be the Hori-Vafa mirror.  \footnote{Indeed the K\"ahler structure needs to be complexified to match the moduli spaces.  Simply put, $\lambda_i$ should be replaced by $\lambda_i + \sqrt{-1} \theta_i$ for some fixed $\theta_i \in \R/2\pi \Z$.}  
Write $t_i = z_i \cdot e^{-\langle u,v_i\rangle}$ for $i=1,\ldots,n$.   
Let $L_u$ be a moment-map fiber decorated by a flat $\C^\times$ connection $\nabla_{\uz}$ (where $\uz = (\uz_1,\ldots,\uz_n)$), and $R(\uz)$ the matrix factorization mirror to $(L_u,\nabla_{\uz})$ under the functor in Theorem \ref{thm:intro-F}.  
\begin{enumerate}
\item $R(\uz)$ takes the form $\left(\largewedge^* \underline{\C^n}, d \right)$ where $\underline{\C^n}$ denotes the module $(\C[t_1^{\pm 1}, \ldots, t_n^{\pm 1}])^{\oplus n}$, and
$d = \sum_{k=0}^n d_{-(2\left\lfloor (k+1)/2 \right\rfloor-1)} $
with $d^2 = W-W(\uz)$.
\item $R_0(\uz) := \left(\largewedge^* \underline{\C^n}, d_1 + d_{-1} \right)$ itself is a matrix factorization of $W-W(\uz)$.
\item $d_1 + d_{-1}$ takes the form
$$\left(\sum_{i=1}^n (z_i - \uz_i) e_i \wedge\right) + \left(\sum_{i=1}^n c_i \iota_{e_i}\right) + \left( \sum_{i=n+1}^m c_i \sum_{j=1}^n \alpha^i_j(z,\uz) \iota_{e_j} \right)$$
where $c_i =e^{-(\langle u, v_i \rangle - \lambda_i)}$ and the explicit formula for $\alpha^i_j(z,\uz)$ is given in Theorem \ref{thm:explicit}.
\item Both $R(\uz)$ and $R_0(\uz)$ are split generators of $D\MF(W-W(\uz))$ (which is non-trivial only when $(\underline{t}_i = \uz_i \cdot e^{-\langle u,v_i\rangle})_{i=1}^n$ is a critical point of $W$).
\end{enumerate}
\end{theorem}

We construct an explicit isomorphism between $R$ and $R_0$ for $\dim X \leq 4$.



The above theorem generalizes the results of Chan-Leung \cite{CL} for $\bP^2$ and Tu \cite{Tu} for toric Fano surfaces.  Given a Lagrangian torus fibration, the work of Tu constructed a functor (away from singular fibers) based on Fourier-Mukai transform from the Fukaya category of smooth torus fibers to the category of sheaves of modules (over a certain $A_\infty$-algebra), whose objects can be interpreted as matrix factorizations.  Moreover Abouzaid-Fukaya-Oh-Ohta-Ono announced a proof of homological mirror symmetry conjecture for toric manifolds by showing that the toric fibers generate.  In general the functor is difficult to write down, and this paper provides a method to compute it by localizing to each torus fiber.

As another application, we transform the real locus $\R\bP^n$ in $\C\bP^n$ to a matrix factorization of the mirror, by using the result of \cite{AlAm}.  We only consider $n$ being odd so that $\R\bP^n$ is orientable.  The result is the following.

\begin{theorem}[Theorem \ref{thm:RP}]
Let $W = T^k \left(z_1 + \ldots + z_n + \frac{1}{z_1 \ldots z_n}\right)$ be the Landau-Ginzburg mirror of $\bP^n$ where $n$ is odd.  (The base point in the moment polytope is chosen suitably so that $W$ takes this form.  Moreover the K\"ahler form is taken such that $k \in \Z_{>0}$.  $T$ denotes the Novikov variable.)
Denote the matrix factorization of $W$ mirror to $\R\bP^n \subset \C\bP^n$ by $(E,d)$. 

Then $E$ is given by the trivial bundle with a basis labelled by $[\pm 1 :  \cdots : \pm 1] \in (\Z_2)^{n+1} / \Z_2$, where the quotient is given by the diagonal action of $\Z_2$ on $(\Z_2)^{n+1}$.  The differential $d$ is determined by $d \, p = \sum_q m_{qp} \, q$ where $p,q \in (\Z_2)^{n+1} / \Z_2$, and $m_{qp}$ are given as follows.
\begin{enumerate}
\item When 
$$p=[a_0 : \cdots : a_{i-1} : -1 : a_{i+1} : \cdots : a_n] \textrm{ and } q=[a_0 : \cdots : a_{i-1} : 1 : a_{i+1} : \cdots : a_n]$$ 
where the number of $a_j = -1$ (for $j \not= i$) is even,
\begin{align*}
m_{qp} &= \dfrac{T^{k/2}}{\prod_{1 \leq j \leq n} z_j^{\delta (a_j, -1)}} \textrm{ and }
m_{pq} = \dfrac{T^{k/2}}{\prod_{1 \leq j \leq n} z_j^{\delta (a_j, 1)}}  \textrm{ if } i=0,\\
m_{qp} &=T^{k/2} z_i \textrm{ and } m_{pq} =T^{k/2}   \textrm{ if } i \neq 0.
\end{align*}
$\delta (a,b) = 1$ when $a=b$ and zero otherwise.
\item $m_{qp} = m_{pq} = 0$ otherwise.
\end{enumerate}
\end{theorem}

\section*{Acknowledgement}
The authors are grateful to Jonathan David Evans and Yanki Lekili for informing them about their very useful results on generation of Fukaya categories of Hamiltonian G-manifolds, and in particular toric Fano manifolds.  Combining with the results in this paper, it gives a proof of homological mirror symmetry for toric Fano manifolds.  C.H. Cho and S.-C. Lau thank The Chinese University of Hong Kong for its hospitality, where part of the work was carried out. 
C.H. Cho thanks Yong-Geun Oh for helpful discussions. S.-C. Lau expresses his gratitude to Kwokwai Chan and Junwu Tu for useful explanations of their works.
The work of S.-C. Lau was supported by Harvard University and Boston University.

\section{Localized Lagrangian Floer potential and Lagrangian Floer complex}\label{sec:2}
Let $X$ be a symplectic manifold.  To avoid technical issues of transversality, we make the assumption that $X$ and the Lagrangian submanifolds under consideration are \emph{positive}.  The assumption is not necessary if one is willing to handle the issue of transversality by using more advanced machinery of Lagrangian Floer theory.

\begin{assumption}\label{assum1}
Let $X$ be a symplectic manifold and $\bL$ be a Lagrangian submanifold.  The pair $(X,\bL)$ is said to be {\em positive} if there exists an almost complex structure $J$ such that
\begin{enumerate}
\item any non-constant $J$-holomorphic sphere in $X$ has a positive Chern number;
\item any non-constant $J$-holomorphic disc with boundary on $\bL$ has a positive Maslov index;
\item $J$-holomorphic discs of Maslov index two with boundary on $\bL$ are Fredholm regular.
\end{enumerate}
\end{assumption}
The assumption holds for monotone Lagrangian submanifolds \footnote{A compact Lagrangian submanifold $\bL \subset M$ is called monotone if $I_\omega = \lambda I_\mu$ for some $\lambda \geq 0$, where $I_\omega, I_\mu: \pi_2(M,\bL) \to \R, \Z$ are symplectic area and Maslov index homomorphism respectively.}.  It also holds for Lagrangian torus fibers of a toric Fano manifold.  

Now consider a Lagrangian torus $\bL$ satisfying Assumption \ref{assum1}, and we shall fix the almost complex structure $J$ satisfying the assumption.  We define a Lagrangian Floer potential in this section, using formal deformations brought by flat $\C^\times$ connections on $\bL$.  

Fix a basis $\{E_i\}_{i=1}^n$ of $H_1(\bL,\Z)$, and its dual basis $\{E_i^*\}_{i=1}^n$ of $H^1(\bL,\Z)
\subset H^1(\bL,\C)$.
Consider
$$ b:= x_i E_i^* \in H^1(\bL,\C)/H^1(\bL,\Z)$$
which is interpreted as a flat $\C^\times$ connection on $\bL$ via the associated representation 
\begin{equation}\label{eq:defrho}
\rho^b: \pi_1(\bL) \to \C^\times;  \gamma \mapsto \exp2 \pi \sqrt{-1} (b, \gamma).
\end{equation}
Define the complex coordinates 
\begin{equation}\label{eq:mirvar1}
z_i := \rho^b(E_i) 
\end{equation}
of the moduli space of flat (possibly non-unitary) complex line bundles.  We denote by $\cL_z$ the flat line bundle with the flat connection  $\nabla_z$ parametrized by $z=(z_1,\ldots,z_n)$.

Note that there is a subtle difference between the definition of the variables $z_i$ here and that in the conventional Strominger-Yau-Zaslow (SYZ) approach (see for instance \cite{auroux07} or \cite{FOOO}).
In the SYZ setting, the mirror variables are defined using Lagrangian torus fibration and they parametrize locations and flat $U(1)$ connections of a Lagrangian torus fiber.  In our setting here,  $\bL$ is fixed and the mirror variables parametrize flat $\C^\times$ (instead of $U(1)$) connections.  In Section \ref{sec:toricFano} for toric manifolds we will relate the two by a change of variables.

Since there are infinitely many holomorphic discs in general, the Floer potential is defined over the Novikov field
$$\Lambda = \left\{ \sum_i a_i T^{\lambda_i} \mid a_i \in \C, \lambda_i \in \R,  \lim_{i \to \infty} \lambda_i = \infty \right\}.$$
Here $T$ is a formal parameter, and $\Lambda$ has a natural energy filtration considering only elements with $\lambda_i \geq \lambda_0$ for some $\lambda_0$.
We also set $\Lambda_0 = \{ \sum_i a_i T^{\lambda_i} \in \Lambda \mid \lambda_i \geq 0\}$ which is known as the Novikov ring.
We will sometimes use the notation $\Lambda_\C$ to emphasize that $a_i \in \C$, and we can define $\Lambda_\Z$ in
a similar way.

Let $\CM_1(\bL,J,\beta)$ be the moduli space of stable $J$-holomorphic discs in a homotopy class $\beta \in \pi_2(X,\bL)$.
By using Assumption \ref{assum1}, to define the Floer potential it is enough to consider those $\beta$ with Maslov index two.
For such a $\beta$ (which is of the minimal Maslov index) the moduli space of holomorphic discs is itself compact and does not require compactification by stable discs. Moreover $\dim \CM_1(\bL,J,\beta) = n$. Hence the image of the evaluation map $ev_{\beta}: \CM_1(\bL,J,\beta) \to \bL$ induced on homology is a constant multiple of the fundamental class of $\bL$, and we denote this multiple by $n_\beta(\bL)$.

\begin{definition}\label{def:W}
The Floer potential of $(\bL,\nabla_z)$ is defined as
$$W(\bL)(z_1,\cdots,z_n) := \sum_{\beta, \mu(\beta)=2} n_\beta(\bL) T^{ \omega(\beta)} \rho^b(\partial \beta)$$
where each $\rho^b(\partial \beta)$ can be expressed as a monomial in $z_i^{\pm 1}$ by Equation \eqref{eq:mirvar1}.
\end{definition}
In general the above expression could be an infinite series.  Thus $W$ is a Laurent series in $z_i$'s whose coefficients are Novikov elements.  In case when the sum is finite, we can simply put $T$ to be the constant $e^{-1}$, and so $W$ is a Laurent polynomial with complex coefficients.  For simplicity of notations we will assume that the sum is finite from now on, and write $W \in \Lambda[z_1^{\pm 1}, \ldots, z_n^{\pm 1}]$.  
\begin{remark}\label{rem:lam}
In general there are infinitely many $\beta$ contributing to $W$.  The potential $W$ belongs to $\Lambda \ll z_1^{\pm 1}, \ldots, z_n^{\pm 1} \gg$, which is the completion of the Laurent polynomial ring with respect to the energy filtration of $\Lambda$. 
\end{remark}


For the purpose of the next section, we now recall the Lagrangian Floer complex between two positive Lagrangian submanifolds.
Let $L_0$ and $L_1$ be two oriented, spin, positive Lagrangian submanifolds $L_0$ and $L_1$ in a  symplectic manifold $(X^{2n}, \omega)$. Lagrangian Floer homology $HF(L_0,L_1)$ in this setting was first defined by Oh \cite{Oh1} in the monotone cases and was later generalized by Fukaya-Oh-Ohta-Ono \cite{FOOO}  (see also Biran-Cornea \cite{BC} for the notion of pearl complex when $L_0=L_1$).  The definition below assumes that $L_0$ and $L_1$ intersect transversely.

The Floer complex $CF^*(L_0,L_1)$ is a free $\Z/2$-graded $\Lambda$-module generated by the intersection points $p \in L_0\cap L_1$.  
The Floer differential $\delta$ is defined by counting $J$-holomorphic strips: for $p \in L_0\cap L_1$, we have
\begin{equation}\label{eq:lft1}
\delta( \langle p \rangle ) = \sum_{q \in L_0 \cap L_1} n(p,q) \langle q \rangle,
\end{equation}
where $n(p,q)$ is the signed number of isolated $J$-holomorphic strips $u: \R \times [0,1] \to M$ modulo time translation
weighted by the symplectic area $T^{\omega(u)}$ 
(we refer readers to \cite{FOOO} for the details on signs):
$$ u(\R \times \{1\}) \subset L_0, \quad u(\R \times \{0\}) \subset L_1, \quad  \overline{\partial}_J u =0.$$
Here, $J=\{J_t\}$ is a time-dependent generic compatible almost complex structure, where $J_0$ and $J_1$ satisfy the positivity assumption of $L_0$ and $L_1$ respectively. To define a signed counting,  we need to fix an orientation of the orientation spaces associated to intersection points $p, q$.  We refer the readers to \cite{FOOO} for the detail.

As before, we denote by $n_{\beta}(L_i)$ the number of Maslov index two $J_i$-holomorphic discs with
a boundary on $L_i$ passing through a point $p \in  L_i$. We set $\Phi(L_i) =\sum_{\beta,\mu(\beta)=2} n_{\beta}(L_i) T^{\omega(\beta)}$.
Standard Floer theory argument (Gromov-compactness and gluing theorem)  produces the identity
\begin{equation}\label{eq1:floer}
\delta^2(x) = \big( \Phi(L_1)  - \Phi(L_0) \big)x  \;\; \textrm{for any} \;\; x \in 
CF^*(L_0,L_1).
\end{equation}
If $\Phi(L_1)=\Phi(L_0)$, then $\delta^2=0$ and hence Floer cohomology can be defined.

As in \cite{kontsevich94} (or \cite{C}), we consider a slight generalization by introducing flat (possibly non-unitary) complex line bundles $\mathcal{L}_0 \to L_0$ and $\mathcal{L}_1 \to L_1$.
For $i=0,1$, consider a representation $\rho_i: \pi_1(L_i) \to \C \setminus \{0\}$, and
we take a flat connection $\nabla_i$ for $\CL_i$ whose holonomy representation is given by $\rho_i$.
Note that the complex given below depends on the choice of gauge of $\nabla_i$. 

The Floer complex  $CF^*((L_0,\mathcal{L}_0),(L_1,\mathcal{L}_1))$ is a free $\Z/2$-graded $\Lambda$-module generated by the intersection points $L_0 \cap L_1$. For each $p \in L_0 \cap L_1$, we consider the vector space $\Hom ((\CL_0)_p, (\CL_1)_p)$, which is identified with $\C$ by fixing the isomorphisms $(\CL_0)_p \cong \C, (\CL_1)_p \cong \C$.  Then 
we tensor with the Novikov field $\Lambda$ to obtain the Floer complex.

The differential $\delta^{\mathcal{L}_0, \mathcal{L}_1}$ also takes account of holonomies of the flat connections $\CL_0, \CL_1$.  Given a $J$-holomorphic strip $u$ from $p$ to $q$, by taking the boundary we obtain a path $\partial_i u$ from $p$ to $q$ in $L_i$.
Parallel transport along $\partial_i u$ gives $ \Pal_{\partial_i u}: (\CL_i)_p \to (\CL_i)_q.$
Using the identifications $(\CL_i)_p \cong \C$ fixed before, $\Pal_{\partial_i u}$ is identified with multiplication by a nonzero complex number.
Each strip $u$ from $p$ to $q$ contributes to the differential as 
\begin{equation} \label{eq:diff}
(-1)^{a(u)} \Pal_{\partial_0 u}^{-1} \cdot \Pal_{\partial_1 u} \cdot T^{\omega (u)},
\end{equation}
where $(-1)^{a(u)}$ is the sign of $u$ from \cite{FOOO},
and summing all $u$ defines $n(p,q)$ in the definition of $\delta(\langle p \rangle)$ \eqref{eq:lft1}.

We also define $\Phi(\mathcal{L}_i) = \sum_{\beta,\mu(\beta)=2} n_{\beta}(L_i) \rho_i(\partial \beta) T^{\omega (\beta)}$, which is a Novikov constant.
Then Floer equation \eqref{eq1:floer} becomes 
\begin{equation}\label{eq2:floer}
(\delta^{\mathcal{L}_0, \mathcal{L}_1} )^2 x =  \big( \Phi(\CL_1)  - \Phi(\CL_0) \big)x,  \;\; \textrm{for any} \;\; x \in 
CF((L_0,\mathcal{L}_0),(L_1,\mathcal{L}_1)).
\end{equation}

Later on we shall put $\bL$ (with a flat line bundle $\cL_z$ whose holonomy varies in $z$) in place of $L_0$.  As we have emphasized, the above definition of differential $\delta^{\mathcal{L}_0, \mathcal{L}_1}$ depends on the choice of gauges of $\mathcal{L}_0$ and $\mathcal{L}_1$.  On the other hand $z_i$'s are parametrizing the isomorphism classes of flat connections over $\bL$.  In order to identify the Floer complex as a matrix factorization over $\C[z_1^{\pm 1},\ldots,z_n^{\pm 1}]$, we need to make a specific choice of gauge in each isomorphism class.  In the next section we will introduce the notion of gauge hypertori in order to fix this choice.

\section{Mirror matrix factorization via Floer complex}\label{sec:3}
\subsection{Gauge hypertori}
Fix a Lagrangian torus $\bL$ equipped with a flat line bundle $\cL_z$ and denote its Floer potential by $W(\bL)$.  We want to transform a positive Lagrangian submanifold $L_1 \subset X$ together with a flat connection $\CL_1$ to a matrix factorization $(P^0, P^1,d)$ of $W(\bL)$. This matrix factorization is given by a Floer complex.  $P^0$ (resp. $P^1$) are defined as free $\Lambda \ll z_1^{\pm1}, \cdots, z_n^{\pm 1} \gg$-modules  generated by even (resp. odd) intersection points of $\bL \cap L_1$, and 
\begin{equation}\label{eq:ddegdelta}
d(\cdot) := (-1)^{{\rm deg} (\cdot)} \delta^{\cL_z, \mathcal{L}_1}(\cdot).
\end{equation}
Equation \eqref{eq2:floer} gives
\begin{equation}\label{eqmf1}
d^2  = W(\bL)-\lambda
\end{equation}
where $\lambda \in \Lambda$ is the potential value $\Phi(\CL_1)$ of the Lagrangian $L_1$.

However the differential $d$ defined using Equation \eqref{eq:diff} is not a Laurent polynomial of $z$ for general choice of gauge of the flat connection on $\bL$.  In what follows we fix a uniform choice of gauge for each isomorphism class of flat $\C^\times$ connections on $\bL$, which is parametrized by the mirror variables $z_i$'s, such that $\delta^{\cL_z, \mathcal{L}_1}$ is expressed in Laurent polynomials of $z$.


For each basic vector $E_i^* \in H^1(\bL,\Z)$, we fix an oriented hyper-torus $H_i \subset \bL$ (which means a codimension $1$ submanifold diffeomorphic to $(S^1)^{n-1}$) whose class is Poincar\'{e} dual to $E_i^*$.

\begin{definition}\label{def:hi}
We fix an identification $\bL \cong (\R/ \Z)^n$, and define $H_i$ to be 
$$(\R/\Z)^{i-1} \times \{0\} \times (\R/\Z)^{n-i}$$
for $i=1\cdots, n$.  Fix $p=(p_1,\cdots,p_n) \in (\R/ \Z)^n$.  $p + H_i=p_i+H_i$ for each $i$ is called  a \emph{gauge hyper-torus}.
We orient $H_i$ so that the intersection $E_i \cap H_i$ is positive.
\end{definition}
\begin{remark}
The gauge hypertorus plays a similar role of bounding cochain in the work of Fukaya-Oh-Ohta-Ono \cite{FOOO}.  A standard bounding cochain $b$ have strictly positive exponents in the Novikov variable $T$ in order to have $e^b$ well-defined.
A constant term (corresponding to $T^0$) can be added formally to $b$ using flat non-unitary line bundles (\cite{FOOOT}, \cite{C}). The above gauge
hypertori will be used to define a flat connection below which corresponds to a constant term of $b$.
\end{remark}

We define a flat connection $\nabla$ with a prescribed holonomy $\rho$ which is trivial away from tubular neighborhoods of the gauge hypertori $H_i + p$.  When crossing $H_i + p$ in the positive transverse orientation (i.e. along the direction of $E_i$), the flat connection acts on the fiber of the line bundle by multiplication of $\rho(E_i)$.

\begin{lemma}\label{lem:flat conn}
Given gauge hypertori $\{p + H_i\}_{i=1}^n$ and $\rho:\pi_1(\bL) \to \C \setminus \{0\}$,
there exists a flat connection $\nabla$ for the trivial complex line bundle $\CL_0$ over $\bL$, such that
 $\rho$ is trivial outside any given small neighborhood of the gauge hypertori, and has the associated holonomy representation $\rho$.
\end{lemma}
\begin{proof}
For simplicity, we consider the case of  the trivial bundle  $\R \times \C$ over $\R$
and we leave the case of $\bL$ as an exercise.
Let $p_i \in \R$, and we define its connection $\nabla$ to be trivial ($\nabla =d$) outside $(p_i  - \epsilon, p_i + \epsilon)$,
and in the interval $(p_i  - \epsilon, p_i + \epsilon)$, $\nabla$ is defined to be 
$$d - 2\pi\bi b_i \delta(t- p_i) dt$$
where $\delta (t)$ is a smooth function whose support is in the interval $( - \epsilon, + \epsilon)$ with
$\int_{-\epsilon}^{+\epsilon} \delta (t) dt = 1.$
$\nabla$ defines the desired flat connection over $\R$.
\end{proof}

The parallel transport of the above chosen connection along a path $\gamma$ in $\bL$ can be expressed in terms of intersection number between $H_i + p$ and $\gamma$.
Namely, for a path $\gamma$ whose endpoints are away from a small neighborhood of $H_i + p$ and transversal to $H_i + p$, the parallel transport  along $\gamma$ is given by 
$$\prod_{i=1}^n \rho (E_i)^{k_i},$$
when $k_i$ is the signed intersection number between $\gamma$ and $H_i + p$.


Now, we express the Floer differential $\delta^{\mathcal{L}_z, \mathcal{L}_1}$ in terms of $z$.  Let $\rho^b: \pi_1(\bL) \to \C \setminus \{0\}$ (see \eqref{eq:defrho}) be the holonomy presentation with $\rho^b(E_i) = z_i$.  Then we fix the flat connection $\nabla_z$ as in Lemma \ref{lem:flat conn}.  Denote by $U_H \subset \bL$ a small neighborhood of $\cup_i (p_i+ H_i)$ so that the parallel transport for $\nabla_z$ is trivial outside $U_H$.  Given another positive Lagrangian submanifold $L_1$ which intersects $\bL$ transversely, we may assume that each point in $\bL \cap L_1$ does not lie in $U_H$ by shrinking the open neighborhood and changing the base point $p$ of the gauge hypertori.

Given two intersection points $p,q \in \bL \cap L_1$,
consider  a $J$-holomorphic strip $u: \R \times [0,1] \to M$ contributing to the differential 
$ \delta^{\mathcal{L}_z, \mathcal{L}_1} $.
Consider the boundary path $\partial_i u$ from $p$ to $q$ in $\bL$ for $i=0$ and $L_1$ for $i=1$.
Recall that the contribution of $u$ to the differential $\delta^{\cL_z, \mathcal{L}_1}$
was given by  $(-1)^{a(u)} \Pal_{\partial_0 u}^{-1} \Pal_{\partial_1 u} \, T^{\omega(u)}$. Here, $\Pal_{\partial_1 u}$ is a complex number.

We claim that  the holonomy factor $\Pal_{\partial_0 u}^{-1}$ along the path $\partial_0u$ is indeed  a monomial in $\Lambda [z_1^\pm,\cdots, z_n^\pm]$. Since $\partial_0u$ is a path starting and
ending away from $U_H$, the  holonomy factor $\Pal_{\partial_0 u}^{-1}$ is given by
$$\prod_{i=1}^n z_i^{-k_i},$$
where $k_i$ is the signed intersection number between $\partial_0 u$ and the hyper-torus $p_i+H_i$.
Hence,  $ \delta^{\cL_z, \mathcal{L}_1}$ from $p$ to $q$ gives an element of $\Lambda [z_1^\pm,\cdots, z_n^\pm]$
(or $\Lambda \ll z_1^{\pm 1}, \ldots, z_n^{\pm 1} \gg$ for an infinite sum)
taking a  sum of all such contributions from isolated $J$-holomorphic strips from $p$ to $q$.

  Intuitively, we are recording how many times the $\bL$-edge of a $J$-holomorphic strip
crosses the gauge hypertori in $\bL$ in terms of $(z_1,\cdots, z_n)$ variables  (see Figure \ref{fig:fourier}).

\begin{figure}[htb!]
\begin{center}
\includegraphics[height=2in]{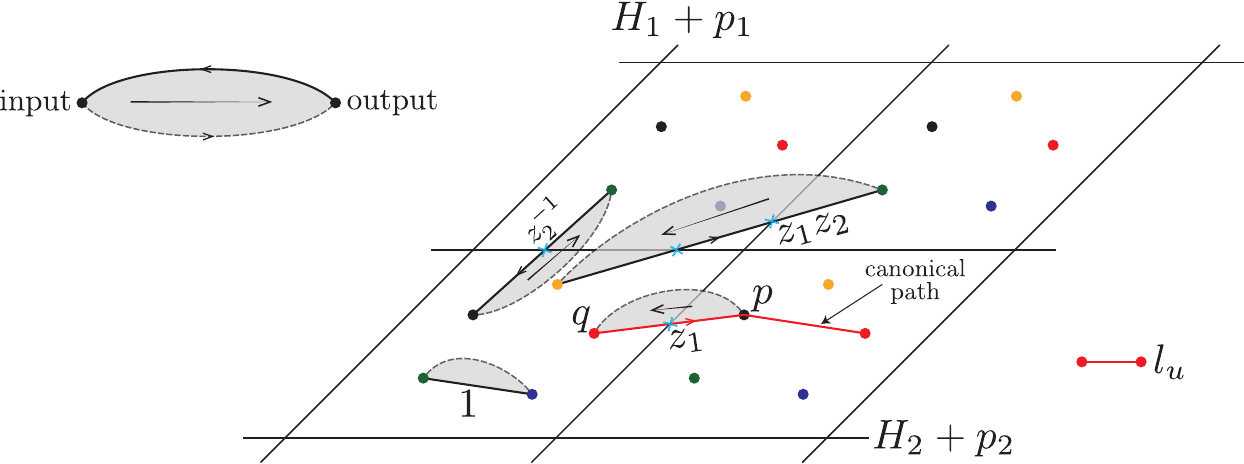}
\caption{$J$-holomorphic strips and corresponding monomials.  The diagram shows the universal cover of the torus $\bL$.  The dots represent intersection points of $\bL$ and $L_1$.}\label{fig:fourier}
\end{center}
\end{figure}
\begin{remark}
Here is another interpretation of the above construction.
Note that by removing the gauge hypertori from $\bL$, we get a simply connected region $U$. Hence
there is a unique path (up to homotopy) for any pair of intersection points of $\bL$ and $L_1$ contained in this simply connected region. Then the path $\partial_0u$ from $p$ to $q$ coming from
a $J$-holomorphic strip $u$ can be concatenated with this unique path (up to homotopy) from $q$ to $p$ to obtain a loop $l_u$ (drawn as the red line in Figure \ref{fig:fourier}), starting and ending at $p$ . In this way, our construction can be regarded as a Fourier transform, in which the counting of $J$-holomorphic strips from $p$ to $q$ whose boundaries correspond to loops (up to homotopy) in $\pi_1(\bL) \cong \Z^n$ transforms into a (Laurent) polynomial in $\Lambda [z_1^\pm,\cdots, z_n^\pm]$.
\end{remark}

Let $P^0$ (resp. $P^1$) be the free $\Lambda [z_1^\pm,\cdots, z_n^\pm]$-module generated by
even (resp. odd) intersection points of $\bL \cap L_1$. 
Here, we are using the canonical $\Z/2$-grading of the Lagrangian Floer complex. The degree of $p \in CF(\bL , L_1)$ is defined using a loop in the Lagrangian Grassmannian of $T_pX$ constructed in the following way.
We choose any path from the oriented Lagrangian subspace $T_pL_1$ to the oriented Lagrangian subspace $T_p\bL$ in the oriented Lagrangian Grassmannian of $T_pM$.  We compose this path
with a canonical path (see \cite[Section 3]{Alston}) from $T_p\bL$ to $T_pL_1$ (without considering orientation).  Hence we obtain a loop in the Lagrangian Grassmannian of $T_pM$, starting and ending at $T_pL_1$.  The winding number of this loop gives the canonical $\Z/2$-grading, since the different choices of oriented path from $T_pL_1$ to $T_p\bL$  change the winding number of this loop by $2\Z$.  $\delta^{\cL_z, \mathcal{L}_1} $ is an odd map with respect to this $\Z/2$-grading.

$\delta^{\cL_z, \mathcal{L}_1} $ can be linearly extended to
$\Lambda [z_1^\pm,\cdots, z_n^\pm]$-module
homomorphisms $P^0 \to P^1$ and $P^1 \to P^0$, which is still denoted as $ \delta^{\cL_z, \mathcal{L}_1} $.
We define $d_0:P^0 \to P^1$ by $ \delta^{\cL_z, \mathcal{L}_1} $,
and $d_1:P^1 \to P^0$ by $- \delta^{\cL_z, \mathcal{L}_1} $.
Then Floer's equation \eqref{eq2:floer} can be rewritten as a matrix factorization identity:
$$d_0 \circ d_1 = d_1 \circ d_0 = (W(\bL) - \Phi(\mathcal{L}_1) ) \cdot Id.$$

\begin{definition} \label{def:mirMF}
The matrix factorization of $W(\bL)$ mirror to a Lagrangian brane $(L_1,\CL_1)$ is defined as $(P^0, P^1, d)$ given by the above construction.
\end{definition}

The above construction depends on the choice of the base point $p$ of the gauge hypertori. We will show in Lemma \ref{lem:gimf} that different choices of gauge hypertori and base point give rise to matrix factorizations in the same isomorphism class.

\subsection{Generalizations}
In this subsection, we briefly discuss how the construction goes in more general situations (without Assumption \ref{assum1}) using the machinery of \cite{FOOO}.

First recall the definition of weakly unobstructed Lagrangian submanifold.
Let $\bL$ be a Lagrangian torus in a general symplectic manifold. Denote by $(C(\bL), \{m_k\})$ a unital filtered $\AI$-algebra of $\bL$ constructed in \cite{FOOO}
or in \cite{FOOOT} (which can be made unital by taking a canonical model).
Recall that we have 
\begin{equation}\label{eq:mkf}
m_k = \sum_{\beta \in H_2(X,\bL)} m_{k,\beta} \otimes T^{\omega(\beta)}
\end{equation}

Denote by $F^+C(\bL)$ the elements of $C(\bL)$ whose coefficients have positive $T$-exponents.
 An element $b_+ \in F^+C(\bL)$ is called a {\em weak bounding cochain} if
  $m(e^b) = \sum_{k=0}^\infty m_k(b,\cdots,b)$ is a multiple of a unit.
  

Choose a flat connection $\nabla^{b_0}$ of a complex line bundle $\mathcal{L}$ over $\bL$  whose holonomy is $\rho^{b_0}$ for $b_0 = \sum x_i E_i^*$ as  in Section \ref{sec:2}.
We can modify \eqref{eq:mkf} to define
 \begin{equation}\label{eq:mkfr}
 m_k^{\rho} = \sum_{\beta \in H_2(X,\bL)} \rho^b (\partial \beta)m_{k,\beta} \otimes T^{\omega(\beta)}.
 \end{equation}
As explained in \cite{FOOOT}, it has the effect of adding a constant term $b_0$ to $b_+$.
 
We need to make the following assumption on $\bL$ in this general setting.
\begin{assumption}\label{assum2}
We require that there exists $b_+ \in F^+C(\bL)$ such that $b = b_+ + b_0$ is a weak bounding cochain for every $b_0 \in H^1(\bL,\C)/H^1(\bL,\Z)$.
\end{assumption}
As we require the existence of a family of weak bounding cochains, this is stronger than the standard weakly unobstructed condition on $b_+$.
 
It was shown in \cite[Section 4]{FOOOT} that a Lagrangian torus fiber in a compact toric manifold satisfies this assumption (with $b_+ = 0$).  Hence a Floer potential for a torus fiber of any toric manifold can be defined.  In this case, the Floer equation \eqref{eq2:floer} was shown in Lemma 12.7 of \cite{FOOOT}.
 
With Assumption \ref{assum2} on $\bL$, the previous construction of mirror matrix factorization generalizes as follows.
Let $L_1$ be  a weakly unobstructed Lagrangian submanifold of $X$ with a weak bounding cochain $b'$ ($L_1$ does not have to satisfy the stronger assumption
\ref{assum2}).
In addition, we assume that  $L_1$ intersects transversely with $\bL$ (by using Hamiltonian isotopy of $X$ if necessary)
and the intersection is away from the neighborhood of the chosen hyper-tori. 

Denote by $b_z = b_0 + b_+$ to emphasize the dependence on $z$.
Then the Lagrangian Floer complex between $(\bL, b_z)$ and $(L_1,b')$ satisfies the Floer equation \eqref{eq2:floer} (\cite{FOOO}), and hence it is easy to see that we can find the mirror matrix factorization of $(L_1, b')$ in the same way.

On the other hand in actual computations, we find that the pearl complex given by \cite{BC} behaves better, since generators of the complex are given by the critical points
of a Morse function which can be chosen away from the hyper-tori.  We will explain them in Section \ref{sec:pearl}.

\section{Examples}
In this section, we explain the Floer potentials and mirror matrix factorizations through a couple of monotone examples.  We will perform the construction systematically for toric fibers of toric manifolds in Section \ref{sec:toricFano}.

\subsection{$\bP^1$}
Consider $\bP^1$ with total symplectic area $k$.
Take $\bL$ to be an oriented great circle $S^1$ in $\bP^1$.  Fix a point $h \in \bL$, and a flat connection $\nabla_z$ (on a complex line bundle $\cL_z$) which is trivial away from a small neighborhood $U_{h} \subset \bL$ of $h$ and has holonomy $z \in \C^\times$ along $\bL$.

There are two holomorphic discs bounded by $\bL$, namely the upper and lower hemispheres.  The Floer potential equals to
$$W(\bL) =  T^{k/2}( z + z^{-1}).$$
(It is equivalent to the well-known Hori-Vafa mirror $z' + T^{k}/z'$ by the change of coordinate $z' = T^{k/2}z$.)

Take $L_1$ to be another great circle  intersecting transversely with $\bL$ at the two antipodal points $p, q \in \bL\cap L_1$.  It is assumed that $h$ and $U_{h}$ are taken such that $p, q \not\in U_{h}$. Equip $L_1$ with a trivial flat line bundle $\mathcal{L}_1$ with holonomy  $\lambda \in \C^\times$.
(It is known from \cite{FOOOT} that the object $(L_1,\mathcal{L}_1)$ is non-trivial in the Fukaya category if and only if $\lambda = \pm 1$.) 

We also fix the gauge of $\mathcal{L}_1$ by fixing a point $h_1 \in L_1$ and requiring that the flat connection on $\mathcal{L}_1$ is trivial away from a small neighborhood $U_{h_1} \subset L_1$ of $h_1$.  (Again $h_1,U_{h_1}$ are chosen such that $p,q \not\in U_{h_1}$.) Orient $\bL, L_1$ as in Figure \ref{fig:excp11}.  Then $p$ and $q$ have even and odd degree respectively.  The differential is given as
$$ \delta^{\cL_z, \mathcal{L}_1} (p) = T^{k_1}\left( z -\frac{1}{\lambda}\right)q,\quad  \delta^{\cL_z, \mathcal{L}_1} (q)= T^{k_2}\left(-1 + \frac{\lambda}{z}  \right)p$$
%
Hence, 
$$ (\delta^{\cL_z, \mathcal{L}_1})^2 = T^{k/2} \left( \lambda + \frac{1}{\lambda} \right) - T^{k/2} \left( z + \frac{1}{z} \right)$$
If we change the location of $h$ or $h'$, we get a different but isomorphic matrix factorization.  This can be checked directly, and indeed it is a general fact (Lemma \ref{lem:gimf}).

\begin{figure}[htb!]
\begin{center}
\includegraphics[height=1.5in]{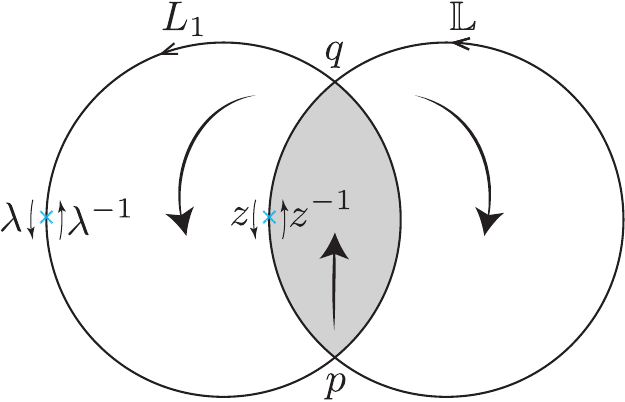}
\caption{$\bL$ and $L_1$ in $\C P^1 \setminus \{ \infty \}$}\label{fig:excp11}
\end{center}
\end{figure}

\subsection{$\mathbb{P}^1$ with another Lagrangian}
$L_1$ in the last subsection is obtained from the great circle $\bL$ by rotation of $\bP^1$, which is a Hamiltonian isotopy.  Let us take another Hamiltonian isotope $L_2$ of $\bL$ as shown in Figure \ref{fig:cp1pert2ex2}.  In particular the areas $\alpha, \beta, \gamma, \delta$ in Figure \ref{fig:cp1pert2ex2} satisfy the relation $\alpha + \beta = \gamma + \delta$.  For simplicity equip $L_2$ with the trivial holonomy.    

\begin{figure}[htb!]
\begin{center}
\includegraphics[height=2.3in]{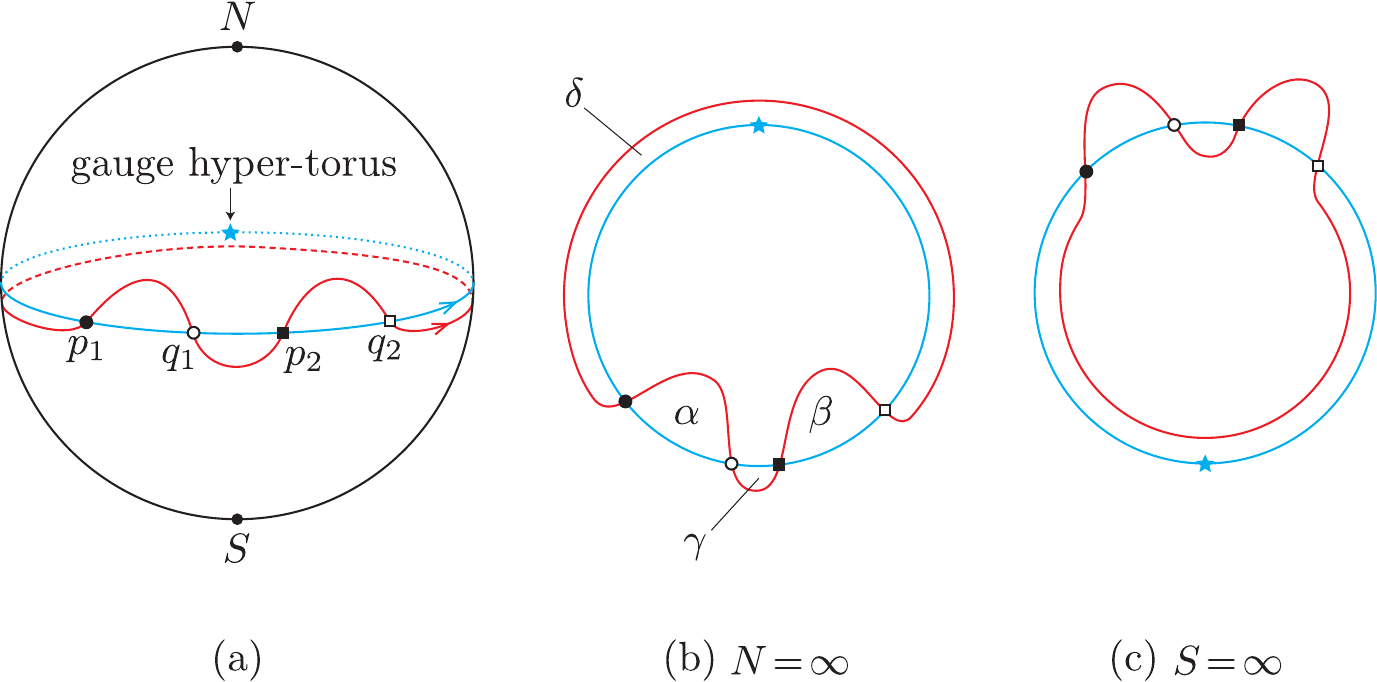}
\caption{$L_2 \subset \mathbb{P}^1$}\label{fig:cp1pert2ex2}
\end{center}
\end{figure}

By Definition \ref{def:mirMF} and counting holomorphic strips, it is easy to obtain the matrix factorization mirror to $L_2$:
\begin{equation}\label{eqn:L2isotL1P1}
\bordermatrix{ & p_1 & p_2 \cr
q_1 & T^{\frac{k}{2} - \alpha}(z -1) & T^{\frac{k}{2}-\gamma}(1-\frac{1}{z}) \cr
q_2 & T^{\frac{k}{2} - \alpha - \beta + \gamma}(z -1) & T^{\frac{k}{2} - \beta}(z-1) \cr }, \quad 
\bordermatrix{ & q_1 & q_2 \cr
p_1 & T^\alpha   & -T^{\alpha + \beta - \gamma} \frac{1}{z} \cr
p_2 & - T^\gamma & T^\beta \cr }.
\end{equation}
Since $L_2$ is Hamiltonian isotopic to $L_1$ in the previous section, one naturally expect that \eqref{eqn:L2isotL1P1} would give an equivalent object in the mirror of $\mathbb{P}^1$. We will prove this fact in more general setting in Section \ref{sec:haminv}.

\subsection{ $\mathbb{P}^1 \times \mathbb{P}^1$}\label{subsec:cp1cp1}
Consider the symplectic manifold $\mathbb{P}^1 \times \mathbb{P}^1$ whose moment map image is $[-\frac{k}{2\pi},\frac{k}{2\pi}]^2$.
We compute matrix factorizations for two specific Lagrangian submanifolds, namely the toric fiber at $0 \in [-\frac{k}{2\pi},\frac{k}{2\pi}]^2$ and the
anti-diagonal $A$. 

Note that toric fibers split-generates the Fukaya category of $\mathbb{P}^1 \times \mathbb{P}^1$ by the result of \cite{AFOOO}. In \ref{sec:qe}, we will see that their mirror matrix factorizations are also split-generators.  
The anti-diagonal $A$ is another interesting object in the Fukaya category, which  appears (as a Lagrangian) only when two $\mathbb{P}^1$-factors have the same symplectic forms. Similar phenomenon happens on the mirror side explained in \cite[8.2]{KL}.
$A \oplus A[1]$ is isomorphic to the sum of two toric fibers with holonomies $(\pm1, \mp1)$ (which is conjectured in \cite{CL} and proven in \cite{CHLee}).

Let $S_0$ and $S_1$ be two distinct oriented great circles of $\mathbb{P}^1$ which intersect transversely with each other at two antipodal points.
Denote the two intersection points of $S_0$ and $S_1$ by $p$ (which has odd-degree) and $q$ (which has even degree).  Set $\bL = S_0 \times S_1$, which is Hamiltonian isotopic to the toric fiber at $0 \in [-\frac{k}{2\pi},\frac{k}{2\pi}]^2$.  Let $\cL_z \to \bL$ be the flat line bundle with holonomy $(z_1, z_2)$.

$\bL$ is monotone, and its Floer potential is given by
\begin{equation}\label{lfpcpcp}
W(\bL) (z_1,z_2) = T^{k/2}\left( z_1 + \frac{1}{z_1} + z_2 + \frac{1}{z_2} \right).
\end{equation}

Let $L_1 = S_1 \times S_0$ and it is equipped with the flat line bundle $\CL_1$ with holonomy $(\lambda_1, \lambda_2)$.
Note that $L_1$ is Hamiltonian isotopic to $\bL$ and hence the toric fiber at $0 \in [-\frac{k}{2\pi},\frac{k}{2\pi}]^2$.

The matrix factorization corresponding to $L_1$ can be found by using the previous calculation for $\mathbb{P}^1$ as follows.
The intersection $\bL \cap L_1$ consists of 4 points, namely two even intersections $(p,p),(q,q)$ and two odd intersections $(p,q), (q,p)$.
The differential $\delta^{\cL_z, \mathcal{L}_1}$ is represented by the matrices
\begin{equation}\label{MFcpcpf}
A:=\bordermatrix{ & (p,p) & (q,q) \cr
(p,q) & T^{k_1}(\lambda_2 z_2 -1) & T^{k_2}(-\frac{1}{\lambda_1}
+ \frac{1}{z_1}) \cr
(q,p) &T^{k_1}(\lambda_1 z_1 -1) & -T^{k_2}(-\frac{1}{\lambda_2}
+ \frac{1}{z_2}) \cr }, 
B:=\bordermatrix{ & (p,q) & (q,p) \cr
(p,p) & T^{k_2}(-\frac{1}{\lambda_2}
+ \frac{1}{z_2})    & T^{k_2}(-\frac{1}{\lambda_1}
+ \frac{1}{z_1}) \cr
(q,q) &T^{k_1}(\lambda_1 z_1 -1) & - T^{k_1}(\lambda_2 z_2 -1)\cr }
\end{equation}
with $k_1+k_2 = k/2$.  The additional signs at (2,2) positions of the matrices come from Koszul sign convention.
One can check that this gives a matrix factorization of \eqref{lfpcpcp}. (Compare it with the one given in \cite[Remark 4.3]{CHLee}.)
Hence we obtain
\begin{prop}
The matrix factorization mirror to $L_1$ 
is given by rank 2 free modules $P_0, P_1$ with  $d_0 = A   , d_1 =- B$, which satisfies
$$d_0 d_1 = d_1 d_0 = (W(\bL)(z_1,z_2) - W(\bL)(\lambda_1, \lambda_2)) \, \textrm{Id}_{2\times 2}.$$
\end{prop}

In the following we compute the matrix factorization mirror to the anti-diagonal which was first studied in \cite{CHLee}.  The anti-diagonal $A$ is defined by
$$A:=\left\{ ([z:w],[\bar{z}:\bar{w}]) \,\, | \, \, [z:w] \in \mathbb{P}^1 \right\} \subset \mathbb{P}^1 \times \mathbb{P}^1.$$
We remark that $\Phi(A)= 0$ since the minimal Maslov index for $A$ is 4. We fix $\mathcal{L}_A$ to be
the flat line bundle on $A$ with trivial holonomy.
The intersection $\bL \cap A$ consists of 2 points $(p,p), (q,q)$,
which generate $CF((\bL,\cL_z),(A,\mathcal{L}_A))$.

Recall the following doubling argument from \cite[Proposition 4.1]{CHLee}.
\begin{lemma}\label{lem:double}
There is an one-to-one correspondence between
\begin{center}
$\{$holomorphic strips in $\mathbb{P}^1 \times \mathbb{P}^1$ bounded by $\bL$ and $A$$\}$ and\\
$\{$holomorphic strips in $\mathbb{P}^1$ bounded by $S_0$ and $S_1$$\}$.
 \end{center}
 Moreover, the correspondence preserves symplectic areas.
\end{lemma}
\begin{proof}
Let $u=(u_1, u_2) : \R \times [0,1] \to \mathbb{P}^1 \times \mathbb{P}^1$ be a holomorphic strip between $S_0 \times S_1$ and $A$.  $u_1$ and $u_2$ can be glued as in (a) of Figure \ref{fig:cp1cp1} to give a holomorphic strip between $S_0$ and $S_1$. For more details, see \cite{CHLee}.


\end{proof}
Therefore the differential from $(p,p)$ to $(q,q)$
is $T^{k_1} (1+z_1 z_2)$, and the differential from $(q,q)$ to $(p,p)$ is $T^{k_2} (\frac{1}{z_1} + \frac{1}{z_2})$ (see (b) of Figure \ref{fig:cp1cp1}).
Hence we obtain 
\begin{figure}[htb!]
\begin{center}
\includegraphics[height=2in]{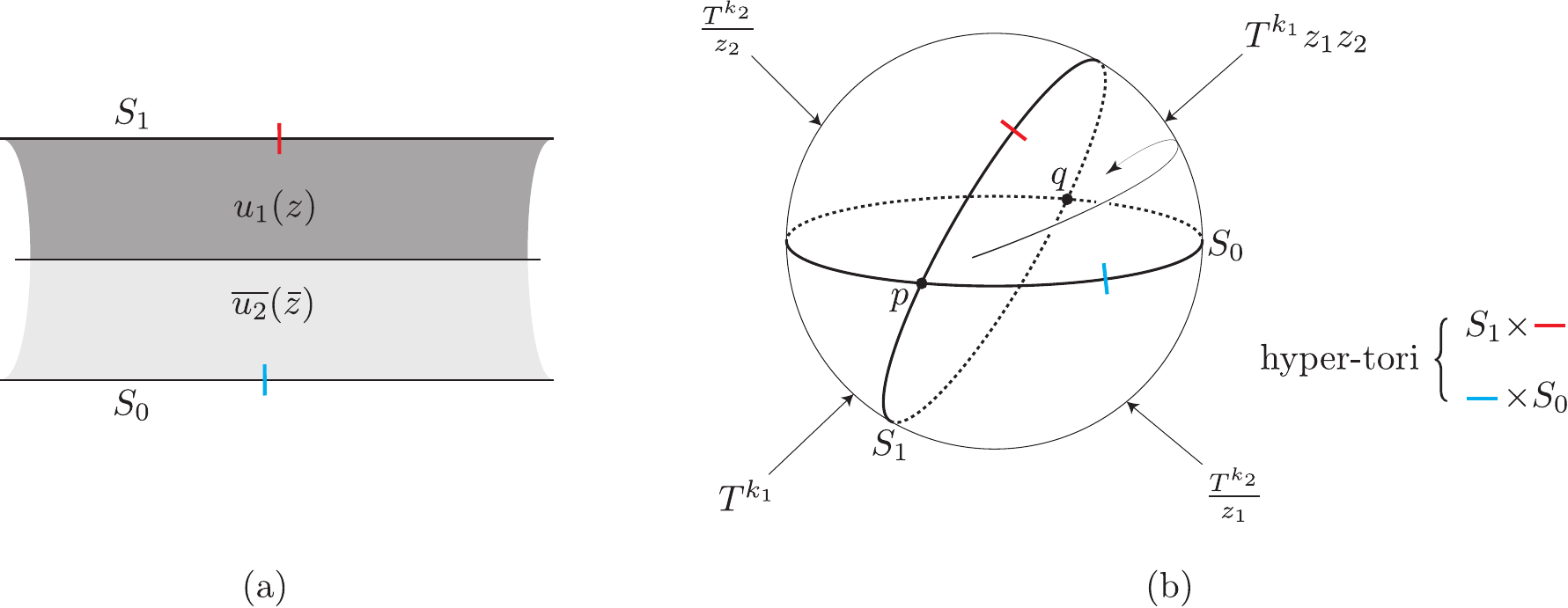}
\caption{Gluing argument for $(S_1 \times S_0, A)$}\label{fig:cp1cp1}
\end{center}
\end{figure}
\begin{prop}
The anti-diagonal $A$ is mirror to the matrix factorization 
$$ W(\bL) =\left(T^{k_1}(1 + z_1 z_2 )\right) \,\cdot \,\left(T^{k_2}\left(\frac{1}{z_1} + \frac{1}{z_2} \right)\right).$$
\end{prop}
The result agrees with the one in \cite{CHLee} which was obtained by using SYZ-fibration structure.

\section{Invariance of matrix factorizations under various choices}\label{sec:inv}
We have made a choice of gauge hypertori and Hamiltonian perturbations on Lagrangian submanifolds.
In this section, we show that a different choice of gauge hypertori gives rise to an isomorphic matrix factorization, and Hamiltonian isotopic Lagrangians induce homotopic matrix factorizations.

\subsection{Infinitesimal gauge equivalences and a choice of gauge hypertori}\label{subsec:gegh}
We first show that the isomorphism class of mirror matrix factorization does not depend on the choice of gauge hypertori.
Indeed, we may interpret changing gauge hypertori as an infinitesimal gauge equivalence.
For a trivial line bundle $E:=\bL \times \C$ on $\bL$ and two flat connection $\nabla_i$ on $E$ $(i=0,1)$, we define the infinitesimal gauge equivalence relation between two flat connections $\nabla_0$ and $\nabla_1$ as follows.
\begin{definition}
$\nabla_0$ and $\nabla_1$ are said to be infinitesimally gauge equivalent if 
we have a trivial bundle $E \times (-\epsilon', 1 +\epsilon') \to \bL \times (-\epsilon', 1 +\epsilon')$ and
a flat connection $\nabla$ on $E \times (-\epsilon', 1 +\epsilon')$ such that
$\nabla$ when restricted to $E \times \{0\} \to \bL \times \{0\}$ becomes $\nabla_0$
and when restricted to $E \times \{1\} \to \bL \times  \{1\}$ becomes $\nabla_1$.
Here $\epsilon'$ is a small positive real number.
\end{definition}
Now consider gauge hypertori 
$\{ H_i + p\} $ for two different choices of $p \in (\R/\Z)^n$.
\begin{lemma}
Flat connections constructed in Lemma \ref{lem:flat conn} from two different choices of gauge hypertori  $\{ H_i + p^j\}_{i=1}^n$ (for $j=0,1$)  are infinitesimally gauge equivalent.
\end{lemma}
\begin{proof}
We prove it for a trivial bundle $\R \times \C$ over $\R$, and the proof for a torus $\bL$ is  similar. 
Consider $p^0, p^1 \in \R$, and 
the corresponding flat connections $\nabla_j$ defined by
$$\nabla_j=d - 2\pi\bi b \,\delta(t- p^j) dt$$
for $j=0,1$.
Then the desired flat connection $\nabla$ on $\R \times \C \times (-\epsilon', 1 +\epsilon') \to \R \times (-\epsilon', 1 +\epsilon') $ can be defined as
$$\nabla=d - 2\pi\bi b \, \delta(t- p^0 + (p^0 - p^1) s) dt - 2(p^0 - p^1)\pi\bi b \, \delta(t- p^0 + (p^0 - p^1)s) ds,$$
where $s$ is a coordinate on $(-\epsilon', 1 +\epsilon')$. One can easily check that  $\nabla$ is flat, and it gives infinitesimal gauge equivalence between $\nabla_0$ and $\nabla_1$.
\end{proof}

Let $(P^0_{p^j}, P^1_{p^j}, d^j)$ be the matrix factorization corresponding to the gauge hypertori $\{ H_i + p^j\}_{i=1}^n$.
Using the infinitesimal gauge equivalence, we define the chain isomorphism between these two matrix factorizations.
\begin{lemma}\label{lem:gimf}
There exists a chain isomorphism between two matrix factorizations $(P^0_{p^j}, P^1_{p^j}, d^j)$ for $j=1,2$, that is, we have an isomorphism
$\phi: P^\bullet_{p^0} \to P^\bullet_{p^1}$ 
such that $\phi \circ d = d \circ \phi$.
\end{lemma}
\begin{proof}
Let ($\WT{\CL}, \WT{\nabla})$ be the trivial bundle over $\bL \times  (-\epsilon', 1 +\epsilon')$, which
defines the infinitesimal gauge equivalence between two flat connections on $\CL$ from two different gauge hypertori.

We define $\phi$ for each generator $p \in \bL \cap L_1$ to be an identity map, multiplied by
a holonomy of the bundle $\WT{\CL}$ along an interval  $ p \times [0,1]$.
To see that this gives a chain isomorphism between two matrix factorizations,
observe that the parallel transport with respect to the flat connection $\WT{\nabla}$ does not depend on the choice of a path with fixed end points. This directly leads to the chain map property $\phi \circ d = d \circ \phi$. One can construct its inverse in a similar way.
\end{proof}
In particular, this prove that the isomorphism class of the mirror matrix factorization under our construction is independent of a choice of gauge hypertori to define the connection $\nabla$.

 
\subsection{Hamiltonian isotopy}\label{sec:haminv}
Consider a Hamiltonian isotopy $\psi$ of $X$, and two Lagrangian submanifolds, $L_1$ and $L_2:=\psi(L_1)$.
For a positive Lagrangian torus $\bL$, with Floer potential $W(\bL)$, 
we obtain two matrix factorizations, $(P^0(L_i), P^1(L_i),d)$ from $L_i$ for $i=1,2$.
\begin{prop}\label{prop:haminv}
Two matrix factorizations  $(P^0(L_i), P^1(L_i),d)$ from $L_i$ for $i=1,2$ of $W(\bL)$ are homotopic to each other.
Namely, there exist  even maps $f :  P^*(L_1) \to P^*(L_2)$, $g: P^*(L_2) \to P^*(L_1)$
such that 
$$g \circ  f \sim \textrm{id},\quad f \circ g \sim \textrm{id}.$$
\end{prop}
\begin{proof}
The standard continuation maps in Floer theory are given by 
$$f_\sharp: CF(\bL, L_1) \to CF(\bL,\psi(L_1)), \quad g_\sharp:  CF(\bL,\psi(L_1)) \to CF(\bL, L_1),$$ 
(see for example \cite{Fl2}, \cite{Oh1} ) which satisfies  
$$g_\sharp \circ f_\sharp = \delta \circ H_1 - H_1  \circ \delta, \quad f_\sharp \circ g_\sharp = \delta \circ H_2 - H_2  \circ \delta.$$
Here, homotopy $H_i$ is also constructed by counting parametrized version of $J$-holomorphic strips.

Now, to prove the proposition, we introduce flat line bundles $\cL_z, \CL_1$ with connections $\nabla_z, \nabla_1$ on $\bL$, $L_1$. ($L_2 = \psi(L_1)$ is  equipped with a flat line bundle $(\psi^{-1} )^\ast (\cL_1)$). Then
the same construction with an addition of holonomies yields required maps between matrix factorizations:
we denote the corresponding maps as $f:=f_\sharp^b, g:=g_\sharp^b, h_1:=H_1^b$ and  $h_2:=H_2^b$, which satisfies
$$g \circ f = d \circ h_1 - h_1 \circ d, \quad f \circ g = d \circ h_2 - h_2 \circ d.$$
\end{proof}

Consequently, if $L_2$ is a Hamiltonian isotopy image of a Lagrangian submanifold $L_1$,
then the resulting matrix factorization of $L_1$ and $L_2$ are quasi-isomorphic.

\section{$\AI$-functor}\label{sec:functor}
In this section, we extend the previous construction of the correspondence between Lagrangians and matrix factorizations in the object level to an $\AI$-functor from the Fukaya category of $X$ to the matrix factorization category of $W(\bL)$.

\subsection{Preliminaries}
Let us first recall the definition of a filtered $\AI$-category to set up notations (see  \cite{Fu1}
for details).

\begin{definition}
A filtered $\AI$-category $\mathcal{C}$ consists of a collection of objects $Ob(\mathcal{C})$ together with
morphisms $\mathcal{C}(C_1,C_2)$ for $C_1,C_2 \in Ob(\mathcal{C})$ given by a graded torsion-free filtered $\Lambda_0$-module,
and degree-$1$ operations $m_k$'s on morphisms 
$$m_k: \mathcal{C}[1](C_1,C_2) \otimes \cdots \mathcal{C}[1](C_{k-1},C_k) \to \mathcal{C}[1](C_0, C_k)$$
for $k=0,1,\cdots$, which respects the filtration and satisfy the following $\AI$-equations:
 for $x_i \in \mathcal{C}[1](C_{i-1},C_i)$ ($i=1,\cdots, n$), we have
 $$\sum_{k_1+k_2= k+1} \sum_{i=1}^{k_1} (-1)^{\epsilon} m_{k_1}(x_1,\cdots, x_{i-1}, m_{k_2}(x_i,\cdots,x_{i+k_2-1}),
 \cdots, x_k)=0$$
where $\epsilon = \sum_{j=1}^{i-1}(  |x_j|+1)$.
\end{definition}
A {\em filtered differential graded category} $\mathcal{C}$ is a filtered $\AI$-category with $m_0=m_{\geq3} \equiv 0$.

An $\AI$-functor $\mathcal{F}$ between two filtered $\AI$-categories $\mathcal{C}_1, \mathcal{C}_2$ can be defined
as follows. First, we have a map between objects $\mathcal{F}_0: Ob(\mathcal{C}_1) \to Ob(\mathcal{C}_2)$.
And given $A,B \in Ob(\mathcal{C}_1)$, 
we have homomorphism of degree 0:
$$\mathcal{F}_j(A,B) :  \bigoplus_{A=C_0,C_1,\cdots,C_j = B}
 \mathcal{C}[1](C_0,C_1) \otimes \cdots \mathcal{C}[1](C_{j-1},C_j) \to \mathcal{C}_2[1](\mathcal{F}_0(A), \mathcal{F}_0(B)),$$
for each $j =1,2,\cdots$ which satisfies $\AI$-functor equation (see \cite{Fu1}).
  %
%

Our main concerns are the following two $\AI$-categories: one is the Fukaya category of a symplectic manifold $X$, and the other is  a filtered differential graded category $\mathcal{MF}(W)$,
of matrix factorizations of $W(\bL)$.

Let us first define the matrix factorization category which is a differential graded category.
\begin{definition}
Let $\mathcal{O}$ be the algebra $\Lambda[z_1^{\pm 1},z_2^{\pm 1}, \cdots, z_n^{\pm 1}]$.
The category of matrix factorization $\mathcal{MF}(W)$ of $W \in \mathcal{O}$ is defined as follows.
An object of $\mathcal{MF}(W)$ is a finite dimensional free $\Z/2$-graded $\mathcal{O}$-modules $P= P^0 \oplus P^1$,
together with an odd map $d:P \to P$ such that 
$$d^2 = (W  - \lambda )\cdot \textrm{Id}_P$$ for some $\lambda \in \C$.

A morphism between two matrix factorizations $(P, d_P), (Q, d_Q)$ is given by
an $\mathcal{O}$-module homomorphism $f:P \to Q$. A differential $d$ on a morphism $f$ is
given as $$d (f) = d_Q \circ f + (-1)^{\textrm{deg f}} f \circ d_P,$$
and composition of morphisms are defined as usual.
\end{definition}
Recall that differential graded category $\mathcal{C}$ with differential $d$, and composition $\circ$
gives rise to an $\AI$-category with the following sign convention.
$$m_1 (x) = (-1)^{\deg x} d(x), m_2(x_1,x_2) = (-1)^{\deg x_1( \deg x_2+1)} x_2 \circ x_1.$$
From now on, we regard $\mathcal{MF}(W)$ as a filtered $\AI$-category with vanishing $m_0, m_{\geq 3}$.

The Fukaya category $\mathcal{F}uk(X)$ of a symplectic manifold $X$ is defined as follows.
We only sketch the setup briefly, and leave detailed construction to \cite{Fu1} and \cite{FOOO}.
An object of Fukaya category $\mathcal{F}uk(X)$ is a spin (oriented) Lagrangian submanifold  with a  flat complex line bundle, and in addition, it is assumed to be {\em positive} (or in general  weakly unobstructed in the sense of \cite{FOOO}).
We remark that the grading datum is not included, as we use $\Z/2$-grading. We require an object to be spin (not  relatively spin)  so that Lagrangian Floer complex with $L_0$ is defined over a characteristic-$0$ field $\Lambda$.
(But we can work in more general setting as illustrated in Section \ref{sect:RP}.)

Morphisms and $m_1$ between two objects $L_0, L_1$ are given by the Lagrangian Floer complex $CF^*(L_0, L_1)$ as explained before.  (Here we omit the notation of flat line bundles on $L_i$ for simplicity.)

Higher morphism $m_k$ is defined by counting $J$-holomorphic polygons:
For distinct Lagrangian submanifolds $L_0,\cdots, L_k$, the $\AI$-operation $m_k$  is
defined as
\begin{equation}\label{eqmkfc}
m_k: CF(L_0,L_1) \times \cdots \times CF(L_{k-1}, L_k) \to CF(L_0, L_k)
\end{equation}
$$m_k(p_1,\cdots, p_k) = \sum_q  n(p_1,\cdots,p_k,y) y$$
where $n(p_1,\cdots,p_k,y)$ are the contributions from signed counting of  $J$-holomorphic
polygon $u$'s together an symplectic area $T^{w(u)}$ with  holonomy effects from flat complex line bundles along their
boundary $u(\partial D^2)$. (see Definition 3.26 \cite{Fu1}).
Here, $J$-holomorphic polygon above is a map $u$ from $D^2$ with $k+1$ punctures $z_0, \cdots, z_{k} \in \partial D^2$
such that a part of the boundary $\partial D^2$ between $z_i, z_{i+1}$ is required to map into $L_i$ for $i=0,\cdots, k$
and the map $u$ limits to the intersection point $p_i$ at the puncture $z_i$ for $i=1,\cdots, k$ whereas
$u$ limits to $\bar{y}$ at the puncture $z_0$. Here, $\bar{y}$ is the intersection point $y$ regarded as an element of $CF(L_k,L_0)$.
Since we assumed that Lagrangians are positive(Assumption \ref{assum1}), we can use domain-dependent perturbations (as in \cite{Se}) to make
the above operation transversal, and satisfy $\AI$-operations between transversal Lagrangians. Lagrangians here can have nontrivial $m_0$ given by Maslov index disc two contributions.

When Lagrangians are not distinct, the construction of $m_k$ needs more advanced machinery
such as Kuranishi structures, and we refer readers  to \cite{Fu1} for details. We remark that one may instead work with $\AI$-pre-categories
defined by  Kontsevich-Soibelman \cite{KS}. The construction of $\AI$-functor will resemble  the Yoneda embedding, and
hence functor is well-defined once the Fukaya category itself is well-defined, which we  assume from now on.

\subsection{Localized mirror functor}
Let us fix a reference positive Lagrangian torus $\bL$ in a symplectic manifold $X$ of real dimension $2n$.
Let us denote its Floer potential by $W(:=W(\bL))$. 
We regard the potential  $W$ as an element of  $\Lambda \ll z_1^{\pm 1},z_2^{\pm 1}, \cdots, z_n^{\pm 1} \gg$.

We fix gauge hypertori $\{p_i + H_i\}_{i=1}^n$ of $\bL$, and its sufficiently small neighborhood $U \subset \bL$.
For the Fukaya category $\mathcal{F}uk(X)$ of $X$, we suppose that it has only countably many objects $(L_i, \mathcal{L}_i)$ for $i =1,2, \cdots$.
We may assume that each Lagrangian submanifold $L_i$ is transverse to $\bL$.
Furthermore, we may assume that the intersection $L_i \cap \bL $ is away from the gauge hypertori and in particular disjoint from $U$.
This can be done by taking a suitable Hamiltonian isotopy of $L_i$'s in the following way if necessary.
For any finite, say $k$ points of $\bL$,  we can move it to  another configuration of $k$ points
by a Hamiltonian isotopy preserving $\bL$ (see Lemma 2.7 \cite{W}). Hence, for any $L_i$, we can move points in $L_i \cap \bL$
away from $U$ by a Hamiltonian isotopy $\phi_H$, and we take $\phi_H ({L_i})$ as an object of the Fukaya category $\mathcal{F}uk(X)$
instead of $L_i$.  

In particular, we do not take the reference Lagrangian $\bL$ itself to be an object of this Fukaya category $\mathcal{F}uk(X)$, but take a suitable Hamiltonian isotopy image of $\bL$ (to be one of $L_i$'s).
The above step is essential to define the mirror functor.
In fact, we first consider slighted extended version of Fukaya category $\mathcal{F}uk_{+0}(X)$ whose
objects are $\{(\bL, \cL) \} \cup \{ (L_i, \mathcal{L}_i) :  i=1,2,\cdots\}$, satisfying the above conditions, and restrict to the objects of $i=1,2,\cdots$
to obtain Fukaya category $\mathcal{F}uk(X)$.

Let us denote by $\{m_k\}_{k=1}^\infty$ the $\AI$-operations on $\mathcal{F}uk_{+0}(X)$ (and then those on $\mathcal{F}uk(X)$ by restriction). Now, recall that for $\bL$, the holonomy of the bundle $\cL_z$ is written in mirror parameters $z_i = \rho^b(E_i)$  (see equation \eqref{eq:defrho}), where $b= \sum x_i E_i^* \in H^1(\bL,\C)/H^1(\bL,\Z)$. To highlight the commonly used notation $b$ of deformation parameter, we write $\cL^b_z$ for $\cL_z$ in this section.

In what follows, we will always put $\bL$ at the first slot in the $\AI$-operation \eqref{eqmkfc}.
Then, we will modify the definition $m_k$ by incorporating the effect of $\cL_z^b$.
Namely, for a relevant $J$-holomorphic polygon $u$, we will record the holonomy contribution of $\cL^b$ along the arc $u(\partial D^2)$ between $0$-th marked point $p_0$ and
$1$-st marked point $p_1$.   Like in the formula \eqref{eq:mkfr}, we modify $m_k$ by
multiplying this holonomy effect along $p_0\,p_1$-arc, and denote it as $m_k^{b,0,\cdots,0}$.
This is not exactly the same as $m_k^{b,0,\cdots,0}$ defined in \cite{FOOO}, but it should be considered as its
line bundle analogue. (Here, we do not define $m_k$-operation for $E_i^*$, and this is the reason of potential confusion. For the proper comparison, we can set a geometric representative  $\widetilde{b} =  \sum z_i (p + H_i)$,
and then what we define may be considered as $m_k^{\widetilde{b},0,\cdots,0}$ defined in \cite{FOOO}.)

In any case, what is important is that we obtain the correct the $\AI$-equation for $m_k^{b,0,\cdots,0}$
from the Gromov-Floer compactification of the moduli space of $J$-holomorphic polygons, by tracking the arc that
we take holonomy along.

From our assumption that the intersection points are away from the chosen gauge hypertori for  $\mathcal{F}uk(X)$,
such a holonomy for $m_k^{b,0,\cdots,0}$ of $\mathcal{F}uk(X)$ is always given as a Laurent monomial in  $z_1,\cdots, z_n$.
For example, the  differential $\delta^{\cL_z^b, \mathcal{L}_i}$ of the Floer complex $ CF((\bL,\cL_z^b), (L_i, \mathcal{L}_i))$ is nothing but $m_1^{b,0}$ between $(\bL, \cL_z^b)$ and $(L_i,\CL_i)$.

\begin{definition}
The $\AI$-functor 
$$\mathcal{LM} : \mathcal{F}uk(X) \to \mathcal{MF}(W).$$
is defined as follows.
\begin{enumerate}
\item  For objects, $\mathcal{LM}_0$ sends an object $(L_i, \mathcal{L}_i)$ of the Fukaya category to
the matrix factorization obtained by the Lagrangian Floer complex 
$$MF(L_1):= \big( CF(\bL,\cL_z^b), (L_i, \mathcal{L}_i)), m_1^{b,0} \big)$$
\item For $x_1 \in CF(L_i, L_j)$,
$$\mathcal{LM}_1(x_1): CF((\bL,\cL_z^b), L_i) \to CF((\bL,\cL_z^b), L_i)$$
is defined as follows. For $y \in CF((\bL,\cL_z^b), L_i)$,
$$\mathcal{LM}_1(x_1) (y) = (-1)^{(\deg x_1+1)(\deg p +1)} m_2^{b,0,0} (y,x_1). $$
\item Similarly,  $\mathcal{LM}_k$ is defined as follows.
For $x_j \in CF(L_{i_j}, L_{i_{j+1}})$ for $j=1,\cdots,k$,
$$\mathcal{LM}_k(x_1,\cdots,x_k): CF((\bL,\cL_z^b), L_{i_1}) \to CF((\bL,\cL_z^b), L_{i_k})$$
sends $y \in CF((\bL,\cL_z^b), L_{i_1})$ to
$$\mathcal{LM}_k(x_1,x_2,\cdots,x_k) (y) = (-1)^{(k+ \sum_i \deg x_i)(\deg p +1)} m_k^{b,0,\cdots,0} (y,x_1,\cdots,x_k). $$
\end{enumerate}
\end{definition}

We refer readers to Section 2 of \cite{CHL} for the algebraic formalism of mirror functor.  The proof of Theorem 2.18 in \cite{CHL} carries over to the setting here and give
\begin{theorem} \label{thm:functor}
The collection of maps $\{\mathcal{LM}_k\}$ defines an $\AI$-functor.
\end{theorem}

We remark that the above construction naturally generalizes to weakly unobstructed Lagrangians.
Namely, the condition that $L_i$'s are positive Lagrangians in a Fano manifold can be relaxed to the condition that $L_i$'s are
weakly unobstructed Lagrangians in a symplectic manifold in the sense of \cite{FOOO}.  In this case,  we consider $(L_i, b_i)$ for the weak bounding cochain $b_i$ of $L_i$ as an object of  $\mathcal{F}uk(X)$, and then replace $m_k^{b,0,\cdots, 0}$ by $m_k^{b, b_1,\cdots, b_k}$. The rest of the procedure is the same, and  we leave the details as an exercise (see also Theorem 2.19 \cite{CHL}).

\section{Matrix factorizations and pearl complex}\label{sec:pearl}
Let $\bL$ be a positive Lagrangian torus with a Floer potential $W(\bL)$.  In the last section, we have defined an $\AI$-functor transforming Lagrangian branes $(L_1,\CL_1)$ to a matrix factorization (see also Definition \ref{def:mirMF}).  In this section, we consider the case of transforming $\bL$ itself, namely  $L_1= \bL$ with a fixed line bundle $\cL_1$ on it.

From Proposition \ref{prop:haminv} on Hamiltonian invariance, we may take
a Hamiltonian isotopy $\phi$ such that $\phi(\bL)$ intersects $\bL$ transversely and define its mirror matrix factorization by the Floer complex of $(\bL, \phi(\bL))$. On the other hand, Bott-Morse Floer theory is very useful for computation.  In this section,  we define the mirror matrix factorization by using pearl complex defined by Biran-Cornea \cite{BC}  (the idea of such a complex appeared in \cite{O_pearl} earlier).  As an application we compute the matrix factorization mirror to the Clifford torus in $\mathbb{P}^2$, which agrees with the result of \cite{CL} from a different approach.  Later in Section \ref{subsec:toric}, we use pearl complex to construct mirror matrix factorizations for general toric Fano manifolds.

\subsection{Pearl complex with decoration by flat bundles}
 We first recall from \cite{BC} the set-up of a pearl complex for a positive spin Lagrangian submanifold $\bL$.
  We fix a generic Morse-Smale function $h$ on  $\bL$.
The pearl chain complex $C^{pearl}(\bL, h)$ is a $\Z/2$-graded free $\Lambda_\Z$-module generated by critical points of $h$,
where the $\Z/2$-grading come from Morse indices of $h$.
The differential of the pearl complex is given by counting pearl trajectories.
\begin{figure}[htb!]
\begin{center}
\includegraphics[height=0.5in]{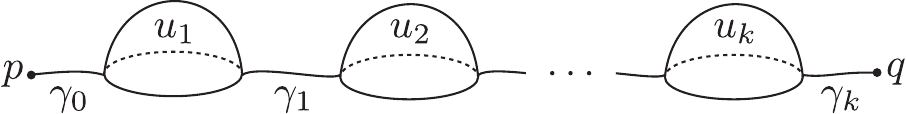}
\caption{A pearl trajectory}\label{fig:pearlcpxtraj}
\end{center}
\end{figure}

First, for each pair of critical points $p,q$ of Morse function $h$,  let $\mathcal{M}_2 (p,q)$ be the moduli space of
pearl trajectories as in Figure \ref{fig:pearlcpxtraj}, where they consist of gradient trajectories of $-h$ together with
$J$-holomorphic discs. More precisely, we have a collection of gradient trajectories
$$\gamma_0:(-\infty,t_0] \to \bL, \gamma_1:[t_0,t_1]\to \bL, \cdots, \gamma_{k-1}:[t_{k-2},t_{k-1}] \to \bL,
\gamma_{k}:[t_{k-1},\infty) \to \bL,$$
satisfying 
$$\lim_{t \to -\infty} \gamma_0(t) = p, \lim_{t \to \infty} \gamma_k(t) = q, \;\; \frac{d \gamma_i}{dt} = - \nabla h (\gamma_i(t)) \; \textrm{for all}\; i=0,1,\ldots,k, $$
and a collection of somewhere injective $J$-holomorphic discs $u_1,\cdots, u_k:(D^2,\partial D^2) \to (X,\bL)$ satisfying
$$\gamma_j(t_j) =u_{j+1}(-1), j=0,1,\ldots,k-1,$$
$$ \gamma_{j}(t_{j-1})=u_{j}(+1), j=1,2,\ldots, k.$$
If we define the total Maslov index $\mu$ by $\sum_{j=1}^n \mu(u_j)$, then the expected dimension of $\CM_2 (p,q)$ is
given by $\textrm{ind}(p) - \textrm{ind}(q) + \mu -1$. 
A pearl trajectory is a collection $\{\gamma_j\}_{j=0}^k \cup \{u_j\}_{j=1}^k$ satisfying the above conditions, and we denote it by $u$.

Now the differential for the pearl complex is given by
\begin{equation}\label{eq:pft1}
\delta_{\textnormal{pearl}} ( \langle p \rangle ) = \delta_{{\rm Morse}} (\langle p \rangle) +  \sum_{q, \textrm{ind}(q) - \textrm{ind}(p) = \mu -1 >1 } (-1)^{{\rm ind} (q)} n(p,q) \langle q \rangle,
\end{equation}
where $\delta_{{\rm Morse}}$ is the usual Morse differential, and $n(p,q)$ is the signed count of isolated pearl trajectories denoted by $u$ between $p$ and $q$ weighted by
the symplectic area $T^{\sum_{j=1}^k \omega(u_j)}$.  It is proved in   \cite[Lemma 5.1.3]{BC} that $\delta_{\textnormal{pearl}}^2=0$.

It is easy to see that $\delta_{\textnormal{pearl}}$ in fact can be decomposed in terms of the total Maslov index $\mu = \sum_{j=1}^n \mu (u_j)$,
since we know that $\mu \geq 0$ (the equality holds for Morse trajectories).
That is, we can write
\begin{equation}\label{eq:decompdpearl}
\delta_{\textnormal{pearl}} = (\delta_{\textnormal{pearl}})_{1} + (\delta_{\textnormal{pearl}})_{-1} + (\delta_{\textnormal{pearl}})_{-3} + \cdots
\end{equation}
where $(\delta_{\textnormal{pearl}})_{i}$ are contributed by the trajectories from $p$ to $q$ with $ \textrm{ind}(q) =  \textrm{ind}(p)-i$. Thus, $(\delta_{\textnormal{pearl}})_i$ increases the degree by $i$, where $\deg = \dim-\textrm{ind}$. In particular, $(\delta_{\textnormal{pearl}})_1$ equals to $\delta_{\rm Morse}$ in \eqref{eq:pft1}.
The above sum is finite since the index of a critical point is at most the dimension of the manifold.

To generalize it to the case where $\bL$ is equipped with flat bundles, let us review the proof of $(\delta_{\textnormal{pearl}})^2=0$.
The main scheme of the proof is to consider the compactification of one dimensional moduli
space of pearl trajectories, and show that the only non-trivial contribution is $(\delta_{\textnormal{pearl}})^2$, and hence obtaining $(\delta_{\textnormal{pearl}})^2=0$.
Namely, the limit where one of the gradient trajectories $\gamma_j$ ( $ 0< j< k$ ) contracts to a point has a 
canceling partner obtained by a disc-bubbling of a corresponding family of pearl trajectories.

 Denote by $\CM_2 (p,p;\beta)$ the moduli space of pearl trajectories of the type $(\gamma_0, u_1, \gamma_1)$ from a critical point $p$ to itself, where the pearl $u_1$ is a $J$-holomorphic disc of Maslov index two of homotopy class $\beta$.
It is easy to see that the dimension of $\CM(p,p,\beta)$ is one, and the possible degenerations are given as follows.
The gradient trajectory $\gamma_0$ may degenerate to a broken trajectory $\gamma_{01}$ and $\gamma_{02}$ so that
$\gamma_{01}$ is a gradient trajectory connecting $p$ to another critical point $q$ ($\textrm{ind}(q) = \textrm{ind}(p)-1$)
and $(\gamma_{02}, u_1, \gamma_1)$ is an isolated pearl trajectory from $q$ to $p$.
Similarly $\gamma_{1}$ can degenerate to a broken trajectory $\{ \gamma_{11}, \gamma_{12} \}$.  These two types of degenerations give $(\delta_{\textnormal{pearl}})^2$.
\begin{figure}[htb!]
\begin{center}
\includegraphics[height=2.4in]{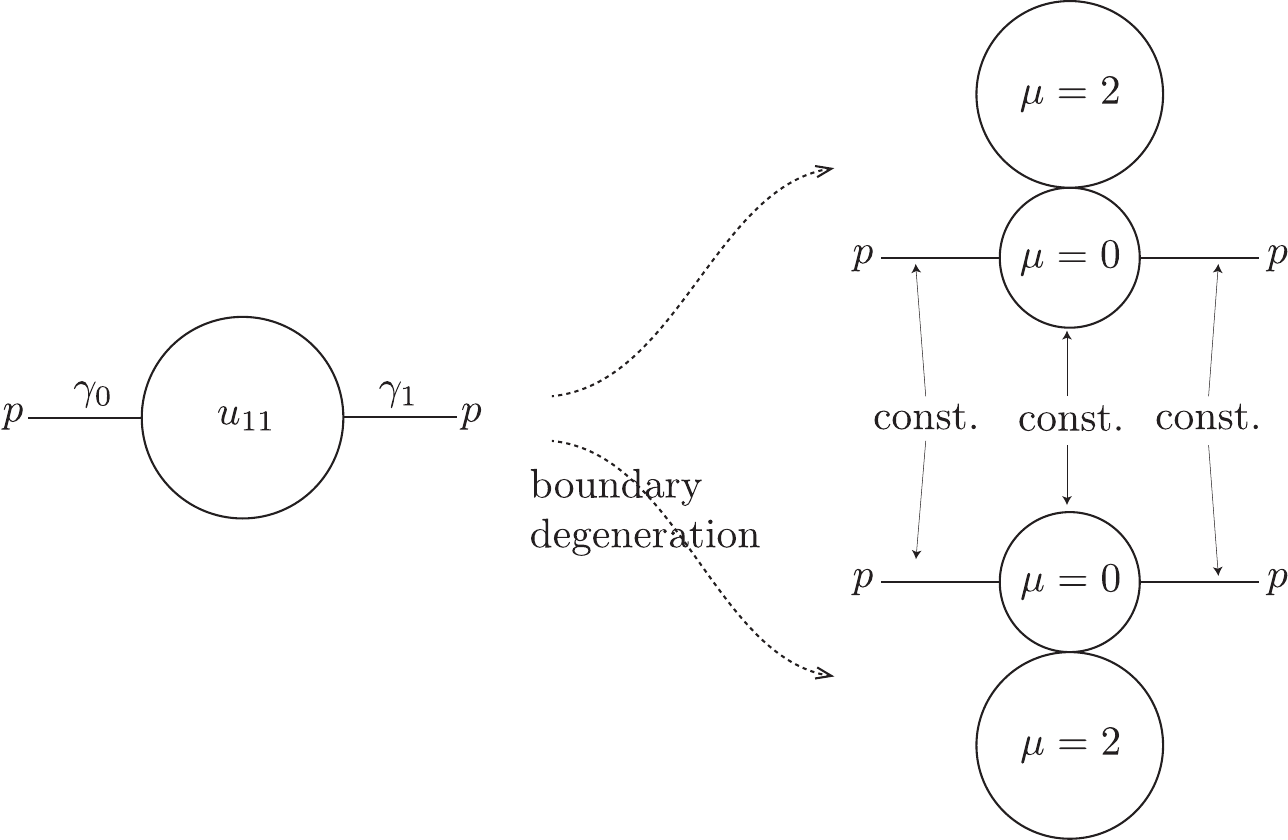}
\caption{Additional types of degenerations}\label{fig:pb}
\end{center}
\end{figure}

There are two additional types of degeneration (see Figure \ref{fig:pb}), which are in fact not found in the discussion of Lemma 5.1.3 \cite{BC} (in the case of \cite{BC} their contributions cancel each other and so do not affect the result).  They play an important role in our story.

Such a degeneration is given by a pearl trajectory $(\gamma_0, u_{11}, \gamma_1)$ with a disc bubble $u_{12}$ attached
either at the upper semi-circle or at the lower semi-circle of $u_{11}$, and $\gamma_0, \gamma_1$ are constant
gradient trajectories attached to the component $u_{11}$ which is a constant disc.  These two contributions cancel each other, and as a result we have $(\delta_{\textnormal{pearl}})^2 = (\Phi(L_0) - \Phi(L_0))= 0.$

Now, we consider the same complex decorated with two flat complex line bundles
$\CL_0, \CL_1$ over $\bL$.
Namely,  we consider a pearl complex for the Floer homology $HF((\bL,\CL_0), (\bL, \CL_1))$.
The pearl chain complex $C^{pearl}(\bL,\CL_0,\CL_1, h)$ is a $\Z/2$-graded free $\Lambda_\C$-module generated by critical points of $h$, where $\Z/2$-grading come from Morse indices of $h$ as before.
The differential of the pearl complex is given by counting pearl trajectories weighted by holonomy and areas.

Given an isolated pearl trajectory $u$ from $p$ to $q$, we have a path $\partial_{0}u$ (resp. $\partial_{1}u$) from $p$ to $q$ obtained by traveling along gradient trajectories and images of upper (resp. lower) semi-circle of $\partial D^2$ of the holomorphic
discs $u_1,\cdots, u_k$.
We denote by $\Pal_{\partial_i u}(\CL_i)$ the holonomy of $\CL_i$ along $\partial_i u$ from $p$ to $q$. 
Then each pearl trajectory from $p$ to $q$ contributes to the differential $\delta^{\CL_0,\CL_1}_{\textnormal{pearl}}$  as 
$$(-1)^{a(u)} \big(\Pal_{\partial_0 u}(\CL_0)\big)^{-1} \Pal_{\partial_1 u}(\CL_1),$$ where $(-1)^{a(u)}$ is the sign for the pearl trajectory
and sum of such contributions together with the
area $T^{\sum_{j=1}^k \omega(u_j)}$ define $n(p,q)$ with holonomy effects for \eqref{eq:pft1}.
Note that  gradient trajectories in a pearl trajectory contribute to both $\Pal_{\partial_i u}(\CL_i)$ for $i=0,1$ with opposite directions, but their holonomies may not cancel out since $\CL_0$ is not equal to $\CL_1$. Here, the holonomy contribution from a gradient flow should be analyzed carefully since it is not clear which part of the flow lies on $(\bL,\mathcal{L}_0)$ or $(\bL,\mathcal{L}_1)$ from the picture of the pearl trajectory itself ((a) of Figure \ref{fig:flowdoubled}). For this, we use the schematic picture of the trajectory drawn as in (b) of Figure \ref{fig:flowdoubled}.

\begin{figure}[htb!]
\begin{center}
\includegraphics[height=1.4in]{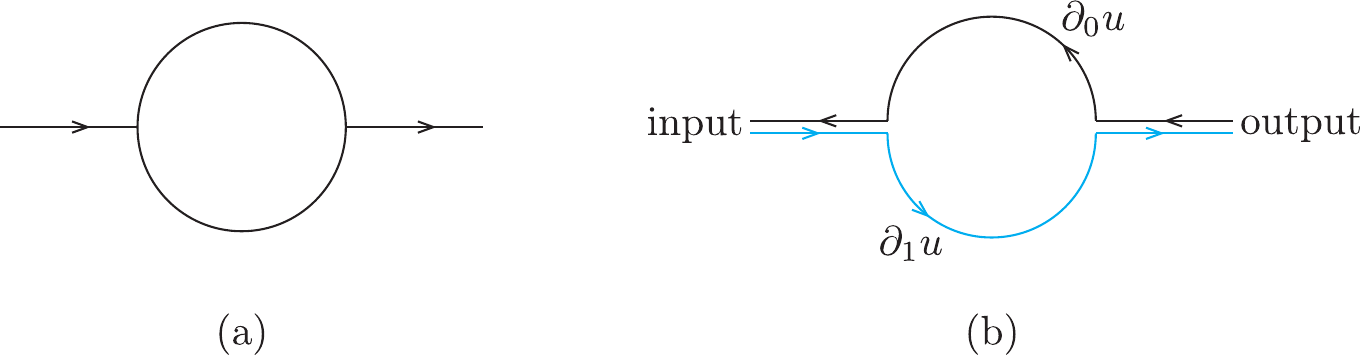}
\caption{Paths $\partial_0{u}$, $\partial_1u$ from a pearl trajectory}\label{fig:flowdoubled}
\end{center}
\end{figure}

As in \eqref{eq2:floer}, we obtain 
\begin{equation}\label{eq:d2wpearl}
(\delta^{\mathcal{L}_0, \mathcal{L}_1}_{\textnormal{pearl}} )^2  =  \big( \Phi(\CL_1)  - \Phi(\CL_0) \big),
\end{equation}
where  the right hand side comes from the two contributions drawn in Figure \ref{fig:pb}. 
The precise sign of the above formula will be proved in Appendix \ref{app:signrule}, Lemma \ref{lem:BCMFsign}.

Now, we vary flat connections on a line bundle to obtain a matrix factorization. i.e. instead of considering $\cL_0$ with a fixed flat connection, we use a family of flat line bundles $\cL_z$ whose holonomies are parametrized by $z$.
As in  \eqref{eq:ddegdelta}, we make the sign change 
$$ d:= (-1)^{{\rm deg}} \delta_{\textnormal{pearl}}$$
to obtain a matrix factorization of $W$:
 $$d^2 = W -\lambda$$
 Note that  \eqref{eq:d2wpearl} implies that with the original $\delta_{\textnormal{pearl}}$ of \cite{BC}, we instead have $(\delta_{\textnormal{pearl}})^2 = \lambda - W$.

 \cite[Proposition 5.6.2]{BC} shows that the homology of the pearl complex $(C^{pearl}(\bL, h), \delta)$ is isomorphic to
the Lagrangian Floer homology $HF(\bL,\phi(\bL))$
for a Hamiltonian isotopy $\phi$.
Such an isomorphism is  constructed by a Lagrangian version of Piuniknin-Salamon-Schwarz morphism
(see for example \cite{Albers}, \cite{KM},\cite{BC},\cite{FOOO}).
Namely, a chain map which induces an isomorphism is given by counting
another version of pearl trajectory:
It is given by a pearl trajectory $\gamma_0, u_1,\gamma_1,u_2,\cdots, u_{k},\gamma_k$
where the last $\gamma_k:[t_{k-1},\infty) \to \bL$ is replaced by $\gamma_k:[t_{k-1},t_k] \to \bL$
with an additional strip component  $u_{k+1}: (\R \times [0,1], \R \times \{0,1\})\to (M,\bL)$ satisfying
$$\partial_s u_{k+1} + J(u_{k+1}) \partial_t u_{k+1}  + \beta(s) \nabla H_t (u_{k+1}) =0$$
where $u(-\infty)$ equals $\gamma_k(t_k)$ and 
and $u(+\infty)$ maps to the Lagrangian intersection point $q \in \bL \cap \phi(\bL)$. (See Figure \ref{fig:pearlcpxtraj}.)
Here, $H_t$ is the Hamiltonian function for $\phi$, and $\beta(s)$ is a cut-off function which vanishes for $s\leq 0$, and
has value $1$ for $s\geq 0$.  Given such a trajectory, say $u$, we can similarly define a path $\partial_0 u$ (resp. $\partial_1 u$) from $p$ to $q$ as before, traveling along gradient trajectories and upper (resp. lower )
 semi-circles and $\R \times \{1\}$ (resp. $\R \times \{0\}$) in the last component.
 
Hence, by incorporating holonomy $\big(\Pal_{\partial_0 u}(\cL_z)\big)^{-1} \Pal_{\partial_1 u}(\CL_1)$ as before,
and proceeding as in Proposition \ref{prop:haminv},  we can prove that the matrix factorization
obtained from the pearl complex $(C^{pearl}(\bL,\cL_z,\CL_1, h), \delta^{\cL_z,\CL_1}_{\textnormal{pearl}})$
is equivalent to the one obtained from the Lagrangian Floer complex
 $(CF((\bL,\cL_z), (\phi(\bL), (\phi^{-1})^*\CL_1)),  \delta^{\cL_z,\CL_1}_{\textnormal{pearl}})$.

%

\subsection{The projective plane} \label{P2}
As an example, we employ a pearl complex to compute the matrix factorization mirror
to the Clifford torus of the projective plane. Let 
$\bL = \{[u_0,u_1,u_2]  \in \bP^2 \mid |u_0|=|u_1|=|u_2| \}$ be the Clifford torus in $\bP^2$.
From \cite{C0}, there are only three holomorphic discs with boundary on $\bL$ (up to $Aut(D^2)$ and $T^2$-action)
whose Maslov index is 
$2$, and these are given by 
$$D_0:=[z,1,1],\quad D_1:=[1,z,1], \quad D_2:=[1,1,z],$$ for $z \in D^2$. 
Denoting their common symplectic areas as $k$,   the Floer potential $W(\bL)$ is 
 $$W(\bL) = T^k \left(z_1 + z_2+ \frac{1}{z_1z_2} \right).$$

 We choose a Morse function $f : \bL \to \R$ such that critical points and gradient flow lines of $f$ are as shown in Figure \ref{CP2pert}. 
Such a Morse function can be chosen as follows. 
Since $\bL$ is a torus, we identify $\bL$ with $(\R/\Z)^n$, and we choose $f(x_1,\cdots, x_n) =  \sum_{i=1}^n (\sin (\pi x_i))^2$,
and compose $f$ with a diffeomorphism $h$ of $\bL$ so that the gradient flow lines of $f \circ h$ are as in Figure \ref{CP2pert}.
We use a diffeomorphism $h$ so that a Maslov index-2 disc do not meet two critical points of index difference $2$.

\begin{figure}[htb!]
\begin{center}
\includegraphics[height=1.7in]{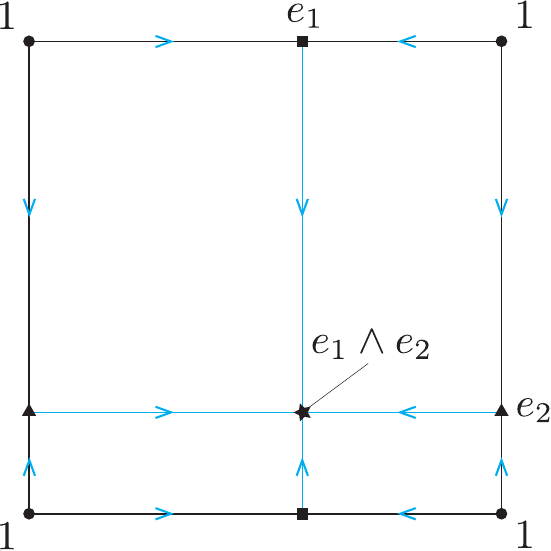}
\caption{Morse function on $\bL$}\label{CP2pert}
\end{center}
\end{figure}

It is convenient to identify the vector space
generated by critical points with the exterior algebra with two generators $e_1,e_2$ so that
four generators $1, e_1, e_2, e_{12}$ correspond to the critical points where $1$ is  the maximum and $e_{12}$ is the minimum of $f$ as in Figure \ref{CP2pert}. Here, $e_{12} = e_1 \wedge e_2$.

Now we choose gauge hypertori as in Figure \ref{hypertoriCP2}. 
Namely, for $\cL_z$, we choose hyper-tori $\{ H_i^0\}$ as
$$H_1^0 =H_1 + (0,p_2),\quad H_2^0= H_2 + (p_1,0),$$
for sufficiently small $p_1, p_2$ where $H_1, H_2$ is defined in Definition \ref{def:hi}.

\begin{figure}[htb!]
\begin{center}
\includegraphics[height=2.3in]{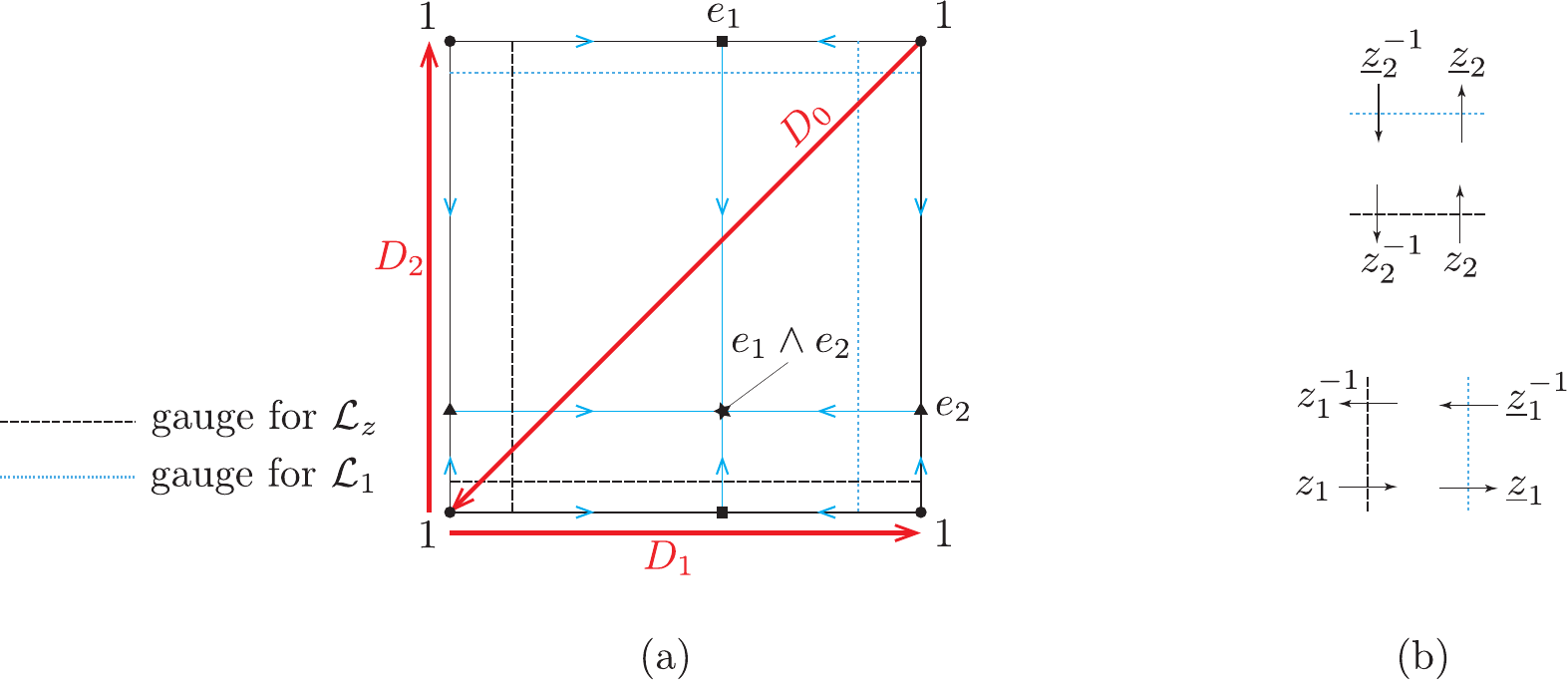}
\caption{The choice of gauge hypertori}\label{hypertoriCP2}
\end{center}
\end{figure}
We also choose a gauge hypertori for $\CL_1$ (over $\bL$) as
$$H_1^1 =H_1 + (0,1-p_2),\quad H_2^1= H_2 + (1-p_1,0).$$
Thus, $\CL_1$ is a flat complex line bundle with fixed holonomy whose connection is trivial away from gauge hypertori $\{ H_i^1\}$. 
Note that critical points of $f$ are away from gauge hypertori.

Given an isolated pearl trajectory $u$ between two critical points of $h$, we will compute the signed intersection number, say $m_1, m_2$ of a path $\partial_0 u$ with gauge hypertori
$H_1^0, H_2^0$ respectively, and the corresponding holonomy factor will be given by $z_1^{m_1} z_2^{m_2}$ for
the mirror variables $z_1,z_2$.
Also from $u$, we compute the signed  intersection number, say $m_1', m_2'$ of a path $\partial_1 u$ with
gauge hypertori $H_1^1, H_2^1$ respectively, and the corresponding holonomy factor is given by $\uz_1^{m_1'} \uz_2^{m_2'}$,
where $\uz_1, \uz_2 \in \C\setminus \{0\}$ are fixed complex numbers. Hence the total contribution
of $u$ to the differential $\delta^{\cL_z, \CL_1}$ is given as (up to sign) 
$$z_1^{-m_1} z_2^{-m_2}\uz_1^{m_1'} \uz_2^{m_2'}.$$

Let us first consider Morse differentials contained in $\delta^{\cL_z, \CL_1}_{\textnormal{pearl}}$.
For each pair of critical points of index difference one, we have two Morse trajectories of opposite directions.
From our choice of gauge hypertori, one can check that 
$$\delta^{\cL_z, \CL_1}_{\textnormal{pearl}} (1) = (z_1 - \uz_1) e_1 + (z_2 - \uz_2) e_2,  \quad\delta^{\cL_z, \CL_1}_{\textnormal{pearl}} (e_1\wedge e_2)=0,$$
$$\delta^{\cL_z, \CL_1}_{\textnormal{pearl}} (e_1) =  (z_2 - \uz_2) e_1 \wedge e_2, \quad\delta^{\cL_z, \CL_1}_{\textnormal{pearl}} (e_2) = -(z_1 - \uz_1) e_1 \wedge e_2.$$
Hence $(-1)^{\textrm{deg}}\delta^{\cL_z, \CL_1}_{\textnormal{pearl}} (\cdot) $ may be written as 
$$(z_1 - \uz_1) e_1 \wedge (\cdot) + (z_2 - \uz_2) e_2 \wedge (\cdot).$$
The precise sign will be discussed in Section \ref{subsec:signpearld1}.

In what follows, we compute the contribution from a pearl trajectory with a single holomorphic disc of Maslov index two,
which will be called  a {\em single pearl (trajectory)} for short.

\subsubsection{$D_1$ disc}
There are two   single $D_1$-pearl trajectories, one from $e_{12}$ to $e_{2}$ and one from $e_1$ to $1$.
Denote by $u = (\gamma_0, D_1, \gamma_1)$ the pearl trajectory from $e_1$ to $1$. In this case, $\gamma_0, \gamma_1$
are constant trajectories (see Remark \ref{rem:pearl1}). Then, paths $\partial_0 u, \partial_1 u$ in $\bL$ from $e_1$ to $1$
are given by part of the boundary of $D_1$ disc.

As illustrated in Figure \ref{thickCP2e1} (a), $\partial_0 u$ does not intersect gauge hypertori $H^0$ (but only intersect $H^1$)
and $\partial_1 u$ does not intersect gauge hypertori $H^1$ (but only intersect $H^0$). Hence the holonomy contribution
for $u$ is trivial.

The same argument works for a pearl trajectory from $e_{12}$ to $e_2$. Hence we may write the $D_1$ disc contributions
as $ T^k \iota_{e_1}$, where $\iota_{e_1}$ is the contraction which sends $e_1$ to $1$ and $e_{12}$ to $e_2$.

\begin{figure}[htb!]
\begin{center}
\includegraphics[height=3.2in]{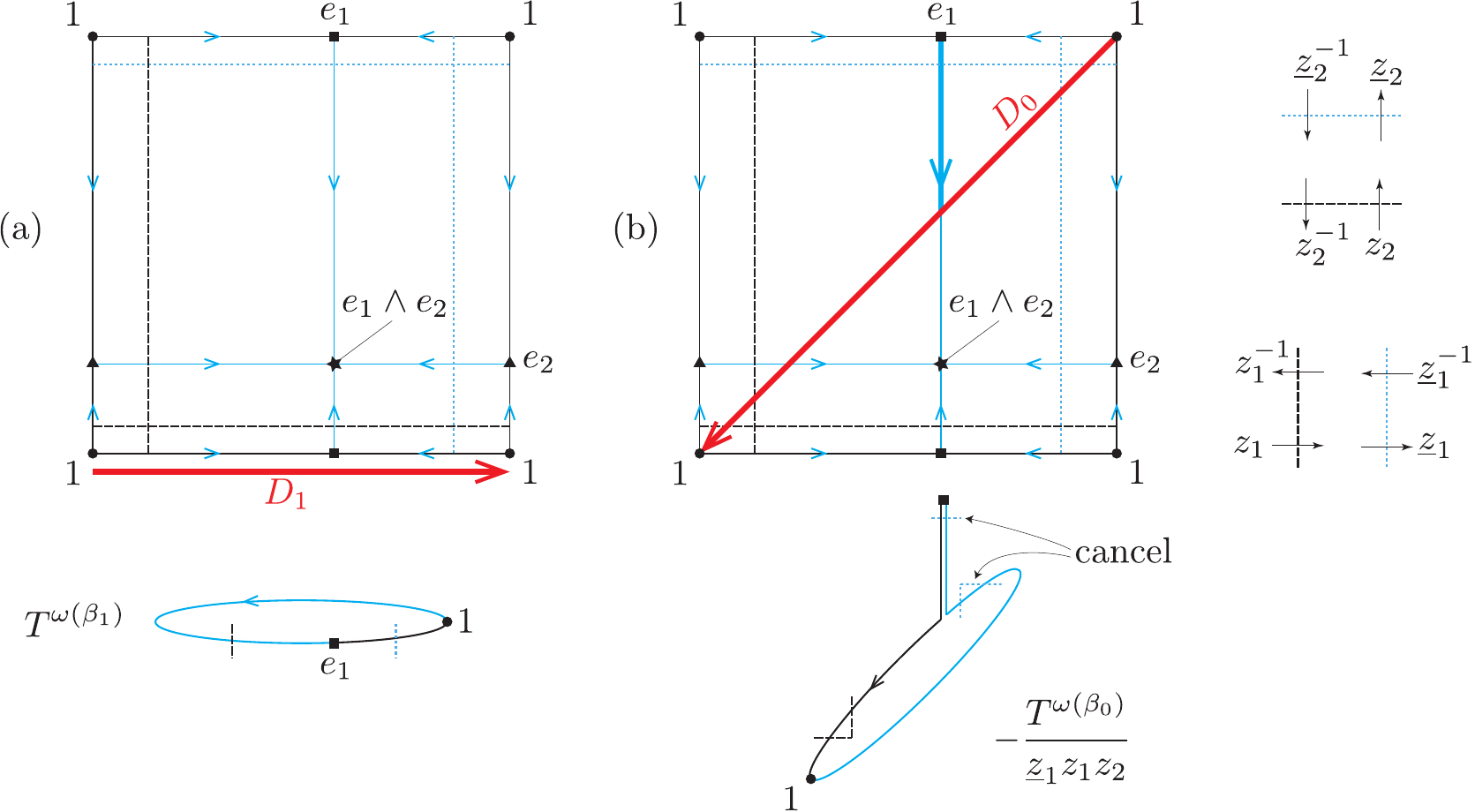}
\caption{Thin strips from $e_1$ to $1$}\label{thickCP2e1}
\end{center}
\end{figure}

\begin{remark}\label{rem:pearl1}
 In fact, in the construction in \cite{BC} using generic $f$ and $J$,  ``constant(flow)-disc-constant(flow)" configuration does not appear as
 discs do not meet two critical point at once generically. But we
 show in Lemma \ref{lem:trans} that this configuration is also transversal in toric cases, and justifies the use of the standard $J_0$, and our choice of the Morse function. Alternatively, one may choose a different Morse function corresponding to another $\Z$-basis of $\R^2$, which do not contain any normal vector to the facets of moment polytope, then such configuration will not appear also.
\end{remark}

\subsubsection{$D_2$ disc}
There are two single $D_2$-pearl trajectories, one from $e_{12}$ to $e_1$ and one from $e_2$ to $1$,
whose contributions are given by $T^k \iota_{e_2}$ as in $D_1$-case.

\subsubsection{$D_0$ disc}
There are four single $D_0$-pearl trajectories,  two from $e_{12}$ to $e_1$ and $e_2$, two
from $e_1,e_2$ to $1$.
Denote by $u = (\gamma_0, D_0, \gamma_1)$ the pearl trajectory from $e_1$ to $1$, illustrated in Figure \ref{thickCP2e1} (b).
In this case, $\gamma_0$ is a non-trivial gradient trajectory from $e_1$ to $D_0$ disc, and $\gamma_1$ is a constant trajectory.
Now, $\partial_0 u$ passes through both $H_1^0, H_2^0$  (contributing $z_1z_2$) and $\partial_1 u$ passes through $H_1^1$ (contributing $\uz_1$) once and $H_2^1$ twice (but
with opposite orientations contributing $1$), and hence the total contribution from $e_1$ to $1$ is 
$$- \frac{T^k}{\uz_1 z_1 z_2}.$$
How to obtain the precise sign will be discussed in Section \ref{subsec:signpearld2}.
One can check that  $D_0$ pearl trajectory from $e_{12}$ to $e_2$ has the same holonomy contribution as above
up to sign.
Hence, we may write them as $- \frac{T^k}{\uz_1 z_1 z_2} \iota_{e_1}$.

Similarly the pearl trajectory from $e_2$ to $1$ is illustrated in Figure \ref{thickCP2e2} (b),
whose contribution is 
$$-\frac{T^k}{\uz_1 \uz_1 z_2},$$
and the same goes for the trajectory from $e_{12}$ to $e_1$.
Hence we may write them as 
$$-\frac{T^k}{\uz_1 \uz_1 z_2} \iota_{e_2}.$$

\begin{figure}[h]
\begin{center}
\includegraphics[height=2.5in]{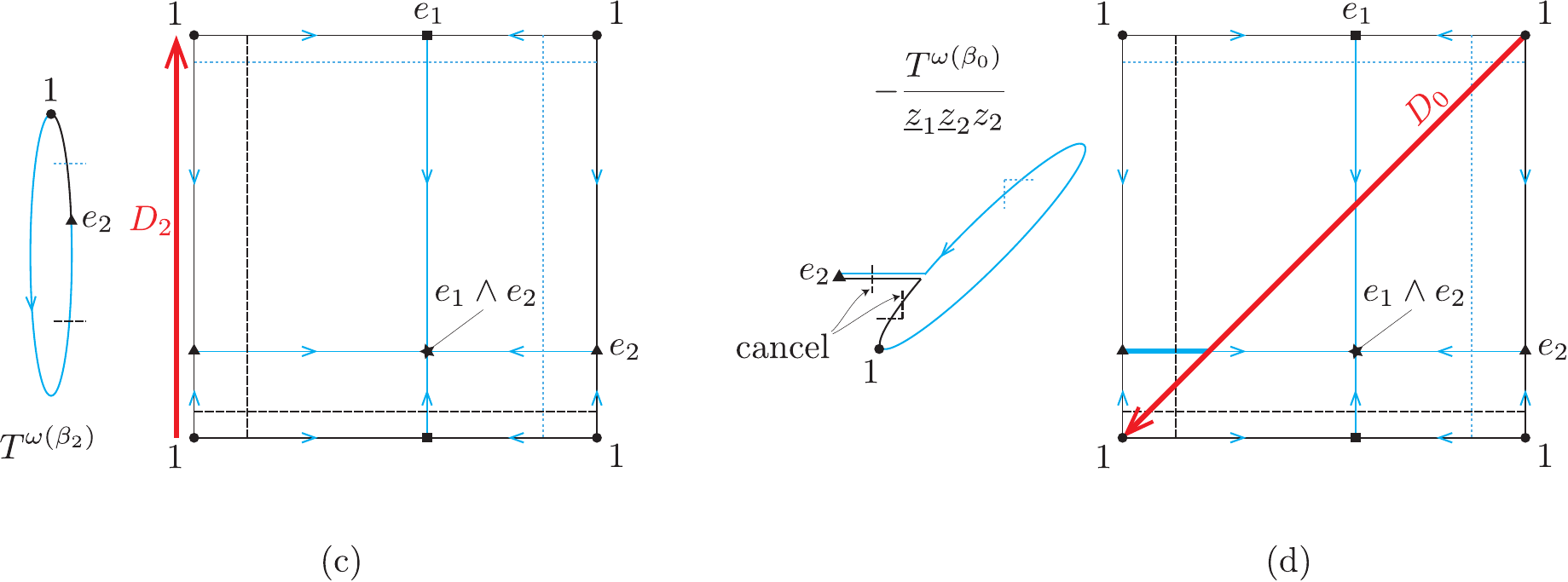}
\caption{Thin strips from $e_2$ to $1$}\label{thickCP2e2}
\end{center}
\end{figure}

Therefore, the pearl differential $(-1)^{\textrm{deg}}\delta^{\cL_z, \CL_1}_{\textnormal{pearl}}$ gives the mirror matrix factorization 
which can be written as
\begin{equation}\label{eq:MFCP2wc}
(z_1 - \uz_1) e_1 \wedge + (z_2 - \uz_2) e_2 \wedge + \left( T^k-\frac{T^k}{\uz_1 z_1 z_2} \right) \iota_{e_1} + \left(T^k - \frac{T^k}{\uz_1 \uz_2 z_2}  \right) \iota_{e_2}.
\end{equation}
The square of the above becomes
$$T^k \left(z_1 +  z_2 + \frac{1}{z_1 z_2} \right) - T^k\left( \uz_1 +  \uz_2 + \frac{1}{\uz_1 \uz_2} \right) = W(z) - W(\uz).$$
By writing the above in a matrix form, we obtain the following
\begin{prop}
The matrix factorization mirror to the Clifford torus $T^2$ with holonomy $(\uz_1,\uz_2)$
is given by
\begin{equation}\label{eq:MFCP2tor3}
\bordermatrix{   & 1 &  e_{12} &    e_1 & e_2 \cr
  1 &   	    0      &          0        & T^k-\frac{T^k}{\uz_1 z_1 z_2}  &  T^k - \frac{T^k}{\uz_1 \uz_2 z_2} \cr
  e_{12} &          0       &           0        &   -(z_2 - \uz_2) &  (z_1 - \uz_1) \cr
 e_1    &   z_1 - \uz_1    &  - T^k + \frac{T^k}{\uz_1 \uz_2 z_2}  &  0 & 0   \cr
  e_2 &  z_2 - \uz_2       &  T^k-\frac{T^k}{\uz_1 z_1 z_2}  & 0 &  0 \cr }.
 \end{equation} 
\end{prop}

In \cite{CL}, Chan-Leung computed the matrix factorization mirror to the Clifford torus in $\bP^2$ from a sketch of arguments based on SYZ.  Let us show that by setting $(\uz_1,\uz_2) = (1,1)$, the above matrix factorization agrees with that in \cite{CL} up to a change of coordinates.

First, the Givental-Hori-Vafa superpotential of $\C P^2$ is
$$W(z') =z_1' + z_2' + \frac{q}{z_1' z_2'}$$
To identify $W(z')$ and $W(z)$, we make the following change of variables
$$ q=T^{3k}, \quad z_1' = q^{1/3} z_1, \quad z_2' = q^{1/3} z_2.$$
As we consider the case $(\uz_1,\uz_2) = (1,1)$, we set $\uz_1' = \uz_2' = q^{1/3}$.

We take the basis of  $C^{pearl} (\bL,f)$) in the following way:
$$\mbox{odd generators}\,:\,\, p_1 := e_2,\quad p_2 := - z_1' e_1,$$
$$\mbox{even generators}\,:\,\, q_1 := q^{1/3} ,\quad q_2 := q^{-1/3}z_1' e_{12}.$$
Then, the above matrix factorization \eqref{eq:MFCP2tor3} can be written as
\begin{equation}\label{eq:MFCP2tor2}
\bordermatrix{   & p_1 & q^{1/3} p_2 &  q^{1/3} q_1 & q_2  \cr
  p_1 &   	     0     &     0             &   z_2' - q^{1/3}  &    z_1' - \frac{q^{2/3}}{ z_2' }  \cr
 q^{1/3} p_2 &            0     &          0         & -\left(1 - \frac{q^{1/3}}{z_1'} \right)  &  1 - \frac{q^{1/3}}{z_2'}   \cr
q^{1/3} q_1 &    1 - \frac{q^{1/3}}{ z_2'}     &  -\left(  z_1' - \frac{q^{2/3}}{ z_2' } \right)  & 0  &   0 \cr
 q_2 &  1- \frac{q^{1/3}}{z_1'}         &   z_2' - q^{1/3}&  0& 0  \cr }.
 \end{equation}
After switching $z_1'$ and $z_2'$, \eqref{eq:MFCP2tor2} is precisely the matrix factorization appearing in \cite{CL}, 

%
%
%

\section{Toric Fano manifolds}\label{sec:toricFano}
In this section, we compute the mirror matrix factorizations of Lagrangian torus fibers of toric Fano manifolds.
We shall use pearl complexes discussed in the last section.  We expect that the same method would work for semi-Fano toric manifolds when we incorporate virtual perturbation techniques to deal with sphere bubbles of indices zero.

We take $\bL$ to be a Lagrangian torus fiber (with non-trivial Floer cohomology) in a toric Fano manifold $(X,\omega)$ and define the mirror potential $W(\bL)$ (also denoted as $W_z$) using family of flat line bundles $\cL_z$.
A weakly unobstructed Lagrangian $L$ corresponds to a matrix factorization of $W(\bL)-c$ via the mirror functor (where $m_0^L = c \, \one_L$).
We are particularly interested in the mirror matrix factorization of a Lagrangian torus fiber $\uL$ equipped with a flat line bundle $\cL_{\uz}$ (where $\uz \in (\C^\times)^n$ is fixed).

For $\uL \neq \bL$,  $\bL \cap \uL = \emptyset$ and hence the corresponding matrix factorization is trivial.  Thus we only consider the case when $\uL$ and $\bL$ are the same Lagrangian torus fiber (equipped with possibly different flat line bundles).  We will use a pearl complex as in the example given in Section \ref{P2}.  The actual computations require much more effort though.

Recall that $d$ in our case is defined by $(-1)^{\rm deg} \delta_{\textnormal{pearl}}$ (we omit upper indices indicating line bundles in $\delta_{\textnormal{pearl}}^{\cL_z, \cL_\uz}$ for simplicity). The main result is the following. The matrix factorization mirror to $(\uL,\cL_\uz)$ takes the form $(\largewedge^* \underline{\Lambda^n}, d)$ ($\Lambda$ is the Novikov ring),
\begin{equation}\label{eq:decompd1}
d = d_{1} + d_{-1} + \cdots + d_{-(2\lfloor(n+1)/2\rfloor-1)}
\end{equation}
where an element in $\largewedge^l \underline{\Lambda^n}$ has degree $l$ (or Morse index $n-l$), and $d_{-k}$ sends $\largewedge^l \underline{\Lambda^n} \to \largewedge^{l-k} \underline{\Lambda^n}$.  (The decomposition of $d$ directly comes from that of $\delta_{\textnormal{pearl}}$ in Equation \eqref{eq:decompdpearl}.)  By definition $d^2 = W_z - W_\uz$.  While it is difficult to write down all the pearl trajectories contributing to $d$, we deduce an explicit formula for the following `approximation' of $d$ (Theorem \ref{thm:explicit}):
$$\tilde{d} := d_{1} + d_{-1}.$$ 
We prove that $\tilde{d}$ itself is a matrix factorization of $W_z - W_\uz$ (Theorem \ref{thm:square=W}) which is of wedge-contraction type whose definition is given below \eqref{eq:dwcDYC}. 

Even though the explicit expression for $d$ is unknown, we can prove that $(\largewedge^* \underline{\Lambda^n}, d)$ generates the category of matrix factorizations of $W_z - W_{\uz}$, see Theorem \ref{thm:gen}.  It uses the method of spectral sequence by Polishchuk-Vaintrob \cite{PV}.  When $n\leq 3$, $d$ simply equals to $\tilde{d}$.  In general, we expect that $d$ and $\tilde{d}$ are equivalent by some quantum change of coordinates of $\largewedge^* \underline{\Lambda^n}$.  We deduce such a change of coordinate for $n=4$ in the end of Section \ref{sec:app}.

Let us recall the definition of wedge-contraction type matrix factorizations. Let $R$ be the formal power series ring on $n$ variables $x_1, \cdots, x_n$ and $W$ an element of $R$. Suppose that the origin is a unique critical point of $W$ in $W^{-1} (0)$, and $W$ can be written as $W = \sum_{i=1}^n x_i w_i$ for some series $w_i$ in $x_1, \cdots, x_n$. Consider the exterior algebra generated by $e_1, \cdots, e_n$ over $R$, which has an obvious $\Z/2$-grading. Then a wedge-contraction of matrix factorization $d_{wc}$ is defined by
\begin{equation}\label{eq:dwcDYC}
d_{wc} = \sum_i x_i e_i \wedge (-) + \sum_i w_i \iota_{e_i}.
\end{equation}
 Dyckerhoff has shown in \cite[Theorem 4.1]{Dyc} that
if $W^{-1} (0)$ has a unique singularity at the origin, then $(M,d_{wc})$ is a generator of $D^{\pi} \mathcal{MF} (W)$.


We will first recall the Floer potentials of toric Fano manifolds.  Then we deduce regularity of pearl trajectories and compute the mirror matrix factorizations.

\subsection{Localized Floer potential in the toric Fano cases}\label{subsec:toric}
Lagrangian Floer theory has been actively developed in the last decade, and Floer cohomology of
Lagrangian torus fibers has been computed from  the classification of all holomorphic discs with boundary on Lagrangian torus fibers (\cite{C},\cite{CO}), and in much more generality by Fukaya-Oh-Ohta-Ono  \cite{FOOOT}, \cite{FOOOT2} 
by introducing (bulk) deformation theories and  $T^n$-equivariant perturbations on the moduli space of holomorphic discs.

Let us first recall the  Floer potential $W(X)$ for toric Fano manifold $X$ introduced by Cho-Oh \cite{CO}
(which were generalized significantly in  \cite{FOOOT} and also in \cite{auroux07}, \cite{CLLT12} based on Strominger-Yau-Zaslow methods to understand mirror symmetry). In the Fano case, $W(X)$ can be identified with the Givental-Hori-Vafa mirror Landau-Ginzburg potential. And we compare it with the Floer potential  $W(\bL)$ for a Lagrangian torus fiber $\bL$ given in Definition \ref{def:W}.  Here, $W(X)$ depends on $X$ only, but $W(\bL)$ depends on the particular Lagrangian torus fiber $\bL$ in $X$ as well as $X$ itself.

Let $(X, \omega)$ be a $n$-dimensional toric Fano manifold with a moment polytope $P$, defined in $M_\R$ by 
the set of inequalities
$$ \langle u, v_i \rangle \geq \lambda_i \;\; \textrm{for} \; i =1,\cdots, m$$
for $u \in M_\R$ and inner normal vectors $v_i \in N$ to facets of $P$.
For each $u$ in the interior of $P$, the inverse image of the moment map  $\mu^{-1}(u)$ gives a Lagrangian torus $\bL = L(u)$, which satisfies Assumption \ref{assum1} (\cite{CO}).

For each normal vector $v_i$  of the moment polytope $P$, there exists a unique  holomorphic disc passing through a generic point $p$ (up to $Aut(D^2)$), whose homotopy class is denoted as $\beta_i \in \pi_2(X, L(u))$. Hence the number $n_{\beta_i}$ of such discs is one, and its symplectic area is
given by  $2\pi ( \langle u, v_i \rangle - \lambda_i)$.
 Let us also denote  $\partial \beta_i = \sum_{j=1}^n v_{ij} e_j$ for a basis $\{e_i\}_{i=1}^n$ of $H_1(\bL;\Z)$.
Consider holonomy parameters $\nu = (\nu_1,\cdots, \nu_n) \in \R^n$ which
is used to consider flat unitary line bundle $\CL$ over $L(u)$ with holonomy $exp(2 \pi \sqrt{-1} \nu_i)$ along $e_i$
(see Section 12 \cite{CO} for more details).

$$ W(X)= \sum_{\beta, \mu(\beta)=2} n_\beta \exp \left(- \frac{1}{2\pi}\int_\beta \omega \right) \Hol_{\CL} (\partial \beta) $$  
$$ = \sum_{i=1}^m  e^{- ( \langle u, v_i \rangle - \lambda_i)} \exp (2\pi \sqrt{-1}  \langle  \nu, v_i \rangle)$$

Hence, it is natural to introduce  mirror variables depending on the positions $u_i = \langle u, v_i \rangle$ and the holonomies $\nu_i = \langle \nu, v_i \rangle$ 
$$t_i = e^{- u_i + 2 \pi \sqrt{-1} \nu_i} \textrm{ for } i=1,\ldots,n.$$
If we denote $$t^{v_i} =  t_1^{v_{i1}} t_2^{v_{i2}} \cdots t_n^{v_{in}},$$ then $W(X)$ can be written in terms of $(t_1,\cdots, t_n)$ as
\begin{equation}\label{eq:WX1}
W(X) = \sum_{i=1}^m e^{\lambda_i} t^{v_i}
\end{equation}

Now, let us compare $W(X)$ and the Floer potential $W(\bL)$.
For this, we choose $\bL = L(u_0)$ for a fixed $u_0$ in the interior of the polytope $P$.
Recall that in our setting, mirror variable is given by the holonomy  (see \eqref{eq:defrho})
$$z_i = exp(2\pi \sqrt{-1} \nu_i). $$
By setting $z^{v_i} =  z_1^{v_{i1}} z_2^{v_{i2}} \cdots z_n^{v_{in}},$
and from the Definition \ref{def:W}, Floer potential is 
\begin{equation}\label{eq:WL1}
W(\bL) = \sum_{i=1}^m  T^{\omega(\beta)} \rho^b(\partial \beta) = 
\sum_{i=1}^m c_i z^{v_i}
\end{equation}
where $c_i :=T^{2\pi (\langle u, v_i \rangle - \lambda_i)}$.

\begin{remark}
The Floer potential in semi-Fano case is given by
$$ W(\bL) = \sum_{i=1}^m \left(\sum_\alpha n_{\beta_i+\alpha} T^{\int_{\beta_i+\alpha} \omega} \right)\Hol_\cL (\partial \beta_i) = \sum_{i=1}^m c_i z^{v_i} $$
where $c_i = \sum_\alpha n_{\beta_i+\alpha} T^{\int_{\beta_i+\alpha} \omega}$ and $z^{v_i} = \prod_{j=1}^n z_j^{v_{i,j}} = \Hol_\cL (\partial \beta_i)$.  The sum is over all $\alpha \in H_2^\eff(X)$ with $c_1(\alpha)=0$, and $n_{\beta_i}=1$.
\end{remark}

Comparing \eqref{eq:WX1} and \eqref{eq:WL1}, we obtain the following lemma which is well-known from \cite{FOOOT}.
\begin{lemma}\label{lem:toriclocid}
For toric Fano manifolds,  the substitution $t_i = z_i \cdot e^{-\langle u,v_i\rangle}$ and $T^{2\pi} = e^{-1}$ gives 
$$W(X) = W(\bL).$$
\end{lemma}


\subsection{The matrix factorizations}\label{subsec:MFtoricfano1}
Let $X$ be a compact toric Fano manifold of dimension $n$  defined by a fan supported in $N_\R$ and fix a reference torus fiber $\bL$. In this section we compute the mirror matrix factorization of $\uL$ induced by the mirror functor given in Section \ref{sec:3} using pearl complexes (Section \ref{sec:pearl}), where $\uL$ is a torus fiber together with a fixed flat line bundle $\cL_\uz$ with a $\C^\times$ connection $\nabla_\uz$. Only when $\uL$ and $\bL$ are fiber over the same point, the resulting matrix factorization is non-trivial (or otherwise $\uL \cap \bL = \emptyset$), and so from now on we assume this is the case.  Thus $\conn_\uz$ belongs to the mirror space $M_{\C^\times} \cong (\C^\times)^n$,  and corresponds to a certain value $\uz\in (\C^\times)^n$.  We have $m_0^{(\uL,\conn_\uz)} = W_\uz \cdot \one_{\uL}$. 

Let $m$ be the number of rays in the fan. Without loss of generality we assume $E_i := v_i = \partial \beta_i$ for $i=1,\ldots,n$ form a basis of $N$, and this gives a coordinate system $z_i := \Hol_\cL (\partial \beta_i), i=1,\ldots,n$ on $M_{\C^\times} \cong (\C^\times)^n$.  As before, we write $v_i = \partial \beta_i = \sum_{j=1}^n v_{i,j} E_j$ for any $i=1,\ldots,m$ in terms of this basis.  Then $v_{i,j} = \delta_{ij}$ when $i=1,\ldots,n$.  We set $s_{i,j}$ to be the sign of $v_{i,j}$.

We use the basis $E_1,\ldots,E_n$ to identify $\bL$ with the standard torus $\R^n/\Z^n$.  Choose $(a_1,\cdots,a_n) \in (0,1)^n$ satisfying the following condition:
\begin{equation}\label{cond:sij}
\begin{array}{l}
 - \quad a_j /a_k \,\, \mbox{is irrational for all} \,\, j \neq k.\\
 - \quad |v_{i,k}| a_j - |v_{i,j}| a_k > 0 \,\, \mbox{for every} \,\, i=1,\ldots,m \,\, \mbox{and}\,\, j<k\,\, \mbox{with}\,\, s_{i,j}=s_{i,k} \\
 - \quad a_j\,\, \mbox{is sufficiently close to} \,\,0 \,\, \mbox{for each} \,\, j
\end{array}
\end{equation}
  Then take a Morse function $f_{(a_1,\ldots,a_n)}$ whose critical points are in one-to-one correspondence with subsets $I \subset \{1,\ldots,n\}$.  The critical points are denoted by $e_I$ for $I \subset \{1,\ldots,n\}$, and
$f_{(a_1,\ldots,a_n)}$ is taken such that $e_I$ has coordinates $(c_1,\ldots,c_n)$ where $c_i = 0$ when $i\not\in I$ and $c_i = a_i$ when $i \in I$.
We will write $e_i:=e_{\{i\}}$, $e_0:=e_{\phi}$, $e_{top}:= e_{\{1, \cdots, n\}}$ for notational convenience. Also for $I=\{ i_1, \cdots, i_k\}$, we identify $e_I$ with $e_{i_1} \wedge \cdots \wedge e_{i_k}$ with $i_1 < \cdots < i_k$. Using this terminology, we can define various $\Lambda$-linear endomorphisms on $\largewedge^* \underline{\Lambda^n}$ such as
$$e_j \wedge e_I = (-1)^{i_l} e_{I'}, \quad \iota_{e_{i_l}} e_I = (-1)^{l-1} e_{I''}$$
where $I= \{i_1 < \cdots < i_k\}$, $I' = \{j\} \cup I = \{ i_1 < \cdots < i_l < j < i_{l+1} < \cdots < i_k\}$ and $I'' = \{ i_1 < \cdots <i_{l-1} < i_{l+1} <\cdots < i_k\}$. 


The flat connections on $\cL_z$ and $\cL_\uz$ are specified by the values $z$ and $\uz$ and the gauge hypertori $H_i + p$ and $H_i + \up$ in $\bL$, respectively.  The points $p = ([p_1],\ldots,[p_n]), \up = ([\up_1],\ldots,[\up_n]) \in \R^n/\Z^n$ are taken such that $0<p_1,\ldots,p_n <1$, $p_1,\ldots,p_n \ll 1$, $0<\up_1,\ldots,\up_n <1$ and $1-p_1,\ldots,1-p_n \ll 1$.  These choices of gauge of the flat connections are used to fix the mirror matrix factorization.  Certainly we can take other choices, and we will get another matrix factorization which is equivalent to the original one by Lemma \ref{lem:gimf}.


The matrix factorization $(\largewedge^* \underline{\Lambda^n}, d)$ transformed from $(\bL,\conn_\uz)$ is defined by $d:=(-1)^{\rm deg} \delta_{\textnormal{pearl}}$ where $\delta_{\textnormal{pearl}}$ counts pearl trajectories connecting every pair of critical points $e_I, e_J \in \bL$ (weighted by area and holonomy).  One has $d^2 = (W_z - W_\uz) \cdot \Id$.  Recall that
$$\delta_{\textnormal{pearl}} = (\delta_{\textnormal{pearl}})_{1} + (\delta_{\textnormal{pearl}})_{-1} + \cdots + (\delta_{\textnormal{pearl}})_{-(2\lfloor(n+1)/2\rfloor-1)}$$
where $(\delta_{\textnormal{pearl}})_{-(2k-1)}$ takes the form
$$ (\delta_{\textnormal{pearl}})_{-(2k-1)} : e_I \mapsto \sum_{\substack{|J| = |I| - (2k-1) \\\Gamma(e_I \to e_J)}} \left( \textrm{sign}(\Gamma(e_I \to e_J)) \,\, T^{-\frac{1}{2\pi} \sum_i \int_{D_i} \omega} \,\, \Hol(\Gamma(e_I \to e_J)) \right) e_J.$$
Here, $\Gamma(e_I \to e_J)$ is a pearl trajectory from $e_I$ to $e_J$ with holomorphic disc components $D_i$ with total Maslov index $\sum_i \mu(D_i) = 2k$, and $\Hol$ denotes the holonomy.  

\begin{remark}
Note that the above expression for $(\delta_{\textnormal{pearl}})_{-(2k-1)}$ is a finite sum, since for a toric Fano manifold a non-constant holomorphic disc class has at least Maslov index two.  Thus we can substitute the Novikov formal variable $T$ by $e^{-1}$ and work over complex numbers.
\end{remark}

We shall deduce an explicit formula for $\tilde{d} := (-1)^{\rm deg} (\delta_{\textnormal{pearl}})_{1} + (-1)^{\rm deg} (\delta_{\textnormal{pearl}})_{-1}$ and prove that $\tilde{d}$ itself is a matrix factorization: $\tilde{d}^2 = (W_z - W_\uz) \cdot \Id$.
First we introduce the following terminology for later convenience.

\begin{definition}
Consider a pearl trajectory from a critical point $p$ to another critical point $q$ which only consists of two flow-line components and a disc component $D$.   The point in $\partial D$ that the flow line from $p$ is glued to is called to be an \emph{entry point}, and the point in $\partial D$ that is glued to the flow line to $q$ is called to be an \emph{exit point}.
\end{definition}

\begin{lemma}\label{lem:pearld1}
\begin{equation}\label{eq:d1pearltoric}
d_{1}=(-1)^{\deg} (\delta_{\textnormal{pearl}})_{1} = \sum_{i=1}^m (z_i - \uz_i) e_i \wedge (\cdot).
\end{equation}
\end{lemma}
\begin{proof}
Given a critical point $e_I$, $(\delta_{\textnormal{pearl}})_1 \cdot e_I$ is a linear combination of $e_{I\cup \{j\}}$'s for $j \not\in I$ whose coefficients count flow lines from $e_I$ to $e_{\{j\}\cup I}$'s, which is standard in Morse homology theory.  There are two such flow lines, one passing through the gauge torus $H_j + p$ and one passing through $H_j + \up$.  Since both of them are positive intersections (as holonomies of pearl trajectories), one contributes $z_i$ and one contributes $\uz_i$. What it remains to check that the signs of $z_i$ and $\uz_i$ are given precisely as in \eqref{eq:d1pearltoric}, which we postpone to Section \ref{subsec:signpearld1}.
\end{proof}

Before computing $d_{-1} = (-1)^{deg} (\delta_{\textnormal{pearl}})_{-1}$, we show that there does not exist a pearl trajectory from $e_I$ to $e_J$ with $|J|= |I|-1$, but $J \nsubseteq I$. 
The question is equivalent to ask when there exists a Maslov index two disc which connects $W^u (e_I)$ and $W^s (e_J)$, and the following lemma proves that this is impossible unless $J \subset I$.

\begin{lemma}\label{lem:nosiglepearl}
If $J$ is not a subset of $I$, then there does not exist a trajectories with a single pearl from $e_I$ to $e_J$.
\end{lemma}

\begin{proof}
Without loss of generality, one may assume that
$I= \{1,2, \cdots, k\}$ and $J= \{ l+1, l+2, \cdots, l+ (k-1)\}$ for $l \geq 2$.
Recall that we have taken our Morse function such that
$$e_{top} = (a_1, \cdots, a_n)$$
with an irrational slope $a_i /a_j$ for $1 \leq i \neq j \leq n$. In terms of coordinates of the Lagrangian torus, $W^u (e_I)$ and $W^s (e_J)$ are given modulo $\Z^n$ as 
$$ W^u (e_I) = \{ (a_1, \cdots, a_k, t_1, \cdots, t_{n-k} \,\, | \,\,  0\leq t_i \leq a_{i} \,\, \mbox{for}\,\,  k+1 \leq i \leq n   \}$$
$$ W^s (e_J) = \{ (0,\cdots, 0 , s_{l+1}, \cdots, s_{l+(k-1)}, 0, \cdots,  0 \,\, | \,\,  0\leq s_i \leq a_{i} \,\, \mbox{for}\,\,  l+1 \leq i \leq l+(k-1)   \}.$$

Now suppose there is a Maslov index two disc whose boundary image connects $W^u (e_I)$ and $W^s (e_J)$. Note that for toric manifolds, such a boundary image has an integral direction (which is normal to a facet of the moment polytope). However, if there is a vector from $W^u (e_I)$ to $W^s (e_J)$, then it should be of the form
$$ (a_1, a_2, \cdots , \cdots) \mod \Z^n$$
which can not be made integral by a multiplication of any scalar because $a_1 / a_2$ is irrational. This gives a contradiction.
\end{proof}

\begin{lemma}\label{lem:pearld2}
$$d_{-1}=(-1)^{\deg} (\delta_{\textnormal{pearl}})_{-1} = \sum_{j=1}^n \left( \sum_{i=1}^m c_i  \alpha^i_j \right) \iota_{e_j}$$
where $c_i = \int_{\beta_i} \omega$ and $\alpha^i_j$ takes the form
$$ \alpha^i_j = s_{i,j} \sum_{\gamma_j \cup D_i} Z_{\gamma_j} \cdot Z_{\partial D_i}$$
for $s_{i,j} = \mbox{the sign of} \,\, v_{i,j}$.
The sum is over the finitely many flow-disc trajectories from $e_j$ to $1$, where the flow part $\gamma_j$ is a segment of a flow line from $e_j$ to $e_{top}$ and the disc part $D_i$ is the basic disc representing $\beta_i$ passing through the point $1$ once.  $Z_{\gamma_j} \in \C^\times$ and $Z_{\partial D_i} \in \C^\times$ denote the holonomy contributions from $\gamma_j$ and $\partial D_i$ respectively.  Moreover $\alpha^i_j = \delta^i_j$ when $i=1,\ldots,n$.
\end{lemma}
\begin{proof}
Let $e_I$ and $e_J$ be two critical points with $|J| = |I|-1 \geq 0$.  By Lemma \ref{lem:nosiglepearl}, it suffices to consider the case when $J = I-\{j\}$ for some $1 \leq j \leq n$.
In such a case, we prove that there is an one-to-one correspondence between pearl trajectories from $e_I$ to $e_{I-\{j\}}$ and those from $e_j$ to $1$ both of which involves a single Maslov-2 disc of class $\beta_i$.  

A pearl trajectory $\Gamma_1$ from $e_I$ to $e_{I-\{j\}}$ consists of two flow line components and one disc component.  Given such a $\Gamma_1$, we can attach it with any chosen flow line from $e_j$ to $e_I$ and obtain a (degenerate) pearl trajectory $\tilde{\Gamma}_1$ from $e_j$ to $e_{I-\{j\}}$.  A pearl trajectory $\Gamma_2$ from $e_j$ to $1$ consists of two components: a flow component from $e_j$, which is glued to a holomorphic disc component representing $\beta_i$ at the entry point such that the disc boundary passes through $1 \in \bL$ at the exit point.  Given such a $\Gamma_2$, we can attach it with any chosen flow line from $1$ to $e_{I-\{j\}}$ and obtain a (degenerate) pearl trajectory $\tilde{\Gamma}_2$ from $e_j$ to $e_{I-\{j\}}$.

Given $\Gamma_2$, we construct an one-parameter family of pearl trajectories $\tilde{\Gamma}^t$ from $e_j$ to $e_{I-\{j\}}$ for $t\in [1,2]$ such that $\tilde{\Gamma}^2$ is $\Gamma_2$ attached with a flow line $F(1\to e_{I-\{j\}})$ from $1$ to $e_{I-\{j\}}$, and $\tilde{\Gamma}^1$ is a pearl trajectory $\Gamma_1$ from $e_I$ to $e_{I-\{j\}}$ attached with a flow line $F(e_j\to e_I)$ from $e_j$ to $e_I$.  (If this is the situation, we will associate $\Gamma_1$ with $\Gamma_2$.)

The construction is as follows.  The torus $T\rq{} \subset \bL$ passing through the entry point $p \in \bL$ generated by the directions $\{E_k: k \in I - \{j\}\}$ intersects the unstable torus of $e_I$ at a unique point $p\rq{}$ in the unstable torus of $e_j$.  Let $\alpha:[1,2] \to T\rq{}$ be a straight line segment with $\alpha(1) = p\rq{}, \alpha(2)=p$.  For each $t$, we have a flow line from $e_j$ to $\alpha(t)$ (chosen to be continuously depending on $t$) and a unique holomorphic disc representing $\beta_i$ whose boundary passes through $\alpha(t)$.  Moreover as we vary $t$ from $2$ to $1$ the exit point of the disc varies continuously from $1 \in \bL$ to other points in the stable torus of $e_{I-\{j\}}$, and there exists a flow line (continuously depending on $t$) from the exit point to $e_{I-\{j\}}$.  Thus for each $t$ we have a pearl trajectory $\tilde{\Gamma}^t$ from $e_j$ to $e_{I-\{j\}}$ with $\alpha(t)$ to be the entry point.  At $t=1$, since $\alpha(1) = p\rq{}$ lies in the unstable torus of $e_I$, the flow component from $e_j$ to $\alpha(t)$ actually degenerates to union of a flow line $F(e_j\to e_I)$ from $e_j$ to $e_I$ and a flow segment from $e_I$ to $\alpha(t)$.  Thus the pearl trajectory at $t=1$ is a pearl trajectory $\Gamma_1$ from $e_I$ to $e_{I-\{j\}}$ attached with $F(e_j\to e_I)$.

Conversely given $\Gamma_1$, we can construct a one-parameter family of pearl trajectories $\tilde{\Gamma}^t$ with the same property as above in a similar way, and obtain a corresponding pearl trajectory $\Gamma_2$.  The constructions are inverses to each other, and hence give the desired one-to-one correspondence.

Since $\tilde{\Gamma}^t$ is a continuous family of pearl trajectories with fixed input $e_j$ and output $e_{I-\{j\}}$, their boundaries give the same holonomy.  Moreover, the flow lines $F(1\to e_{I-\{j\}})$ and $F(e_j\to e_I)$ give exactly the same holonomy.  This implies $\Gamma_1$ and $\Gamma_2$ give exactly the same holonomy. 

In order to show that $d_{-1}$ is of the form $ \sum_{j=1}^n w_j \iota_{e_j}$ for some Laurent series $w_j$'s, we have to additionally check that the sign difference between $\Gamma_1 (e_I \to e_{I - \{j\}})$ and $\Gamma_2 (e_j \to 1)$ equals that of  $e_I$ and $e_j \wedge e_{I-\{j\}}$ (i.e. $s^\ast$ such that $\iota_{e_j} e_I = s^\ast e_{I - \{j\}}$). Here, we simply assume that they are equal, and postpone the proof of this to Section \ref{subsec:signpearld2}.

Consequently, we only need to consider pearl trajectories from $e_j$ to $1$ whose disc component represents $\beta_i$ in order to compute the coefficients $\alpha^i_j$.  Such pearl trajectories take the form $\gamma_j \cup D_i$ as stated.  If $v_{i,j} = 0$, any disc representing $\beta_i$ passing through $1 \in \bL$ cannot intersect the unstable torus of $e_j$, and hence there is no such pearl trajectory $\gamma_j \cup D_i$ i.e., $\alpha^i_j =0$.  

Suppose $v_{i,j} \not=0$. 
 There are just finitely many pearl trajectories (parametrized by the finitely many entry points), and     the disc component has area $\int_{\beta_i} \omega$ which contributes the factor $c_i$, and the holonomy contribution is $Z_{\gamma_j} \cdot Z_{\partial D_i}$. In Section \ref{subsec:signpearld2}, 
 the sign of this contribution will turn out to be $s_{i,j}$, which is $1$ (or $-1$) when $\partial D_i$ passes through the unstable torus of $e_j$ positively (or negatively).
Note that $\alpha^i_j = \delta^i_j$ for $i=1,\cdots,n$ due to our special choice of the basis $E_i = \partial \beta_i$ for $i=1, \cdots,n$.

Finally, the regularity of pearl trajectories contributing to  $(\delta_{\textnormal{pearl}})_{-1}$ will be shown in Section \ref{subsec:regularitypearls}. We will also prove in Lemma \ref{lem:Jreg1} that there is an almost complex structure which makes all trajectories for $\delta_{\textnormal{pearl}}$ regular, but does not changes $(\delta_{\textnormal{pearl}})_{-1}$, which justifies our computations in this section.
\end{proof}

Combining the above two lemmas,
$$ \tilde{d} = d_{1} + d_{-1} = \sum_{i=1}^m (z_i - \uz_i) e_i \wedge (\cdot) + \sum_{i=1}^m c_i \sum_{j=1}^n \alpha^i_j \iota_{e_j}. $$
We next prove that $(\largewedge^* \underline{\Lambda^n},\tilde{d})$ indeed forms a matrix factorization:
\begin{theorem} \label{thm:square=W}
$\tilde{d}^2 = W_z - W_{\uz}$.
\end{theorem}
\begin{proof}
We have
$$ \tilde{d}^2 = \sum_{i=1}^n c_i (z_i - \uz_i) + \sum_{i=n+1}^m c_i \sum_{j=1}^n \alpha^i_j (z_j - \uz_j) $$
and, comparing with \eqref{eq:WL1}, we need to prove that $\sum_{j=1}^n \alpha^i_j (z_j - \uz_j) = z^{v_i} - \uz^{v_i}$ where
\begin{equation} \label{eq:sum-entry}
\begin{aligned}
\sum_{j=1}^n \alpha^i_j (z_j - \uz_j) &= \sum_j \sum_{\gamma_j \cup D_i} Z_{\gamma_j} \cdot Z_{\partial D_i} (s_{i,j} (z_j - \uz_j))\\
&= \sum_{\textrm{entry point } \zeta} Z^\zeta_{\gamma_{j(\zeta)}} \cdot Z^\zeta_{\partial D_i} (s_{i,j(\zeta)} (z_{j(\zeta)} - \uz_{j(\zeta)})).
\end{aligned}
\end{equation}
Here, the summation is over all the entry points $\zeta \in \partial D_i$, and we order the entry points counterclockwisely.  The entry point counterclockwisely closest to $1 \in \partial D \subset T$ is said to be the first entry point.  Each entry point $\zeta$ is an intersection of a flow line from $e_{j(\zeta)}$ and $\partial D_i$ (we will write $j=j(\zeta)$ for simplicity).  Each summand $Z_{\gamma_j} \cdot Z_{\partial D_i} (s_{i,j} (z_j - \uz_j))$ is a difference of two terms.  We want to prove that for every two adjacent summands, the second term of the first one cancels with the first term of the second one.  Hence only the first term of the first summand and the last term of the last summand remain, and we claim that those equal to $z^{v_i}$ and $\uz^{v_i}$ respectively.  This will finish the proof that $d^2 = W_z - W_{\uz}$.

Consider two consecutive entry points $\zeta_1,\zeta_2 \in \partial D$, where $\zeta_1$ is flowing from $e_j$ and $\zeta_2$ is flowing from $e_k$ ($j$ may equal to $k$).  The unstable submanifolds of $e_j$ and $e_k$ are two hypertori intersecting each other along a sub-tori $T^{\perp} = T\langle\{E_1,\ldots,E_n\} - \{E_j,E_k\}\rangle$ passing through $e_{top}$.  (Here we use $E_1,\ldots,E_n$ to denote the standard basis of $\R^n$ and $e_1,\ldots,e_n$ are critical points in $T$ with index $n-1$.)  Write the holonomy terms $Z^{\zeta_1}_{\gamma_j} \cdot Z^{\zeta_1}_{\partial D_i}$ and $Z^{\zeta_2}_{\gamma_k} \cdot Z^{\zeta_2}_{\partial D_i}$ (which are monomials in $z_l^{\pm 1}$'s and $\uz_l^{\pm 1}$'s) as products of two factors $Z^{jk}_{\zeta_1} Z^\perp_{\zeta_1}$ and $Z^{jk}_{\zeta_2} Z^\perp_{\zeta_2}$ respectively, where $Z^{jk}_{p_l}$ only has variables $z_j,z_k,\uz_j,\uz_k$ and $Z^\perp_{p_l}$ has variables $z_a,\uz_a$'s for all $a \not= j,k$.  Then $Z^\perp := Z^\perp_{\zeta_1} = Z^\perp_{\zeta_2}$, and we only need to compare $Z^{jk}_{\zeta_1}$ and $Z^{jk}_{\zeta_2}$ in order to compute
$$ Z^{\zeta_1}_{\gamma_j} Z^{\zeta_1}_{\partial D_i} (s_{i,j} (z_j - \uz_j)) + Z^{\zeta_2}_{\gamma_k} Z^{\zeta_2}_{\partial D_i} (s_{i,k} (z_k - \uz_k)) =  Z^\perp (Z^{jk}_{\zeta_1}(s_{i,j} (z_j - \uz_j)) + Z^{jk}_{\zeta_2} (s_{i,k} (z_k - \uz_k))).$$

To analyze $Z^{jk}_{\zeta_1}$ and $Z^{jk}_{\zeta_2}$, we can project $\gamma_j, \gamma_k, \partial D_j, \partial D_k \subset T$ to $T / T^{\perp}$, which is one-dimensional when $j=k$ and two-dimensional when $j\not=k$.  When $j=k$, we choose another direction $E_l$ other than $E_j$ and project to $T/(T^{\perp}/T\langle E_l \rangle)$instead.  Thus in any case the computation goes back to dimension two.

First consider the case $j=k$.  $Z^{jk}_{\zeta_1}$ and $Z^{jk}_{\zeta_2}$ are monomials in $z_j$ and $\uz_j$.  See Figure \ref{MF-cancel-parallel}, where the flow lines are shown by dotted lines and $\partial D_i$ is shown by a solid line.  In this case $\partial D$ passes through the gauge hypertori $p+H_j$ and $\up+H_j$ once in between the two entry points $\zeta_1$ and $\zeta_2$.  For $s_{i,j}=\mbox{the sign of} \,\,v_{i,j}=1$, $Z^{jk}_{\zeta_1}$ has one more factor of $z_j$ and one less factor of $\uz_j$ than $Z^{jk}_{\zeta_2}$.  Thus $\uz_j Z^{jk}_{\zeta_1} = z_j Z^{jk}_{\zeta_2}$.  This shows that the two terms in the middle cancel each other and only the first and last ones are left:
\begin{equation} \label{eq:cancel}
Z^{\zeta_1}_{\gamma_j} Z^{\zeta_1}_{\partial D_i} (s_{i,j} (z_j - \uz_j)) + Z^{\zeta_2}_{\gamma_k} Z^{\zeta_2}_{\partial D_i} (s_{i,k} (z_k - \uz_k)) = Z^{\zeta_1}_{\gamma_j} Z^{\zeta_1}_{\partial D_i} s_{i,j} z_j -  Z^{\zeta_2}_{\gamma_k} Z^{\zeta_2}_{\partial D_i} s_{i,k} \uz_k.
\end{equation}
For $s_{i,j}=-1$, $Z^{jk}_{\zeta_1}$ has one more factor of $z^{-1}_j$ and one less factor of $\uz^{-1}_j$ than $Z^{jk}_{\zeta_2}$.  Thus $z_j Z^{jk}_{\zeta_1} = \uz_j Z^{jk}_{\zeta_2}$.  This implies the two terms in the middle cancel each other and the same equation holds.

\begin{figure}[htb!]
  \centering
   \begin{subfigure}[b]{0.4\textwidth}
   	\centering
    \includegraphics[width=\textwidth]{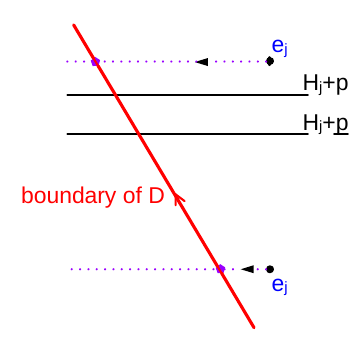}
    \caption{$j=k$.}
    \label{MF-cancel-parallel}
   \end{subfigure}
   \hspace{10pt}
   \begin{subfigure}[b]{0.4\textwidth}
   	\centering
   	 \includegraphics[width=\textwidth]{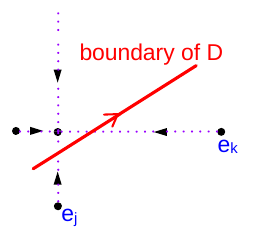}
    \caption{$j\not=k$.}
    \label{MF-cancel-jk}
   \end{subfigure}
\caption{Two consecutive entry points in $\partial D$.}
\label{MF-cancel}
\end{figure}

Now consider the case $j\not=k$.  See Figure \ref{MF-cancel-jk}.  Consider the case $s_{i,j} = s_{i,k} = 1$, and the other three cases $(s_{i,j},s_{i,k})=(1,-1),(-1,1),(-1,-1)$ are similar. $Z^{jk}_{\zeta_1}$ has one more $z_k$ factor than $Z^{jk}_{\zeta_2}$, because the flow segment $\gamma_j$ from $e_j$ hit the gauge hypertorus $H_k + p$ once while the flow segment $\gamma_k$ from $e_k$ does not.  Similarly, $Z^{jk}_{\zeta_2}$ has one more $\uz_j$ factor than $Z^{jk}_{\zeta_2}$ because the flow segment $\gamma_k$ from $e_k$ hit the gauge hypertorus $H_j + \up$ once while the flow segment $\gamma_j$ from $e_j$ does not.  Hence $\uz_j Z^{jk}_{\zeta_1} = z_k Z^{jk}_{\zeta_2}$.  This implies Equation \eqref{eq:cancel} also holds in this case.

In conclusion,  the right hand side of Equation \eqref{eq:sum-entry} equals to $\textrm{first term } - \textrm{ last term}$ because all intermediate terms cancel.  The first term is $Z^{\first}_{\gamma_{j(\first)}} Z^\first_{\partial D_i} z_{j(\first)}$ when $s_{i,j(\first)}=1$ and is $Z^\first_{\gamma_{j(\first)}} Z^\first_{\partial D_i} \uz_{j(\first)}$ when $s_{i,j(\first)}=-1$ where $\first$ is the first entry point.  The last term is $Z^\last_{\gamma_{j(\last)}} Z^\last_{\partial D_i} \uz_{j(\last)}$ when $s_{i,j(\last)}=1$ and is $Z^\last_{\gamma_{j(\last)}} Z^\last_{\partial D_i} z_{j(\last)}$ when $s_{i,j(\last)}=-1$ where $\last$ is the last entry point.

Now consider the first term which corresponds to the entry point $\zeta$ anti-clockwisely closest to $1 \in T$ along $\partial D$, and let $j = j(\first)$.  The flow segment from $e_j$ never hits the gauge hypertorus $H_j + p$ nor $H_j + \up$.  For $k\not=j$, it hits the gauge hypertorus $H_k+p$ if and only if $s_{i,k}=1$, and hits $H_k+\up$ if and only if $s_{i,k}=-1$.  Hence
$Z^\first_{\gamma_{j(\first)}} = \prod_{k\not=j} z_k^{\delta(s_{i,k},1)} \uz_k^{\delta(s_{i,k},-1)}$, where $\delta(a,b):=1$ if $a=b$ and zero otherwise.  For the disc component, for every $k$ the arc from $1$ to $\zeta$ (counterclockwisely) hits $H_k + p$ if and only if $s_{i,k}=1$, and hits $H_k+\up$ if and only if $s_{i,k}=-1$.  Also $\partial D$ hits $H_k + p$ and $H_k + \up$ $|v_{i,k}|$ times respectively.  Recall that on the arc from $1$ to $\zeta$ (which is mapped to $\tilde{L}$), only intersection with $H_k + \up$ (but not $H_k + p$) contributes; on the opposite arc from $\zeta$ to $1$ (which is mapped to $\bL$), only intersection with $H_k+p$ contributes.  Thus
$$ Z^\first_{\partial D_i} = \prod_{k=1}^n \left(\uz_k^{-\delta(s_{i,k},-1)} \frac{z^{v_{i,k}}}{z_k^{\delta(s_{i,k},1)}}\right)$$
and so $Z^\first_{\gamma_{j(\first)}} Z^\first_{\partial D_i} = \uz_j^{-\delta(s_{i,j},-1)} z_j^{-\delta(s_{i,j},1)} \prod_{k=1}^n z^{v_{i,k}}$.  Thus the first term is $z^{v_i} = \prod_{k=1}^n z^{v_{i,k}}$.

The derivation of the last term to be $\uz^{v_i}$ is very similar and left to the reader.  This proves $\sum_{j=1}^n \alpha^i_j (z_j - \uz_j) = z^{v_i} - \uz^{v_i}$.
\end{proof}

\subsection{Computation of the main terms}\label{subsec:MFtoricfano2}
We now derive an explicit expression of $\tilde{d}$ by computing the coefficients
$$ \alpha^i_j = s_{i,j} \sum_{\gamma_j \cup D_i} Z_{\gamma_j} \cdot Z_{\partial D_i}. $$
By Lemma \ref{lem:pearld2}, it suffices to compute the holonomy contribution from a single pearl trajectory from $e_j$ to $1$ involving the $D_i$-disc.
$\partial D_i$ is a circle spanned by the direction $v_i$ in $\bL$, which hits the unstable submanifold of $e_j$ $|v_{i,j}|$ number of times.  Hence there are in total $|v_{i,j}|$ number of pearl trajectories, which are parametrized by the entry points along $\partial D_i$.


\begin{figure}[htb!]
  \centering
   \begin{subfigure}[b]{0.45\textwidth}
   	\centering
    \includegraphics[width=\textwidth]{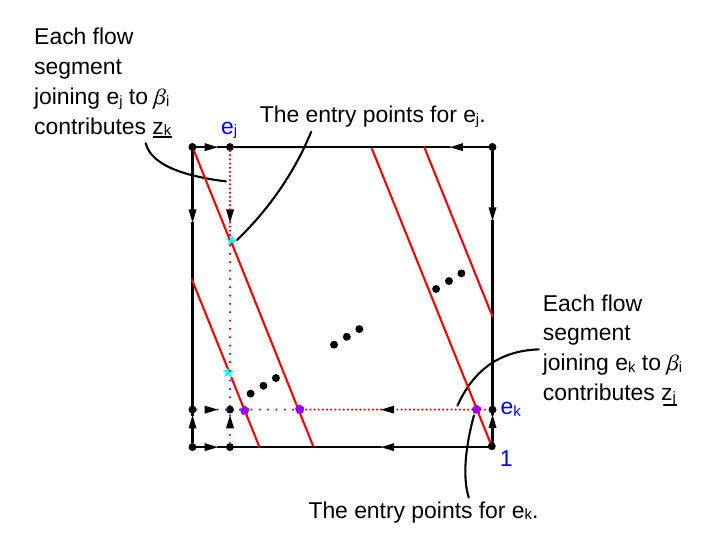}
    \caption{The case when $\{s_{i,j},s_{i,k}\} = \{1,-1\}$.}
    \label{MF-jk-plane1}
   \end{subfigure}
   \begin{subfigure}[b]{0.45\textwidth}
   	\centering
   	 \includegraphics[width=\textwidth]{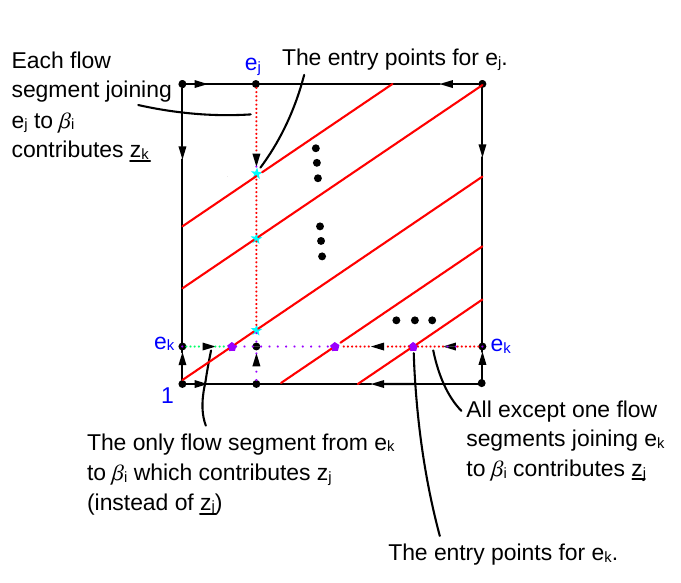}
    \caption{The case when $s_{i,j} = s_{i,k} \not=0$.}
    \label{MF-jk-plane2}
   \end{subfigure}
\caption{The $jk$-plane, $j<k$.}
\label{MF-jk-plane}
\end{figure}

First let's analyze the holonomy contribution $Z_{\gamma_j}$ from the flow segment $\gamma_j$.  It is easier to do so by projecting to the $jk$-plane as in Figure \ref{MF-jk-plane}, where the assumption \eqref{cond:sij} on the slope of $\partial D_i$ and $a_j / a_k$ is used.  Since the flow is contained in a hypertorus normal to $e_j$, it never hits the gauge hypertori $p + H_j$ and $\up + H_j$.  Now consider gauge hypertori $p + H_k$ and $\up + H_k$ for $k\not=j$.  When $s_{i,k}=0$, the whole pearl trajectory is contained in the hypertorus containing $1 \in T$ parallel to $p + H_k$, and hence it never hits the gauge hypertori $p + H_k$ and $\up + H_k$.  When $s_{i,k} \not= 0$ and $s_{i,k} \not= s_{i,j}$, the flow segment hits the gauge hypertorus $\up + H_k$ but not $p + H_k$ in the negative transverse orientation (see Figure \ref{MF-jk-plane1}).  It contributes $\uz_k$ to the holonomy.  When $s_{i,k} = s_{i,j} \not=0$, we further divide into two cases: $j<k$ and $j>k$ (see Figure \ref{MF-jk-plane2}).  When $j<k$, the flow segment hits the gauge hypertorus $\up + H_k$ but not $p + H_k$ in the negative transverse orientation as in the previous case; it contributes $\uz_k$ to the holonomy.  When $j>k$, all but one flow-disc configurations have their flow segments hitting $\up + H_k$ but not $p + H_k$ in the negative transverse orientation, which contributes $\uz_k$ to the holonomy; the exceptional flow-disc configuration has its flow segment hitting $p + H_k$ but not $\up + H_k$ in the positive transverse orientation, contributing $z_k$ to the holonomy.

Second, consider the holonomy contribution $Z_{\partial D_i}$ from the disc component $D_i$, which is the basic holomorphic disc whose boundary passes through the critical point $1$ whose class is $\beta_i$.  The holonomy is computed by considering the intersection points of $\partial D$ with the gauge hypertori $H_l+p$ and $H_l + \up$ for $l=1,\ldots,n$.  When $s_{i,l} = 0$, there is no intersection with $H_l+p$ nor $H_l + \up$.  When $s_{i,l} > 0$, $\partial D_i$ passes through $H_l+p$ (or $H_l+\up$) in the positive transverse direction, and so each of the intersections is marked by $z_l$ (or $\uz_l$ resp.); otherwise when $s_{i,l}<0$, each of the intersections is marked by $z_l^{-1}$ (or $\uz_l^{-1}$ resp.).  The number of points on $\partial D_i$ marked as $z_l^\pm$ is the same as that marked as $\uz_l^\pm$, which is $|v_{i,l}|$.  Figure \ref{MF-nd-circ} shows the disc component $D_i$.  Walking along the circle in positive orientation starting from $1$, if $s_{i,l}=1$ then we first encounter $H_l+p$, marked as $z_l$, and later $H_l+\up$, marked as $\uz_l$; if $s_{i,l}=-1$, then we first encounter $\uz_l^{-1}$ and later $z_l^{-1}$.

\begin{figure}[htb!]
\begin{center}
\includegraphics[height=3in]{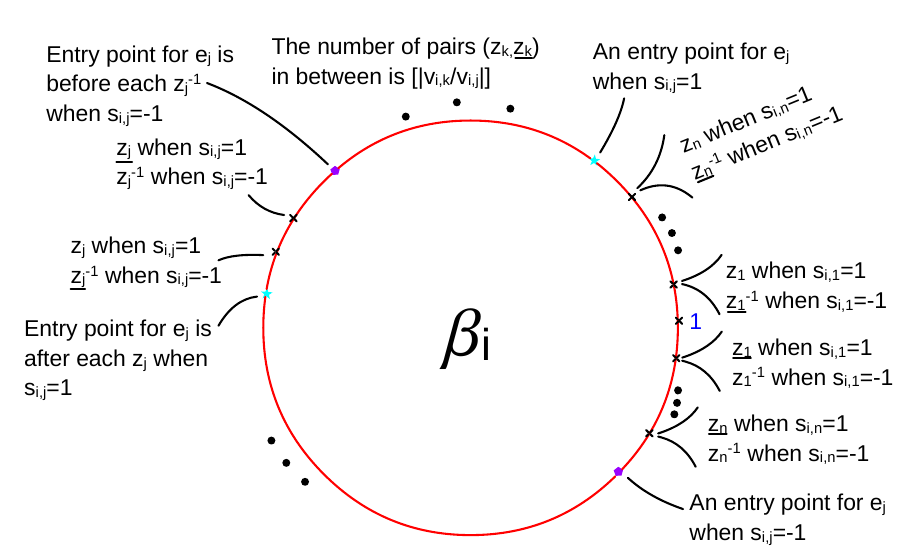}
\caption{The basic holomorphic disc passing through the critical point $1 \in T$ representing $\beta_i$.  When the boundary passes through a gauge hypertorus, there is a holonomy term as marked in the diagram.}\label{MF-nd-circ}
\end{center}
\end{figure}

For each possible entry position, we need to count the number of markings $\uz_k^{\pm 1}$ in the arc of boundary circle counterclockwisely from $1$ to the entry position, and the number of $z_k^{\pm 1}$ in the arc from the entry position to $1$ (which is the part mapped to $\bL$).  There are $\lfloor p |v_{i,k}/v_{i,j}|\rfloor$ pairs of $(z_k^{s_{i,k}},\uz_k^{s_{i,k}})$ between the first and $(p+1)$-th entry positions.  

From now on we assume $s_{i,j} = 1$, and the case $s_{i,j} = -1$ can be analyzed in a similar way.  For $k=j$, the arc from $1$ to the first entry position (counterclockwisely) contains one $z_j$ and no $\uz_j$.  Then the opposite arc from the first entry position to $1$ contains $(|v_{i,j}|-1)$ number of $z_j$'s and $|v_{i,j}|$ number of $\uz_j$'s.  Moreover, the flow segment for such a configuration contributes no $z_j$ nor $\uz_j$.  Thus the holonomy contribution is $z_j^{|v_{i,j}|-1} = z_j^{v_{i,j}-1}$.

For $k<j$ with $s_{i,k} = 1$, the arc from $1$ to the first entry position contains one $z_k$ and no $\uz_k$ (see Figure \ref{MF-jk-plane2} with roles of $j$ and $k$ switched).  Then the opposite arc from the first entry position to $1$ contains $(|v_{i,k}|-1)$ number of $z_k$'s and $|v_{i,k}|$ number of $\uz_k$'s.  Moreover, the flow segment for such a configuration contributes $z_k$.  Thus the holonomy contribution is $z_k^{|v_{i,k}|-1} \cdot z_k = z_k^{v_{i,k}}$.

For $k>j$ with $s_{i,k} = 1$, the arc from $1$ to the first entry position contains one $z_k$ and no $\uz_k$ (see Figure \ref{MF-jk-plane2}).  Then the opposite arc from the first entry position to $1$ contains $(|v_{i,k}|-1)$ number of $z_k$'s and $|v_{i,k}|$ number of $\uz_k$'s.  Moreover, the flow segment for such a configuration contributes $\uz_k$.  Thus the holonomy contribution is $z_k^{|v_{i,k}|-1} \cdot \uz_k = z_k^{v_{i,k}-1} \uz_k$.

For $k\not=j$ with $s_{i,k} = -1$, the arc from $1$ to the first entry position contains one $\uz_k^{-1}$ and no $z_k^{-1}$ (see Figure \ref{MF-jk-plane1}).  Then the opposite arc from the first entry position to $1$ contains $|v_{i,k}|$ number of $z_k^{-1}$'s and $(|v_{i,k}|-1)$ number of $\uz_k^{-1}$'s.  Moreover, the flow segment for such a configuration contributes $\uz_k$.  Thus the holonomy contribution is $z_k^{v_{i,k}} \cdot \uz_k^{-1} \cdot \uz_k = z_k^{v_{i,k}}$.

Multiplying all holonomies for $k=1,\ldots,n$ together, the total holonomy of the configuration corresponding to the first entry position is
$$z_1^{v_{i,1}} \ldots z_{n}^{v_{i,n}}z_j^{-1} \left(\prod_{k\not=j} z_k^{-\delta(s_{i,k},1)} \right) \left( \prod_{k\not=j} \uz_k^{-\delta(s_{i,k},-1)} \right) \left(\prod_{k>j} \uz_k^{|s_{i,k}|} \right) \left(\prod_{k<j} \uz_k^{\delta(s_{i,k},-1)}z_k^{\delta(s_{i,k},1)} \right).$$
The factor $z_1^{v_{i,1}} \ldots z_{n}^{v_{i,n}}z_j^{-1} \left(\prod_{k\not=j} z_k^{-\delta(s_{i,k},1)} \right)$ comes from the arc from first entry position to $1$; the factor $\left( \prod_{k\not=j} \uz_k^{-\delta(s_{i,k},-1)} \right)$ comes from the arc from $1$ to the first entry position; the remaining factor $\left(\prod_{k>j} \uz_k^{|s_{i,k}|} \right) \left(\prod_{k<j} \uz_k^{\delta(s_{i,k},-1)}z_k^{\delta(s_{i,k},1)} \right)$ comes from the flow segment.

Now we compute holonomies of configurations corresponding to the $(p+1)$-th entry position, $p=1,\ldots,v_{i,j}-1$.  For $k=j$ or $s_{i,k} = 0$, the flow segment contributes nothing; otherwise the flow segment always contributes $\uz_k$.  Thus the holonomy contribution of the flow segment is
$$ \prod_{k\not=j} \uz_k^{|s_{i,k}|}. $$
For the arc from $1$ to the $(p+1)$-th entry position, the number of $\uz_k^{s_{ik}}$'s is $\lfloor p |v_{i,k}/v_{i,j}|\rfloor$ plus that for the arc from $1$ to the first entry position.  Thus the holonomy contribution of the arc from $1$ to the $(p+1)$-th entry position is
$$\left(\prod_{k=1}^n \uz_k^{s_{i,k} \lfloor p |v_{i,k}/v_{i,j}|\rfloor}\right) \left( \prod_{k\not=j} \uz_k^{-\delta(s_{i,k},-1)} \right).$$
Similarly the number of $z_k^{s_{i,k}}$'s for the arc from the $(p+1)$-th entry position to $1$ is $\lfloor p |v_{i,k}/v_{i,j}|\rfloor$ less than that for the arc from the first entry position to $1$.   Thus the holonomy contribution of the arc from $1$ to the $(p+1)$-th entry position is
$$ \left(\prod_{k=1}^n z_k^{-s_{i,k} \lfloor p |v_{i,k}/v_{i,j}|\rfloor}\right) z_1^{v_{i,1}} \ldots z_{n}^{v_{i,n}}z_j^{-1} \left(\prod_{k\not=j} z_k^{-\delta(s_{i,k},1)} \right).$$
Hence the total holonomy from the flow segment and the disc boundary is
$$z_1^{v_{i,1}} \ldots z_{n}^{v_{i,n}}z_j^{-1}\left( \prod_{l\not=j} \uz_l^{-\delta(s_{i,l},-1)} \right)\left( \prod_{l\not=j} z_l^{-\delta(s_{i,l},1)} \right) \left(\prod_{l\not=j} \uz_l^{|s_{i,l}|}\right) \prod_{l=1}^n \left(\frac{\uz_l}{z_l}\right)^{s_{i,l} \left\lfloor p \left|\frac{v_{i,l}}{v_{i,j}}\right| \right\rfloor}.$$

The other case $s_{i,j}=-1$ can be analyzed similarly.  We obtain

\begin{theorem} \label{thm:explicit}
The matrix factorization $(\largewedge^* \underline{\Lambda^n}, \tilde{d})$ is
\begin{equation}
\tilde{d} = \left(\sum_{i=1}^n (z_i - \uz_i) e_i \wedge\right) + \left(\sum_{i=1}^n c_i \iota_{e_i}\right) + \left( \sum_{i=n+1}^m c_i \sum_{j=1}^n \alpha^i_j \iota_{e_j} \right)
\end{equation}
where $\alpha^i_j=0$ when $v_{i,j} = 0$,
\begin{align*}
\alpha^i_j =& z_1^{v_{i,1}} \ldots z_{n}^{v_{i,n}}z_j^{-1}\left( \prod_{l\not=j} \uz_l^{-\delta(s_{i,l},-1)} \right)\left( \prod_{l\not=j} z_l^{-\delta(s_{i,l},1)} \right) \left(\left(\prod_{l>j} \uz_l^{|s_{i,l}|} \right) \left(\prod_{l<j} \uz_l^{\delta(s_{i,l},-1)}z_l^{\delta(s_{i,l},1)} \right) \right. \\
&+ \left.\left(\prod_{l\not=j} \uz_l^{|s_{i,l}|}\right) \sum_{p=1}^{v_{i,j}-1} \left(\frac{\uz_j}{z_j}\right)^p \prod_{l\not=j}\left(\frac{\uz_l}{z_l}\right)^{s_{i,l} \left\lfloor p \left|\frac{v_{i,l}}{v_{i,j}}\right| \right\rfloor} \right)
\end{align*}
when $s_{i,j} = 1$, and
\begin{align*}
\alpha^i_j =& \uz_1^{v_{i,1}} \ldots \uz_{n}^{v_{i,n}}z_j^{-1}\left( \prod_{l\not=j} z_l^{-\delta(s_{i,l},-1)} \right)\left( \prod_{l\not=j} \uz_l^{-\delta(s_{i,l},1)} \right) \left(\left(\prod_{l>j} \uz_l^{|s_{i,l}|} \right) \left(\prod_{l<j} \uz_l^{\delta(s_{i,l},1)}z_l^{\delta(s_{i,l},-1)} \right) \right. \\
&+ \left.\left(\prod_{l\not=j} \uz_l^{|s_{i,l}|}\right) \sum_{p=1}^{|v_{i,j}|-1} \left(\frac{\uz_j}{z_j}\right)^p \prod_{l\not=j}\left(\frac{z_l}{\uz_l}\right)^{s_{i,l} \left\lfloor p \left|\frac{v_{i,l}}{v_{i,j}}\right| \right\rfloor} \right)
\end{align*}
when $s_{i,j} = -1$.
\end{theorem}

As a simple application,  the wedge-contraction type matrix factorization $(\largewedge^* \underline{\Lambda^n}, \tilde{d})$ for $X= \C P^n$ is given as follows:

\begin{corollary}
The matrix factorization $\tilde{d}$ corresponding to the Clifford torus with the holonomy $(\uz_1, \cdots, \uz_n)$ in $\C P^n$ is 
$$\tilde{d} = \sum_{i=1}^n (z_i - \uz_i) e_i \wedge + \sum_{i=1}^n \left(T^k - \frac{T^k}{\uz_1 \cdots \uz_i z_i \cdots z_n}\right) \iota_{e_i}$$
where $k$ is the (common) area of $(n+1)$ Maslov-2 discs bounding the central fiber. 
\end{corollary}

\subsection{Transversality}\label{subsec:regularitypearls}
In this section, we discuss the regularity of relevant moduli space of pearl trajectories which appeared throughout the section.
We first show that the moduli spaces $\CM_2 (e_I,e_J)$ of pearl trajectories in $X$ from $e_I$ to $e_J$ for $|J|=|I|-1$ are regular.  Because of degree reason, this moduli space consists of pearl trajectories with a unique disc of  Maslov index two.

The moduli of Maslov-$2$ holomorphic discs are known to be regular in toric Fano case by Cho-Oh \cite{CO}.  Also the Morse functions that we have chosen are Morse-Smale and hence, satisfy the transversality condition.  As the moduli space $\CM_2 (e_I,e_J)$ is given by $ev_\beta^{-1} (W^u (e_I) \times W^s (e_J))$ for the unstable manifold $W^u (e_I)$ of $e_I$ and the stable manifold $W^s (e_J)$ of $e_J$, it only remains to prove that the map $ ev_\beta =(ev_1, ev_0) : \mathcal{M}_2 (\bL,J,\beta) \to \bL \times \bL$
is transversal to $W^u (e_I) \times W^s (e_J)$.

\begin{lemma}\label{lem:trans}
With the setting as above,
$ev_\beta :\mathcal{M}_2 (\bL,J,\beta) \to \bL \times \bL$ is transversal to $W^u (e_I) \times W^s (e_J)$.
\end{lemma}

\begin{proof}
From the condition on $I$ and $J$, we may assume that $I=\{ 1, 2, \cdots,k \}$ and $J= \{1,2 \cdots, k-1\}$. At an intersection point $p$ of $ev_\beta$ and $W^u (e_I) \times W^s (e_J)$, directions of flow lines in $W^u (e_I) \times W^s (e_J)$ generates 
$$\R \langle (E_{k+1}, 0),(E_{k+2},0) \cdots, (E_n,0)\rangle \oplus \R \langle(0,E_1), (0,E_2), \cdots, (0,E_{k-1}) \rangle$$
in $T_p (\bL \times \bL)$.

Moreover, the translations in $\mathcal{M}_2 (\bL,J,\beta)$ due to the torus action give rise to
$\bigoplus_{i=1}^n \R \langle (E_i, E_i) \rangle$
Combining these two, $ev_\beta$ and $W^u (e_I) \times W^s (e_J)$ already generate
\begin{equation}\label{eq:genexcek}
\R \{ (E_i, E_j) : i \neq k, j \neq k\} \oplus \R \langle (E_k,E_k) \rangle \leq T_p (\bL\times \bL),
\end{equation}
and hence, it suffices to prove that the movement of markings in $\mathcal{M}_2(\bL,J,\beta)$ induces the $(E_k,0)$-direction. To see this, recall that
$$W^u (e_I) \subset \{ (a_1, \cdots,a_{k-1}, a_k, t_1, t_2 \cdots, t_{n-k}) : t_i \in \R \}$$
$$W^s (e_J) \subset \{ (s_1 a_1, \cdots, s_{k-1} a_{k-1}, 0, 0, \cdots, 0) : s_j \in \R \}$$
(modulo $\Z^n$).
Since $a_k$ is not an integer, the boundary image of the disc associated with the intersection point $p$ should have a nontrivial $e_k$-component. Therefore, if we vary the location of the first marking, we obtain a vector $(v,0)$ in $T_p (\bL\times \bL)$ where $v$ has a nontrivial $E_k$-component. This together with \eqref{eq:genexcek} proves the lemma.
\end{proof}


In \cite{BC}, Biran and Cornea proved that there exists a second category $J_{reg}$ of compatible almost complex structures on $X$ whose moduli space of pearl trajectories is a smooth manifold of expected dimension.
It is not clear that the standard complex structure $J_0$, which is regular for a trajectory with a single pearl, belongs to $J_{reg}$.
But we can show that the single pearl computation with a generic $J$ sufficiently close to $J_0$ is the same as that with $J_0$,
basically due to the fact that single pearl trajectories for $J_0$ are already Fredholm regular.

\begin{lemma}\label{lem:Jreg1}
There exists an almost complex structure $J \in J_{reg}$ sufficiently close to $J_0$ such that
the single pearl contribution of the matrix factorization (of $J$) is identical to that of $J_0$.
\end{lemma}
\begin{proof}
Let $p$, $q$ be two critical points of the chosen Morse function, and consider the one-dimensional moduli space of single pearl trajectories $\mathcal{M} := \cup_{t \in [0,1]} \{ t \} \times \mathcal{M}_2 (p,q,\beta,J_t)$ for one parameter
family of compatible almost complex structures $\{J_t\}_{t \in [0,1]}$.
By standard argument as in \cite{BC}, we can find such an one parameter family of almost complex structures $\{J_t\}_{t \in [0,1]}$ 
starting from the standard complex structure $J_0$ such that $\mathcal{M}$ is a smooth one-dimensional manifold with boundary.
We can choose $t_{p,q}$ such that in the part $ 0 <t  < t_{p,q}$, there exists no creation or cancellation phenomenon of the
cobordism $\mathcal{M}$. We take a minimum of $t_{p,q}$ for all such $p,q$ and denote it by $t_e$. Recall that the matrix
factorization is given from the relative homotopy class of a path (with fixed end points $p,q$) for a pearl trajectory by considering
intersection data with hyper-tori. It is easy to see that for sufficiently small $0<t<t_e$ with $J_t \in J_{reg}$, single pearl trajectories for $J_t$ and $J_0$ give rise to the same homotopy class of paths, and hence provides the identical matrix factorization.
%
\end{proof}

\section{Applications}\label{sec:app}

From the previous section, the matrix factorization mirror to a Lagrangian torus fiber $(\uL,\cL_\uz)$ takes the form $d = d_{1} + d_{-1} + \cdots + d_{-(2\lfloor(n+1)/2\rfloor-1)}$ on $\largewedge^* \underline{\Lambda^n}$ where $d^2 = W_z - W_\uz$, see Equation \eqref{eq:decompd1}.   We derived an explicit expression for the approximated matrix factorization $\tilde{d} = d_{1} + d_{-1}$ in Theorem \ref{thm:explicit}.  

In this section, we prove that $d$ generates $D\MF(W(z) - W(\uz))$ by using the spectral-sequence method by Polishchuk-Vaintrob \cite{PV} and the result of Dyckerhoff \cite{Dyc}.  The key is that the higher order terms $d_l$ for $l < -1$ have no contribution to the spectral sequence which computes the cokernel of $d$, and hence cokernel of $d$ is just the same as that of its approximation $\tilde{d}$ which generates the category.
We also deduce a change of coordinates which brings $d$ to $\tilde{d}$ in low dimensions $n \leq 4$.

\subsection{Generating the category of matrix factorizations}

Let $N$ be a lattice of rank $n$, and let $\{E_1,\ldots,E_n\}$ be a basis of $N$.  For $v_i \in N$ for $i=1,\ldots,m$, let $W = \sum_{i=1}^m c_i z^{v_i}$ be a Laurent polynomial, where $c_i \in \C$ are constants.  Here $z^{v_i}$ denotes the monomial $\prod_{l=1}^n z_l^{v_{i,l}}$, where $v_i = \sum_{l=1}^n v_{i,l} E_l$ is written in terms of the basis $\{E_l\}$.  Let $\Jac(W) = \C[z_1^{\pm 1}, \ldots, z_n^{\pm 1}]/(z_1 \partial_{z_1} W, \ldots, z_n \partial_{z_n} W)$ be the Jacobian ring of $W$, which can be regarded as the deformation space of $W$.  (We change from $\Lambda$ to $\C$ by setting $T = e^{-1}$ in this section.)

The number of critical points of $W$ (counted with multiplicities) equals to $\dim \Jac(W)$.  The category of matrix factorizations of $W$ can be written as a direct sum:
$$ \bigoplus_{l} D\MF (W(z) - W(\uz^{(l)})) $$
where $l$ labels the critical points.  For each $l$, $W(z) - W(\uz^{(l)})$ is identified as an element in $\C[[z_1-\uz_1^{(l)},\ldots,z_n-\uz_n^{(l)}]]$, which has an isolated critical point at $z=\uz^{(l)}$.

Fix a critical point $\uz = \uz^{(l)}$.  We have the matrix factorization $(\largewedge^* \underline{\C^n}, R_0(\uz))$ given in Theorem \ref{thm:explicit}:
\begin{equation*}
R_0(\uz) = \left(\sum_{i=1}^n (z_i - \uz_i) E_i \wedge\right) + \left( \sum_{i=1}^m c_i \sum_{j=1}^n \alpha^i_j \iota_{E_j} \right)
\end{equation*}
where $\alpha^i_j$'s are as stated there.  All terms involved in the definition are combinatorial and only depend on $W$.  This matrix factorization is of wedge-contraction type, and hence serves as a generator of $D\MF (W(z) - W(\uz^{(l)}))$ by \cite{Dyc}.  At this point we do not need to assume $W$ to be the Landau-Ginzburg mirror of a toric manifold.  For instance $W$ could be mirror to a Grassmannian.  These generators should be helpful for proving homological mirror symmetry.

Now suppose $W$ is mirror to a toric Fano manifold.  A critical point $\uz$ corresponds to a non-displaceable Lagrangian torus fiber \cite{FOOOT}.  The matrix factorization $R^{(l)}$ given above is the approximated matrix factorization $\tilde{d} = d_1 + d_{-1}$ mirror to a critical Lagrangian torus fiber.  By spectral sequence technique below, we see that the higher order terms $d_k$ for $k \leq -3$ are useless to generation, and so the mirror matrix factorizations $(\largewedge^* \underline{\C^n},d)$ generates as well.

\begin{theorem} \label{thm:gen}
The matrix factorization $R=(\largewedge^* \underline{\C^n},d)$ split generates $D\MF(W(z) - W(\uz))$.
\end{theorem}

\begin{proof}
$W(z) - W(\uz)$ (where $\uz$ is constant) has an isolated singularity at $z = \uz$.  Let $R = \Lambda[z_1^{\pm_1},\ldots,z_n^{\pm_1}]$ and $I = (z_1-\uz_1,\ldots,z_n-\uz_n)$ a maximal ideal.   $d$ consists of $d^{1 \to 0}: \largewedge^{\textrm{odd}} \underline{\C^n} \to \largewedge^{\textrm{even}} \underline{\C^n}$ and $d^{0 \to 1}: \largewedge^{\textrm{even}} \underline{\C^n} \to \largewedge^{\textrm{odd}} \underline{\C^n}$.  

Suppose $n$ is even.  We shall show that the cokernel of $d^{1 \to 0}$ has the $R/(W-W(\uz))$-module $R/I$ (which is geometrically the point $\uz$ of the hypersurface $\{W(z) - W(\uz)=0\}$) as a direct summand.  Since $R/I$ split generates the category, $d^{1 \to 0}$ also split generates.  For $n$ is odd, we shall replace $d^{0 \to 1}$ in the above statement.

The cokernel is the cohomology of the second term of the sequence 
$$0 \to \largewedge^{\textrm{odd}} \underline{\C^n} \to \largewedge^{\textrm{even}} \underline{\C^n} \to 0$$
where the middle arrow is $d^{1 \to 0}$.  Since both $\largewedge^{\textrm{odd}}$ and $\largewedge^{\textrm{even}}$ are graded, it gives the following spectral sequence whose total cohomology compute the cokernel of $d^{1 \to 0}$.

\begin{equation*}
\begin{small}
\xymatrix{ \largewedge^{n} \underline{\C^n} &  &   &   \\
\largewedge^{n-1} \underline{\C^n} \ar[r]^{d_{-1}} \ar[u]^{d_{1}} \ar@{.>}[rrd]^>>>>>>>{d_{-3}}& \largewedge^{n-2}\underline{\C^n} & &  \\
 & \largewedge^{n-3} \underline{\C^n} \ar[r] \ar[u] \ar@{.>}[rrd] &  \largewedge^{n-4}\underline{\C^n} & \cdots \\
 & & \vdots& \ddots  \\
 &  E^0 \,\, \mbox{page} & &  }
\xymatrix{ R/I &  &   &   \\
\ker \left(d_{1}|_{\largewedge^{n-1}} \right)\ar[r]^>>>>>{d_{-1}} \ \ar@{.>}[rrd]^>>>>>>>>>>{d_{-3}}& {\rm coker} \left(d_1|_{\largewedge^{n-3}} \right) & &  \\
 & \ker \left(d_{1}|_{\largewedge^{n-3}} \right) \ar[r]  \ar@{.>}[rrd]^>>>>>>>>>>{d_{-3}} &   {\rm coker} \left(d_{1}|_{\largewedge^{n-5}} \right) &   \\
 & & \vdots&  \\
 &  E^1\,\, \mbox{page} & &  }
\end{small}
\end{equation*}
By Lemma \ref{lem:pearld1}, $d_{1} = \sum_{i=1}^m (z_i - \uz_i) e_i \wedge (\cdot)$.  Thus ${\rm coker} (d_1:\largewedge^{n-1} \underline{\C^n} \to \largewedge^{n} \underline{\C^n}) = R/I$ which is the top left corner of the $E^1$ page.  Moreover from the exactness of the Koszul complex $(\largewedge^{*} \underline{\C^n}, d_{1})$, we have
$\ker \left(d_{1}|_{\largewedge^{k}} \right) = {\rm im} \left(d_{1}|_{\largewedge^{k-1}} \right)$ for $k=0,\ldots,n-1$.  Since $d_{-1} \circ d_1 = (W_z - W_{\uz}) \cdot \Id$ by Theorem \ref{thm:square=W}, the horizontal maps $d_{-1}$ in the $E^1$ page are identified with the multiplication by $(W_z - W_{\uz})$, which is injective.
As a result the $E^2$ page has only diagonal terms being non-zero and hence stabilizes.  We see that $R/I$ is a direct summand of the cokernel of $d^{1 \to 0}$.

When $n$ is odd, we consider the cokernel of $d^{0 \to 1}$, which is the cohomology of the second term of the sequence 
$$0 \to \largewedge^{\textrm{even}} \underline{\C^n} \to \largewedge^{\textrm{odd}} \underline{\C^n} \to 0. $$  The spectral sequence is the same as above and we have the same conclusion.
\end{proof}

\subsection{Mirror matrix factorizations in low dimensions} \label{sec:qe}
Suppose $\dim_\C X \leq 4$.  We prove that the matrix factorization $R$ in Theorem \ref{thm:main} is exactly of wedge-contraction type in this section.
\begin{theorem}
The mirror matrix factorization of $W(\bL)$ corresponding to the critical point $(\bL,\uz)$ ($\in \Fuk (X)$) is of {\em wedge-contraction type}. Namely,
we have the $\Z/2$-graded free $\Lambda$-module  $\largewedge^* \underline{\Lambda^n}$
generated by $e_I^{new}$ for $I \subset \{1,\cdots, n\}$ and 
$$d  =  \sum x_i e_i^{new} \wedge_{new}  + \sum_{i \in I}  w_i \iota_i^{new},$$
such that for $\lambda = W(\bL)(\uz)$, we have $d^2 = W(\bL)(z) - \lambda$.
\end{theorem}

Here, we used the notation $e_I^{new}$ instead of $e_I$ as we need to define a new (quantum) exterior algebra structure
to make it of wedge-contraction type (in dimension 4).
\begin{proof}
Denote by $(P,d)$ the matrix factorization of $W(\bL)$ corresponding to $(\bL,\uz)$ obtained from the pearl complex.
 By the degree reason (with $\dim(X) \leq 4$), the decomposition of $d$ \eqref{eq:decompd1} has only three components:
$$d = d_{1} + d_{-1}  + d_{-3}.$$
First, let us assume that $\dim X \leq 2$
In this case, $d_{-3}$ vanishes  by degree reason, and we already show that  $d_{1} + d_{-1}$ is  a wedge-contraction type matrix factorization of $W-\lambda$.

Now, let $\dim X =3$. We prove that $d_{-3} =0$.
By degree reason, it is enough to show that $d_{-3} (e_{top})=0$ for $e_{top}=e_1 e_2 e_3$.
We expand $d^2$ using the above decomposition of $d$, and the degree -($-2$) component
of $d^2 = (W-\lambda ) \cdot Id$ gives $d_{1} \circ d_{-3} + d_{-3} \circ d_{1}=0$.
Applying it for $e_2 \wedge e_3$, we get
$$0=(d_{1} \circ d_{-3} + d_{-3} \circ d_{1} ) ( e_2 \wedge e_3) = d_{-3} ( (z_1- \uz_1) e_{top} )= (z_1 - \uz_1) \,d_{-3} (e_{top}).$$
Since $P$ is torsion-free, this implies that $d_{-3} (e_{top})=0$ as desired. 

Finally, let $\dim X =4$. For notational convenience, we set $x_i:=z_i -\uz_i$, and write $d_{1} = \sum_{i=1}^4 x_i e_i \wedge$ and $d_{-1} := \sum_{i=1}^4 w_i \iota_{e_i}$ (then $W= \sum x_i w_i$). Also, for degree three generators, we set the notation as follows:
\begin{equation*}
\begin{array}{llll}
e_{top \setminus 1} :=   e_2 \wedge e_3 \wedge e_4 &  e_{top \setminus 2} :=  e_1 \wedge e_3 \wedge e_4 &
e_{top \setminus 3} :=  e_1 \wedge e_2 \wedge e_4 & e_{top \setminus 4} :=  e_1 \wedge e_2 \wedge e_3 
\end{array}
\end{equation*}

We do not know if $d_{-3}$ vanishes in this case, but we can show that there is 
a ``new exterior algebra structure'' on the pearl complex so that $d$ becomes a wedge-contraction type.
For this purpose, we need some calculations.

From $d^2 = (W-\lambda) \cdot Id$, we have the following two identities:
\begin{equation}\label{eq:dim4wc1}
d_{1} \circ d_{-3} + d_{-3} \circ d_{1} \equiv 0
\end{equation}
\begin{equation}\label{eq:dim4wc2}
d_{-1} \circ d_{-3} + d_{-3} \circ d_{-1} \equiv 0
\end{equation}

Write $d_{-3} (e_{top\setminus i}) := f_i$ which is a function on $x_i$'s. From \eqref{eq:dim4wc1} with $e_{top\setminus i}$, it follows that
$$ f_i \cdot \left( \sum x_j e_j \right) + d_{-3} ( (-1)^{i-1} x_i e_{top} ) = 0.$$
(Here, we used $d_{+1} (e_{top \setminus i} )= (-1)^{i-1} x_i e_{top}$.)
Since the second term is divisible by $x_i$, $f_i x_j e_j$ for $i \neq j$ in the first summand should be a multiple of $x_i$, also. This implies that $f_i = x_i g_i$ for some $g_i$. Then, one gets
$$d_{-3} (e_{top}) = (-1)^{i} g_i \left( \sum_j x_j e_j \right),$$
and hence $(-1)^i g_i$ does not depend on $i$, and we may set $g:= (-1)^i g_i$. In summary, we have
\begin{equation}\label{eq:dim4delta3}
 d_{-3} (e_{top}) = g \left( \sum_j x_j e_j \right) \quad d_{-3} (e_{top \setminus i} ) = (-1)^i x_i g \;\; \textrm{for}\; i=1,2,3,4.
\end{equation}
We will consider a new ``exterior algebra structure'' on $\wedge \Lambda^4$.
We define a new basis $e_{I}^{new}$ whose interior and exterior multiplication defined in a standard way: For a disjoint $I,J$, 
$$e_{I}^{new} \wedge_{new} e_{J}^{new} = \pm e_{I \cup J}^{new}.$$
where the sign is determined by identifying $e_I^{new}$ with $e_{i_1}^{new} \wedge_{new} \cdots \wedge_{new} e_{i_k}^{new}$ for $I=\{i_1 < \cdots <i_k\}$. (The wedge product is zero if $I,J$ are not disjoint.) The new interior multiplication $\iota_{i}^{new}:=\iota_{e_i^{new}}^{new}$ is
similarly defined.

We would like to define such a new exterior algebra structure so that
the differential $d$ becomes a wedge-contraction type as in the statement of the theorem.
In the case of $n=4$, we define a new {\em exterior algebra structure} by setting
$$e_{top}^{new} := e_{top} - g, \;\;\; e_{I}^{new} := e_{I} \;\; \textrm{for}\; I \neq \{1,2,3,4\}.$$
Then, from the previous setting,
$$d ( e_{top \setminus i}^{new}) = d ( e_{top \setminus i}) = (d_{1} + d_{-1} + d_{-3}) (e_{top \setminus i})
= d_{-1} (e_{top \setminus i}) + (d_{1} + d_{-3})(e_{top \setminus i}),$$
and one can easily check that
$$(d_{1} + d_{-3})(e_{top \setminus i}) = (-1)^{i-1} x_i  e_{top}^{new}.$$

Also, 
$$d (e_{top}^{new}) = (d_{1} + d_{-1} + d_{-3}) (e_{top} -g)
= d_{-1} (e_{top})  - d_{1} g + d_{-3} e_{top}= \sum (-1)^{i-1} w_i \, e_{top \setminus i} +0$$
since we have 
$$- d_{1} g + d_{-3} e_{top} = - \sum_i x_i ge_i+ \sum_i x_i g e_i =0.$$ 
This provides the desired wedge-contraction type structure obtained from the ``quantum correction of exterior algebra structure".
\end{proof}

\section{The real Lagrangian in the projective space} \label{sect:RP}
In this section, we compute the mirror matrix factorization of the real Lagrangian $\mathbb{R} \bP^n$ in $\mathbb{C}\bP^n$  for $n \geq 2$ when we take the reference to be the Clifford torus $\bL=L_u$. Here, $u$ is the center of the moment polytope of $\mathbb{C} \bP^n$, and $L_u$ is the fiber of the moment map at $u$.  We will give a detailed  description for $\mathbb{R} \bP^3$ at the end of the section.

Alston and Amorim \cite{AlAm} presented a comprehensive examination of Lagrangian Floer theory between the torus fiber and $\mathbb{R} \bP^n$. They first twisted Floer cohomology by a locally constant sheaf over a characteristic-2 ring instead of a line bundle, and took the Novikov ring
$$\Lambda^{\mathbb{F}_2}:= \left\{ \sum_{i=1}^\infty a_i T^{\lambda_i} \,\, | \,\, a_i \in \mathbb{F}_2, \, \lambda_i \in \R, \, \lambda_i \to \infty \right\}$$
over $\mathbb{F}_2$ as the coefficient ring, where $\mathbb{F}_2$ is the algebraic closure of $\Z_2$.  
Following their approach, we take our mirror variables to live in $\mathbb{F}_2^\times$.

From \cite{Oh1}, real Lagrangian $\mathbb{R} \bP^n$ has a minimal Maslov number $n+1$, and hence there exists no
Maslov index two disc with boundary on $\mathbb{R} \bP^n$ for $n\geq 2$. Thus, the real Lagrangian $\mathbb{R} \bP^n$ should
correspond to the matrix factorization of $W - 0$ for the Floer potential $W$ of $\bL$. However, one can check easily that
0 is not a critical value of $W$ which implies that the real Lagrangian $\mathbb{R} \bP^n$ corresponds to a trivial object  with the usual Novikov coefficients.
The same phenomenon happens in Floer theory also. Namely, Clifford torus $\bL$ and $\mathbb{R} \bP^n$ have different potential values,
and hence, its Floer cohomology cannot be defined.
But, if we use a characteristic-2 coefficient ring, 0 is a critical value for $n$ odd, hence $\mathbb{R} \bP^n$  can provide a non-trivial 
matrix factorization in the mirror.  We assume that $n$ is odd from now on.

Let us briefly review the construction of the Floer cohomology in \cite{AlAm} which involves locally constant sheaves (analogous to flat connections over the field $\mathbb{F}_2$). 
For each homomorphism $\rho : \pi_1 (\bL) \to \mathbb{F}_2^\times$, one can equip $\bL$ with a locally constant sheaf defined by
$$\mathcal{L}_\rho = \widetilde{\bL} \times \mathbb{F}_2 / (x \cdot \gamma, v) \sim (x, \rho (\gamma) v ) \qquad \gamma \in \pi_1 (\bL) $$
where $\widetilde{\bL}$ is the universal cover of $\bL$. This process is analogous to the construction of  $\C^\times$-flat line bundles from elements of $\Hom (\pi_1 (\bL), \C^\times)$. As fibers are discrete, one can define a parallel transport in $\mathcal{L}_\rho$ along a path in $\bL$. Then,

$$CF((\bL, \mathcal{L}_\rho),\mathbb{R} \bP^n) := \bigoplus_{\bL \cap \mathbb{R} \bP^n} \Hom (\mathcal{L}_\rho|_p, \mathbb{F}_2) \otimes \Lambda^{\mathbb{F}_2}$$
and the Floer differential is defined as in the case of flat complex line bundles (see Section \ref{sec:3}). 

Let us now choose a hyper-tori for the reference Lagrangian $\bL$. Recall that we fix the gauge of the (flat) connection for $\mathcal{L}_z$ in such a way that we put holonomy effect from $\rho$ whenever the (upper) boundary $\partial_0 u$ of a strip $u$ passes through chosen gauge hypertori. i.e. if $\partial_0 u$ traverses a hyper-torus $H_i$ positively, we have a variable $z_i$ which takes value in $\mathbb{F}_2^\times$. 
Here, the gauge hyper-tori are chosen in terms of  homogeneous coordinates of $\C \bP^n$ as follows:
if we write $\bL = \{[1 :e^{i \theta_1} :\cdots :e^{i \theta_n}] | 0 \leq \theta_j <2 \pi \}$ (a part of $\bL$ that lies inside the affine chart $\{x_0 \neq 0 \}$),
$$H_i := \{[1 :e^{i \theta_1} : \cdots : e^{i \theta_{i-1}} : e^{i \epsilon_i} : e^{i \theta_{i+1}} : \cdots : e^{i \theta_n}]  | 0 \leq \theta_j < 2 \pi, j \neq i \}$$
for a positive $\epsilon_i$ close to $0$.

Consequently, the potential
$$W = T^k \left(z_1 + z_2 + \cdots  + z_n + \frac{1}{z_1 z_2 \cdots z_n} \right)$$
is regarded as a function on $\left(\mathbb{F}_2^\times \right)^n$ 
(the potential itself does not depend on the choice of gauge hyper-tori.)
Note that $z_i = - z_i$ since now the coefficient ring is of characteristic $2$.

To compute the matrix factorization associated with $\mathbb{R} \bP^n$, we first need to find intersection points between the torus fiber and $\mathbb{R} \bP^n$ and then classify holomorphic strips of Maslov index one among these intersection points.
$\bL$ and $\mathbb{R} \bP^n$ intersect at $2^n$ points, which can be written in homogeneous coordinates as
$$[\pm 1 : \pm 1 : \cdots, \pm 1]$$
whose degrees depend on the number of $-1$'s in the entries (or equivalently the number of $1$'s since $n$ is odd).

For two different intersection points $p$ and $q$, there is a (index-1) holomorphic strip between $p$ and $q$ only when $n$ coordinate components of $p$ and $q$ agree (after over-all multiplication by $-1$ if necessary) i.e.,  
$$p=[ a_0  : \cdots, a_{i-1}  : -1 :a_{i+1}  : \cdots : a_{n} ]$$
$$q=[ a_0  : \cdots, a_{i-1}  : +1:a_{i+1}  : \cdots : a_{n} ]$$
where $a_0, \cdots, a_{i-1}, a_{i+1}, \cdots, a_n$ are $\pm 1$, and the even number of them are $-1$. According to \cite{AlAm}, there are two holomorphic strips between $p$ and $q$ (one from $p$ to $q$ and the other from $q$ to $p$) both of which are halves of the holomorphic disc corresponding to the $z_i$-term  in the potential for $1 \leq i \leq n$ and the $\frac{1}{z_1 \cdots z_n}$-term for $i=0$. One can figure out the input and the output of the holomorphic strip from the orientation of the boundary of the discs (with help of \cite[Proposition 4.1]{AlAm}). 
%

The contributions of these strips to the entries of the mirror matrix factorization for $\R \bP^n$ are given as follows:

\vspace{0.2cm}

\noindent(i) $1 \leq i \leq n$ : The upper boundary of the holomorphic strip from $p$ to $q$ intersects the $i$-th hyper torus in positive direction, and hence gives the term $T^{k/2} z_i$. The other half of the $z_i$-disc runs from $q$ to $p$ not intersecting any hyper tori, so it produces $T^{k/2}$-term.

\noindent(ii) $ i=0$ : The  strip from $p$ to $q$ is the half of the $\frac{1}{z_1 \cdots z_n}$-disc and its upper boundary passes negatively through the $j$-th hyper torus for each $a_j =-1$. Thus, it gives the term 
$\frac{T^{k/2}}{\prod_{j \leq n} z^{\delta (a_j, -1)}}.$
where $\delta (a, b) := 1$ if $a = b$ and zero otherwise as in Section \ref{subsec:MFtoricfano1}.

On the other hand, the upper boundary of the strip from $q$ to $p$ intersects the $l$-th hyper torus if $a_l = 1$, and hence induces
$\frac{T^{k/2}}{\prod_{j \leq n} z^{\delta (a_j, 1)}}.$
\vspace{0.2cm}

The case for $i=0$ looks distinguished from the other cases due to our specific choice of basis of $\pi_1(\bL)$ and gauge hypertori: recall that the chosen basis is $\{\partial \beta_1,\ldots,\partial \beta_{n}\}$ for the basic disc classes $\beta_1,\ldots,\beta_n$, while $\partial \beta_{0} = - \sum_{i=1}^n \partial \beta_i$.

In conclusion, we have the following theorem:

\begin{theorem} \label{thm:RP}
The mirror matrix factorization of the real Lagrangian $\R \bP^n$ in $\C \bP^n$ for odd $n$ is formally generated by $[\pm 1 :  \cdots : \pm 1]$ over $\Lambda^{\mathbb{F}_2}[z_1, \cdots, z_n]$, and equipped with a module map whose matrix coefficients $m_{qp}$ and $m_{pq}$ for two generators $p$ and $q$ are given as follows:
\begin{enumerate}
\item for $p=[a_0 : \cdots : a_{i-1} : -1 : a_{i+1} : \cdots : a_n]$ and $q=[a_0 : \cdots : a_{i-1} : 1 : a_{i+1} : \cdots : a_n]$ with even number of $-1$'s in $a_j$ for  $j\neq i$,
\begin{itemize}
 \item $m_{qp} = \dfrac{T^{k/2}}{\prod_{1 \leq j \leq n} z^{\delta (a_j, -1)}}, \quad m_{pq} = \dfrac{T^{k/2}}{\prod_{1 \leq j \leq n} z^{\delta (a_j, 1)}}, \quad$ if $i=0$;
\item  $ m_{qp} =T^{k/2} z_i, \quad m_{pq} =T^{k/2} \quad$ if $i \neq 0$;
\end{itemize}
\item $m_{qp}$ and $m_{pq}$ are zero in other cases.
\end{enumerate}
\end{theorem}

%

We provide the mirror matrix factorization of $\R \bP^3 (\subset \C \bP^3)$ in an explicit matrix form. 
 We arrange $8$-intersection points between $\bL$ and $R$ as 
\begin{equation*}
\begin{array}{ll}
p_1 = [1 :-1 : -1 : -1], & p_2 = [1:-1 : 1 : 1 ],\\
p_3 = [1: 1 :-1 : 1], & p_4 = [1: 1 :1 : -1 ],
\end{array}
\end{equation*}
and
\begin{equation*}
\begin{array}{ll}
q_1 = [1 :1 : 1 : 1], & q_2 = [1 : 1 : -1 : -1],\\
q_3 = [1: -1 :1 : -1 ], & q_4 = [1: -1 :-1 : 1 ].
\end{array}
\end{equation*}
Restricting the previous computation to dimension $3$, 
the matrix factorization mirror to $R$ is as follows:
\begin{equation}\label{eq:MFRP3}
\bordermatrix{& p_1 & p_2 & p_3 & p_4  & q_1 & q_2 & q_3 & q_4 \cr
          p_1 &    0  & 0   & 0 & 0 & \frac{T^{k/2}}{z_1 z_2 z_3}  & T^{k/2} & T^{k/2} & T^{k/2} \cr
                   p_2 &    0  &  0 & 0  & 0 & T^{k/2} & -\frac{T^{k/2}}{z_1} & -T^{k/2} z_3 & T^{k/2} z_2 \cr
                    p_3 &    0  & 0  & 0  & 0 & T^{k/2} & T^{k/2} z_3 & -\frac{T^{k/2}}{z_2} &  -T^{k/2} z_1 \cr
                    p_4  &     0 &  0 &  0 & 0 & T^{k/2} & -T^{k/2} z_2 & T^{k/2} z_1 & -\frac{T^{k/2}}{z_3} \cr            
 q_1 &   T^{k/2}  &   T^{k/2} z_1  &  T^{k/2} z_2  & T^{k/2}  z_3  & 0 & 0 & 0 & 0 \cr
 q_2 & T^{k/2}  z_1  &   -\frac{T^{k/2}}{z_2 z_3}   & T^{k/2} & -T^{k/2}  &0  & 0 &0  &0  \cr
 q_3 & T^{k/2} z_2  &  -T^{k/2}  & -\frac{T^{k/2}}{z_3 z_1} & T^{k/2}  & 0 &  0 & 0 & 0 \cr
 q_4 &  T^{k/2} z_3  &  T^{k/2}   &  -T^{k/2} & -\frac{T^{k/2}}{z_1 z_2}  & 0 & 0 & 0 & 0 \cr                }
 \end{equation}
Although $1=-1$ in $\mathbb{F}_2$, we put signs so that \eqref{eq:MFRP3} also defines a matrix factorization over a characteristic zero field.

\appendix
\section{Orientations}\label{app:signrule}

\subsection{Orientation conventions}
Biran and Cornea has shown that when a Lagrangian submanifold $L$ is spin, the pearl complex
 can be defined over $\Z$-coefficient in \cite[Appendix A]{BC2}. We follow their orientation convention for the computations of matrix factorizations in this paper.
  
 First, we recall some of elementary orientation conventions. 
 An intersection $A \cap B$ is oriented as follows (see \cite{BC2}, \cite{FOOO}).
For each $x \in A \cap B$, we choose the orientation of the normal bundle $N_AL$ so that
we have $o(T_xL) = o(N_{x,A}L) \wedge o( T_x A)$, where we denote the orientation of a vector space $V$ as $o(V)$.
If $A$, $B$ intersect transversely, we define the orientation of $A\cap B$ by
$$o(T_xL) = o(N_{x,A}L) \wedge o(N_{x,B}L) \wedge o(T_x(A\cap B)).$$
For $f:A \to L, g:B \to L$, the  orientation convention of the fiber product $A \times_L B$ 
 in \cite{BC2}, \cite{FOOO} are the same.
If $f$ is submersion, and $g$ is an embedding, then choose orientation of $\textrm{Ker} f$ so that 
$o(A) = o(\textrm{Ker} f) \wedge o( f^{-1}([L]))$. Then we set the orientation of the fiber product to be
$$o(A \times_L B) = o(\textrm{Ker} f) \wedge o(B),$$ 
as in Section 7 \cite{C0}.
When both $f$ and $g$ are embedding, the fiber product becomes the intersection $B \cap A$ as
oriented spaces. The key point of this convention is the following identity as oriented spaces, which can be checked easily:
$$ \partial (A \times_L B)= \partial A \times_L B \sqcup (-1)^{n +\dim A} A \times_L \partial B$$

\subsection{Signs for pearl trajectories}\label{ss:sign}
We first fix the orientation of stable manifold $W^{s}(p)$ for each critical point $p$, and
orient the unstable manifold so that we have $T_pL = T_pW^{u}(p) \oplus T_pW^{s}(p)$ as oriented spaces.
In particular, the intersection  $W^{s}(p) \cap W^{u}(p)$ gives a positive intersection number.
Hence, we define the orientation of the Morse trajectory  from $p$ to $q$ as $W^{s}(q) \cap W^{u}(p)$.
As usual, the signed count of such trajectory is obtained by comparing this orientation with that of the flow orientation.

Now, consider the moduli space $\CM_2(\beta)$ of $J$-holomorphic discs of class $\beta$ (of Maslov index two)
with two marked points, with an evaluation map 
$$ev=(ev_0,ev_1): \CM_2(\beta) \to L \times L.$$
For the inclusion $i : W^{u}(p) \times W^{s}(q) \to L \times L$,
Biran-Cornea \cite{BC} defined the moduli space of the pearl complexes with a single pearl from $p$ to $q$ as the fiber product
$$W^{u}(p) \times_{L, ev_0} \big( \CM_2(\beta)_{ev_1} \times_L  W^{s}(q)\big).$$

With this sign rule in mind, we find the precise sign in Equation\eqref{eq:d2wpearl}.
\begin{lemma}\label{lem:BCMFsign}
We have $(\delta^{\mathcal{L}_0, \mathcal{L}_1}_{\textnormal{pearl}} )^2  =  \big( \Phi(\CL_1)  - \Phi(\CL_0) \big)$.
\end{lemma}

\begin{proof}
We want to show that the sign of the above identity is correct, and the rest of the proof is standard and given in the main body.
We write $\delta^{\mathcal{L}_0, \mathcal{L}_1}_{\textnormal{pearl}}$ by $\delta_{\textnormal{pearl}}$ for simplicity.

Let us consider the moduli space of single pearls from $p$ to $p$ itself, which is given as the fiber product
\begin{equation}\label{eq:pearlptop}
W^u (p) \times \mathcal{M}_2 (\beta) \times W^s (p)
\end{equation}
where $\mu (\beta)=2$ and the product is taken over $L$. Taking the boundary of \eqref{eq:pearlptop} gives
\begin{equation}\label{eq:threepearlbdy}
\begin{array}{l}
\quad (\partial W^u (p) ) \times \mathcal{M}_2 (\beta) \times W^s (p)\\
\cup \,\, (-1)^{n+{\rm ind}(p)} W^u (p) \times ( \partial \mathcal{M}_2 (\beta)) \times W^s(p) \\
\cup \,\,(-1)^{(n + {\rm ind}(p)) + (n + (n-1) )} W^u (p) \times \mathcal{M}_2 (\beta) \times (\partial W^s (p)).
\end{array}
\end{equation}
By the orientation convention in \cite{BC2}, parts of boundaries of unstable and stable manifolds are oriented as follows:
$$ \partial W^u (p) = m (p,q) \times W^u (q), \qquad  \partial W^s (p) = (-1)^{n + {\rm ind} (p)} m(r,p) \times W^s (r) $$
where $m(p,q)$ denotes the moduli space of gradient flow lines from $p$ to $q$. Thus, the first and the second components of \eqref{eq:threepearlbdy} can be rewritten as
$$ m(p,q) \times \left( W^u (q) \times \mathcal{M}_2 (\beta) \times W^s (p)  \right) \cup (-1) \left( W^u (p) \times \mathcal{M}_2 (\beta) \times W^s (r) \right) \times m(r,p) $$
which corresponds to $ (-1)^{{\rm ind} (p)} (\delta_{\textnormal{pearl}})^2$. 
%
On the other hand, the boundary of the middle factor $\mathcal{M}_2 (\beta)$ in \eqref{eq:pearlptop} is nontrivial due to disc bubbles, and we have two more terms $\Phi (\mathcal{L}_0)$ and $\Phi (\mathcal{L}_1)$ in addition to $(-1)^{{\rm ind} (p)} (\delta_{\textnormal{pearl}})^2$. We only check the sign for $\Phi (\mathcal{L}_0)$, and the sign of $\Phi (\mathcal{L}_1)$ can be similarly proven to be opposite to that of $\Phi (\mathcal{L}_0)$. $\Phi (\mathcal{L}_0)$ comes from a  disc bubble attached along the upper boundary of the disc component of the original single pearl, which corresponds $\mathcal{M}_3 (\beta_0) _{ev_1} \times_{ev_0} \mathcal{M}_1 (\beta)$ where $\beta_0$ represents the constant class.

Recall that there is a subtle difference between the orientation conventions in \cite{FOOO} and \cite{BC2}.

\noindent(i) Both of them fix two markings on  the boundary of discs and consider the action of $1$-dimensional automorphism group which preserves these two markings. See for example, \cite[(8.3.2)]{FOOO}. However, they used the opposite orientations for this group  so that the moduli space of discs has opposite orientations for \cite{FOOO} and \cite{BC2}. 

\noindent(ii) Moreover, the role of two markings $z_0, z_1$ used to attach two discs are opposite in \cite{FOOO} and \cite{BC2}, which is equivalent to the switch of positions of two factors in the fiber product $\mathcal{M}_{k_1} (\beta_0) \times_L \mathcal{M}_{k_2} (\beta_1).$

Let us consider the inclusion of the boundary stratum:
$$\mathcal{M}_{k_1} (\beta_0) \times_L \mathcal{M}_{k_2} (\beta) \hookrightarrow \mathcal{M}_2 (\beta).$$
As explained in \cite[Remark A.1.1]{BC2},  if $k_1 =k_2=2$, then  the sign of this inclusion is $(-1)^{n+1}$ for both \cite{FOOO} (see \cite[Proposition 8.3.3]{FOOO}) and \cite{BC2}. In this case, sign differences from (i) and (ii) are both $(-1)$ and 
and cancel each other.

We claim that the sign of this  inclusion is $(-1)^{n}$ when $k_1=3, k_2=1$. First, the inclusion from \cite{FOOO}
has sign $(-1)^{n+1}$, with an additional sign $(-1)$ from (i). Now, in this case, it is not hard to compute the
effect of  switching of two factors (from (ii)) and find that there is no additional sign contribution from this.
This will be needed  for the sign of the second term in \eqref{eq:threepearlbdy}.

Here, we orient $L$ so that we have  $\mathcal{M}_3 (\beta_0) \cong L$.
Consequently, the sign of $\Phi (\mathcal{L}_0)$ in the equation (induced by) \eqref{eq:threepearlbdy} is $(-1)^{ (n + {\rm ind} (p) )} \cdot (-1)^{n} = (-1)^{\rm{ind} (p)  }$, according to the sign rule of \cite{BC2}. Therefore, we have
$$ (-1)^{\rm{ind} (p)} (\delta_{\textnormal{pearl}})^2 + (-1)^{\rm{ind} (p)} ( \Phi (\mathcal{L}_0)) - \Phi (\mathcal{L}_1)) =0,$$
or equivalently, $(\delta_{\textnormal{pearl}})^2 = \Phi (\mathcal{L}_1) - \Phi (\mathcal{L}_0)$.
\end{proof}

\subsection{Sign computations for Lemma \ref{lem:pearld1}}\label{subsec:signpearld1}
We first make the following the sign convention for the Morse-differential $(\delta_{\textnormal{pearl}})_{1}=\delta_{\rm Morse}$.
We first choose the orientation of stable manifolds of $e_I$ as $E_I$ where $E_I = E_{i_1} \wedge \ldots E_{i_{|I|}}$ for $I=\{i_1,\ldots,i_{|I|}\}$ with $i_1<\ldots<i_{|I|}$. Hence, unstable manifolds are oriented as
$$o(T_{e_I} \bL) =  o(T_{e_I} W^u ({e_I})) \wedge E_I$$
For $j \notin I$, suppose $E_j \wedge E_I = s E_{\{j\} \cup I}$ for $s \in \{ \pm 1\}$.
Then,  we have 
$$o(T_{e_{ \{j\} \cup I}}W^{u} ({e_I})) = s \cdot o\big( T_{e_{ \{j\} \cup I}}W^{u} (e_{\{j\} \cup I}) \big) \wedge E_j.$$
Now consider the moduli space $\CM(q'\to q)$ of negative gradient flow lines from $q'$ to $q$,
which is oriented as an intersection $W^{s} (q) \cap W^{u} (q')$ (see subsection \ref{ss:sign} for sign convention).  For $\gamma \in \CM(q'\to q)$, let $T_\gamma$ be the direction of the flow of $\gamma$. Then, we have
$$ o(T_{q}W^u (q)) \wedge  o(T_{q'} W^s (q')) \wedge T_\gamma 
\wedge o(\CM(q'\to q)) = o(T_{q'} \bL),$$ where we assume  $\CM(q'\to q)$ to be oriented 0-dimensional vector spaces.
In particular,  $\gamma \in \CM(q'\to q)$ has a positive sign if the splitting 
$ (-1)^{\deg q'}  o(T_{q}W^u (q)) \wedge T_\gamma$  equals that of $o(T_{q'}W^{u}(q'))$
and has a negative sign otherwise.

We now analyze the sign of a flow line $\gamma$ from $e_I$ to $e_{I\cup \{j\}}$.  The one which passes through $H_j + p$ has tangent vector $T_\gamma = E_j$, and the one which passes through $H_j + \up$ has tangent vector $T_\gamma = -E_j$.  First consider the case $j < \min I$ (when $I = \emptyset$, $\min I := +\infty$).  By our choice of orientations for unstable submanifolds, the splitting $TW^u (e_I)|_\gamma \cong T_{e_{\{j\} \cup I}} W^u ({e_{\{j\} \cup I})} \oplus T_\gamma$ preserves orientation if $T_\gamma = E_j$ and reverses orientation if $T_\gamma = -E_j$.  Hence excluding the factor $(-1)^{|I|}$, the flow line has a positive sign if it is the one passing through $H_j + p$, and has a negative sign if it is the one passing through $H_j + \up$.  As a result when $j < \min I$, the coefficient of $e_{\{j\} \cup I} = e_j \wedge e_I$ in $(\delta_{\textnormal{pearl}})_{1} \cdot e_I$ is $(-1)^{|I|}(z_i - \uz_i) = (-1)^{{\rm deg} (e_I)} (z_i - \uz_i)$.

For the general case $j=1,\ldots,n$, we claim that the flow line which passes through $H_j + p$ has the sign 
$s$ with $E_j \wedge E_I = s E_{\{j\} \cup I}$, and the one which passes through $H_j + \up$ has the sign $-s$.  This finishes the proof of this lemma.  To see this, we compare the orientation of $TW^u  (e_I)|_\gamma$ and that of $T_{e_{\{j\} \cup I}} W^u(e_{\{j\} \cup I} ) \oplus \langle E_j \rangle$.  By definition the orientation forms have the relation
$$ o(T_{e_I} \bL|_\gamma) = o(TW^u (e_I)|_\gamma) \wedge E_I = o(T_{e_{\{j\} \cup I}} W^u (e_{\{j\} \cup I})) \wedge E_{\{j\} \cup I} = o(T_{e_{\{j\} \cup I}} W^u (e_{\{j\} \cup I})) \wedge (s \cdot E_j \wedge E_I) $$
and hence $o(TW^u (e_I)|_\gamma) = s\cdot o(T_{e_{\{j\} \cup I}} W^u (e_{\{j\} \cup I})) \wedge  E_j$.  By definition, the flow line $\gamma$ has the sign $(-1)^{|I|}s = (-1)^{{\rm deg} (e_I)} s$.

\subsection{Sign computations for Lemma \ref{lem:pearld2}}\label{subsec:signpearld2}
We next derive the sign of a single $\beta_i$ pearl trajectory from $e_I$ to $e_{I-\{j\}}$ appearing in the proof of Lemma \ref{lem:pearld2}. The moduli space of such pearl trajectories is oriented as
$W^u (e_{I}) \times_{\bL,ev_1} (\CM_2(\beta_i) \times_{ev_0,\bL} W^s (e_{I- \{j\}}))$ (see Section \ref{ss:sign}).
In the toric cases, we can equip the Lagrangian torus $\bL$ with the standard spin structure and if $\beta_i$ is a basic disc class,
then we have that $ev_0:\CM_1(\beta_i) \to \bL$ is an orientation preserving homeomorphism.
(This follows from Proposition 8.1 \cite{C0} since we take the opposite orientation of both $\CM_1(\beta_i)$ and $\bL$ 
compared to that of \cite{C0}.)
Hence, $o(\CM_2(\beta_i)_{ev_0})= (-1)^{n+1} o(\partial D_0^2 \times \CM_1(\beta_i) )$ and we have
$$ o(\CM_2(\beta_i)_{ev_0}\times_L W^{s}(e_{I-\{j\}}))=(-1)^{n+1} o( \partial \beta_i \times W^{s}(e_{I-\{j\}}) ),$$
and hence
$$W^u (e_{I}) \times_{\bL} (\CM_2(\beta_i) \times_L W^s (e_{I- \{j\}}))= \left( (-1)^{n+1} \partial \beta \times W^s (e_{I-\{ j\}}) \right) \cap W^u (e_I) .$$

Since $W^s (e_I) \cap W^u (e_I) = +1$, it suffices to compare the orientations on $(-1)^{n+1} \partial \beta_i \times W^s (e_{I-\{ j\}}) $ and $W^s (e_I)$. Note that the former (after being projected onto $E_I$-plane) is equivalent to $(-1)^{n+1} s_{i,j} \times E_{I-\{ j\}}$ and the latter is simply $E_I$ itself. Thus, the total sign difference between $(-1)^{n+1} \partial \beta_i \times W^s (e_{I-\{ j\}})$ and $W^s (e_{I})$ is $ (-1)^{n+1} s_{i,j} s^*$
for $s^*\!=\,$the sign difference of $E_{I}$ and $(E_j \wedge E_{I-\{j\}})$. Consequently, 
$$\langle (\delta_{\textnormal{pearl}})_{-1} (e_I), e_{I-\{j\}} \rangle = (-1)^{{\rm ind} (e_{I - \{j\}} )} (-1)^{n+1} s_{i,j} s^\ast Z_{\gamma_j} \cdot Z_{\partial D_i} = (-1)^{{\rm deg} (e_I)} s^\ast \left( s_{i,j}  Z_{\gamma_j} \cdot Z_{\partial D_i} \right)$$
where $(-1)^{{\rm ind} (e_{I - \{j\}} )}=(-1)^{n-|I| + 1}$ comes from the sign factor in the second term of \eqref{eq:pft1}. 

Note that $\iota_{e_j} e_I = s^* e_{I-\{j\}}$, and hence
we have shown that the sign factor of $(-1)^{\deg} \delta_1$ equals $s_{i,j}$ as claimed.  

\end{document}